\documentclass[oneside,english]{amsart}
\usepackage[T1]{fontenc}
\usepackage[latin9]{inputenc}
\usepackage{babel}
\usepackage{textcomp}
\usepackage{amsbsy}
\usepackage{amstext}
\usepackage{amsthm}
\usepackage{amssymb}
\usepackage[unicode=true,pdfusetitle,
 bookmarks=true,bookmarksnumbered=false,bookmarksopen=false,
 breaklinks=false,pdfborder={0 0 1},backref=false,colorlinks=false]
 {hyperref}

\makeatletter
\numberwithin{equation}{section}
\numberwithin{figure}{section}
\theoremstyle{plain}
\newtheorem*{thm*}{\protect\theoremname}
\theoremstyle{plain}
\newtheorem{thm}{\protect\theoremname}
\theoremstyle{definition}
\newtheorem{defn}[thm]{\protect\definitionname}
\theoremstyle{definition}
\newtheorem{example}[thm]{\protect\examplename}
\theoremstyle{remark}
\newtheorem{rem}[thm]{\protect\remarkname}
\theoremstyle{plain}
\newtheorem{lem}[thm]{\protect\lemmaname}
\theoremstyle{plain}
\newtheorem{prop}[thm]{\protect\propositionname}
\theoremstyle{plain}
\newtheorem{cor}[thm]{\protect\corollaryname}

\numberwithin{thm}{section}
\usepackage[margin=3.3cm]{geometry}

\DeclareMathOperator{\Ric}{\mathrm{Ric}}
\DeclareMathOperator{\W}{\mathrm{W}}
\DeclareMathOperator{\R}{\mathrm{R}}
\DeclareMathOperator{\s}{\mathrm{s}}
\DeclareMathOperator{\PP}{\mathrm{P}}
\DeclareMathOperator{\G}{\mathrm{G}}

\DeclareMathOperator{\cl}{\mathbf{c}}
\DeclareMathOperator{\dVol}{\mathrm{dVol}}

\DeclareMathOperator{\II}{\mathrm{I\!I}}
\DeclareMathOperator{\shape}{\mathrm{S}}
\DeclareMathOperator{\HH}{\mathrm{H}}

\DeclareMathOperator{\CS}{\mathrm{CS}}
\DeclareMathOperator{\sk}{\mathrm{skew}}
\DeclareMathOperator{\symm}{\mathrm{symm}}
\DeclareMathOperator{\id}{\mathrm{id}}
\DeclareMathOperator{\curl}{\mathrm{curl}}

\DeclareMathOperator{\symmtf}{\mathring{\mathrm{symm}}}
\DeclareMathOperator{\tf}{\mathrm{tf}}

\DeclareMathOperator{\Spin}{\mathrm{Spin}}
\DeclareMathOperator{\SO}{\mathrm{SO}}
\DeclareMathOperator{\SL}{\mathrm{SL}}
\DeclareMathOperator{\SU}{\mathrm{SU}}
\DeclareMathOperator{\GL}{\mathrm{GL}}

\DeclareMathOperator{\tr}{\mathrm{tr}}

\DeclareMathOperator{\phg}{\mathrm{phg}}
\DeclareMathOperator{\Fr}{\mathrm{Fr}}

\DeclareMathOperator{\res}{\mathrm{res}}
\DeclareMathOperator{\hyp}{\mathrm{hyp}}

\makeatother

\providecommand{\corollaryname}{Corollary}
\providecommand{\definitionname}{Definition}
\providecommand{\examplename}{Example}
\providecommand{\lemmaname}{Lemma}
\providecommand{\propositionname}{Proposition}
\providecommand{\remarkname}{Remark}
\providecommand{\theoremname}{Theorem}

\begin{document}
\title{Nahm poles and 0-instantons}
\author{Marco Usula}
\begin{abstract}
We study self-dual $0$-connections, or $0$-instantons, on asymptotically
hyperbolic 4-manifolds. These connections develop a uniform singularity
along the conformal infinity, and are asymptotic, at each point of
the boundary, to a ``Nahm pole'' model solution on $\HH^{4}$.
Examples include the Levi-Civita spin connections on $\mathbb{S}_{+}$
over spin Poincaré--Einstein 4-manifolds. Inspired by the Fefferman--Graham
expansion for Poincaré--Einstein metrics, we study the asymptotic
expansion of these $0$-instantons. We prove that the expansion is
log-smooth, and that the coefficient of the first log term --- which
we call the $0$-instanton obstruction tensor --- is a conformal
invariant related to the Weyl curvature of the ambient conformal metric.
We then show that this invariant vanishes if and only if the $0$-instanton
is smooth modulo gauge. Finally, we study the renormalized Yang--Mills
energy: we prove that, if the metric is asymptotically Poincaré--Einstein
to third order, then this energy is a well-defined conformal invariant,
and equals the negative Chern-Simons invariant of the conformal infinity.
\end{abstract}

\maketitle
\tableofcontents{}

\sloppy

\section{Introduction}

This paper is situated at the intersection of three topics in mathematical
physics and geometric analysis: the study of \emph{conformally compact
geometry}, and particularly \emph{Poincaré--Einstein metrics}; the
study of 4-manifolds via \emph{gauge theory}, and particularly the
\emph{self-duality equation}; and Witten's recent approach to knot
invariants, through the study of gauge theoretic\emph{ singular boundary
value problems}.

Conformally compact geometry has been an intense topic of research,
since the pioneering work of Fefferman--Graham in the 1980s \cite{FeffermanGrahamConformal, FeffermanGrahamAmbient}.
A metric $g$ on the interior $X^{\circ}$ of a compact manifold with
boundary $X$ is conformally compact if its conformal rescaling $x^{2}g$
by the square of a boundary defining function\footnote{A non-negative smooth function on $X$ which vanishes transversely
along $\partial X$.} $x$ extends to a metric on $X$. These metrics generalize the metric
on hyperbolic space, and share many properties with it; in particular,
they induce a conformal class on the boundary $\partial X$, called
its \emph{conformal infinity}. Fefferman and Graham discovered that,
given a closed conformal manifold $\left(Y^{n},\mathfrak{c}\right)$,
one can construct a corresponding \emph{formal}\footnote{The word \textquotedblleft formal\textquotedblright{} means that the
metric is constructed only as a formal expansion near its boundary.} conformally compact and Einstein metric $\left(X,g\right)$ on the
cylinder $X=Y\times[0,+\infty)$. This construction provides a fundamental
tool to construct \emph{conformal invariants} of $\left(Y,\mathfrak{c}\right)$,
in terms of the formal expansion of the Riemannian invariants of $\left(X,g\right)$.

The mathematical study of self-dual instantons goes back to work of
Atiyah--Hitchin--Singer \cite{AtiyahHitchinSingerSelfDuality} and
Atiyah--Hitchin--Drinfeld--Manin \cite{ADHM} and, thanks
to pioneering work by Donaldson, Taubes, Uhlenbeck and others, has
remained an active topic ever since. A \emph{self-dual instanton}
on a closed, oriented, conformal 4-manifold is a connection $A$ satisfying
the equation $F_{A}=\star F_{A}$. This equation is gauge-invariant
and elliptic modulo gauge, and this implies that the moduli space
of solutions of this PDE is essentially a finite-dimensional object,
typically a manifold with isolated singularities. Donaldson
showed that, in favorable situations, this moduli space is naturally
\emph{compactifiable}, and this allows one to extract enumerative
invariants out of it. When the dimension is $0$, these enumerative
invariants are simply appropriate signed counts of the solutions.
Remarkably, these invariants detect very subtle properties of the
\emph{smooth structure} of the underlying 4-manifold, and have been
central in showing the existence of infinitely many exotic smooth
structures on some closed 4-manifolds \cite{donaldson1983application, donaldson1990polynomial, taubes1996seiberg, taubes2000seiberg, UhlenbeckRemovableSingularities}.
This effectively opened the \textquotedblleft Pandora's box\textquotedblright{}
of 4-dimensional differential topology, a topic which is currently
being studied intensely, but in which many questions remain out of
reach.

Closely related to the previous ideas is a recent approach proposed
by Witten, to find \emph{knot invariants} via gauge theory in four
dimensions. More precisely, Witten discussed in \cite{WittenFivebranes}
the idea of realizing the coefficients of the Jones polynomial of
a knot in $\mathbb{R}^{3}$ as \textquotedblleft counts\textquotedblright{}
of solutions to another gauge-theoretic PDE, the \emph{Kapustin--Witten
equation}, on the upper half-space $\mathbb{R}_{+}^{4}$. A distinguishing
feature of his approach is that the complex connection involved in
these equations is \emph{singular} along both the boundary and the
knot. The asymptotic boundary condition around these singularities
is known as the \emph{Nahm pole boundary condition}. A mathematically-oriented
analysis of this boundary problem is carried out in \cite{MazzeoWittenNahmI, MazzeoWittenNahmII}.
These ideas have inspired many interesting papers, among which we
mention for example \cite{he2020extended, he2019extended, Dimakis1, Dimakis2, Dimakis3}.

In this paper, we study the self-duality equations for $\SU\left(2\right)$
\emph{$0$-connections} on a bundle $E\to X$ over an asymptotically
hyperbolic 4-manifold $\left(X,g\right)$. These connections allow
taking directional derivatives along \emph{$0$-vector fields}, namely
those vector fields on $X$ that vanish along $\partial X$. These
connections arise naturally in conformally compact geometry. For example,
the Levi-Civita connection of a conformally compact metric always
extends to a $0$-connection.

This class of connections appeared in \cite{FineHerfrayKrasnovScarinciAH}
and implicitly in \cite{MazzeoWittenNahmI}, and was formalized by
the author in \cite{UsulaYangMills}. In that previous work, the author
solved a boundary value problem for \emph{Yang--Mills} connections
on conformally compact manifolds. However, in contrast to that setting,
the connections studied in the present paper develop a \emph{uniform
singularity} along the boundary. More precisely, near each point at
infinity, the metric is asymptotic to the hyperbolic half-space metric
$x^{-2}\left(dx^{2}+dy^{2}\right)$ on $\HH^{4}$, and the $0$-connections
are required to be asymptotic to the \emph{Nahm pole solution }on
$\HH^{4}$
\[
\nabla=d+\sum_{i=1}^{3}\frac{dy_{i}}{x}\otimes\mathfrak{s}_{i}
\]
where $\mathfrak{s}_{1},\mathfrak{s}_{2},\mathfrak{s}_{3}$ is a basis
of $\mathfrak{su}\left(2\right)$ satisfying $\left[\mathfrak{s}_{i},\mathfrak{s}_{j}\right]=\varepsilon_{ij}{^{k}}\mathfrak{s}_{k}$.
Invariantly, this solution represents the Levi-Civita spin $0$-connection
on the spinor bundle $\mathbb{S}_{+}\to\HH^{4}$. This boundary condition
appears naturally in the study of $\SU\left(2\right)$ $0$-instantons:
we prove in Proposition \ref{prop:nahm-residues} that, under a mild
non-degeneracy condition, an $\SU\left(2\right)$ $0$-connection
satisfies the Nahm pole boundary condition if and only if it is asymptotically
self-dual.

\subsection*{Main results of the paper}

Fix an asymptotically hyperbolic 4-manifold $\left(X^{4},g\right)$,
an $\SU\left(2\right)$ vector bundle $E$, and call $\mathfrak{c}_{\infty}\left(g\right)$
the conformal infinity of $g$. As proved in \cite{GrahamLeeEinstein},
for every choice of $h_{0}\in\mathfrak{c}_{\infty}\left(g\right)$
we can choose a boundary defining function $x$ such that $\left(x^{2}g\right)_{|\partial X}=h_{0}$,
and, calling $\overline{g}=x^{2}g$, we have $\left|dx\right|_{\overline{g}}^{2}\equiv1$
in a neighborhood of $\partial X$. Such an $x$ is called a \emph{geodesic
boundary defining function }inducing $h_{0}$. In a sufficiently small
collar of $\partial X$ induced by $\mathrm{grad}_{\overline{g}}x$,
the metric $g$ takes the normal form
\[
g=\frac{dx^{2}+h\left(x\right)}{x^{2}},
\]
where $h:[0,\varepsilon)\to S^{2}\left(T^{*}\partial X\right)$ is
a family of metrics on $\partial X$ starting at $h_{0}$.

If a $0$-connection $A$ satisfies the Nahm pole boundary condition,
then it can be put in a similar normal form, dependent on the choice
of $h_{0}\in\mathfrak{c}_{\infty}\left(g\right)$. More precisely,
we show in §\ref{subsec:Nahm-poles-and-spin-structures-on-the-boundary}
and §\ref{subsec:A-normal-form-for-0connections-adapted-to-g-at-infinity}
that the Nahm pole and the choice of $h_{0}$ determine canonically
a spin structure $\left(\mathbb{S},\cl_{h_{0}}\right)$ on $\left(\partial X,h_{0}\right)$,
and we can then associate to $A$ a family $\alpha_{h_{0}}\left(x\right)$
of \emph{$\SU\left(2\right)$ }connections on $\mathbb{S}$ defined
on $\left(0,\varepsilon\right)$; we call it the \emph{geodesic normal
family }of $A$ relative to $h_{0}$. This family of connections is
an invariant of the gauge class $\left[A\right]$, and it has a polyhomogeneous
expansion as $x\to0$ with a leading $O\left(x^{-1}\right)$ term.
Our first main result is a direct analogue of Theorem 4.8 of \cite{FeffermanGrahamAmbient}
for Poincaré--Einstein metrics; it shows that, if $A$ is self-dual,
then this expansion takes a special form:
\begin{thm*}
(Theorems \ref{thm:index-set-0-instantons}, \ref{thm:coefficients-of-the-instanton-expansion})
If $A$ is a polyhomogeneous self-dual $0$-instanton, then the expansion
of the geodesic normal family $\alpha_{h_{0}}\left(x\right)$ is log-smooth;
more precisely, it takes the form
\begin{align*}
\alpha_{h_{0}}\left(x\right) & \sim\alpha_{-1}x^{-1}+\alpha_{0}+\alpha_{1,1}x\log x+\alpha_{1}x+\\
 & +\sum_{k=2}^{\infty}\sum_{l=0}^{k}\left(x^{2k-1}\left(\log x\right)^{l}\alpha_{2k-1,l}+x^{2k}\left(\log x\right)^{l}\alpha_{2k,l}\right).
\end{align*}
In this expansion $\alpha_{0}$ is an $\SU\left(2\right)$ connection
on $\mathbb{S}$, the other coefficients are $\mathfrak{su}\left(\mathbb{S}\right)$-valued
$1$-forms, and the following conditions are satisfied:
\begin{enumerate}
\item $\alpha_{-1}$ is determined by $h_{0}$;
\item $\alpha_{0}$ and $\alpha_{1,1}$ are completely determined by $h_{0}$
and $\left(\partial_{x}h\right)_{|x=0}$;
\item the $\mathfrak{su}\left(\mathbb{S}\right)$-valued $1$-form $\alpha_{1}$
admits an algebraic decomposition into three parts, $\tr\alpha_{1}$,
$\sk\alpha_{1}$ and $\symmtf\alpha_{1}$; while the first two are
completely determined by $h_{0}$ and $\left(\partial_{x}h\right)_{|x=0}$,
the latter is in fact \emph{formally free};
\item for $k\geq2$, the coefficient $\alpha_{k,l}$ is completely determined
by $\left(\partial_{x}^{j}h\right)_{|x=0}$ for $j\leq k$ and $\symmtf\alpha_{1}$.
\end{enumerate}
\end{thm*}
The coefficients $\alpha_{1,1}$ and $\symmtf\alpha_{1}$ are particularly
significant. While the expansion depends a priori on the choice of
$h_{0}\in\mathfrak{c}_{\infty}\left(g\right)$, the coefficient $\alpha_{1,1}$
is in fact an \emph{extrinsic conformal invariant }.
It is the direct analogue of the famous \emph{obstruction tensor}
of Fefferman and Graham, which obstructs the existence of \emph{smooth}
Poincaré--Einstein metrics in odd bulk dimension. For this reason,
we call $\alpha_{1,1}$ the\emph{ $0$-instanton obstruction tensor}.
The second main result justifies this analogy:
\begin{thm*}
(Theorems \ref{thm:0-instanton-obstruction-tensor-conf-invariant},
\ref{thm:obstruction-zero-iff-smooth} and Corollary \ref{cor:asymptotically-PE-implies-vanishing-obstruction})
The $0$-instanton obstruction tensor can be canonically identified
with $2\overline{\W}_{0}^{B}-2\overline{\W}_{0}^{E}$, where $\overline{\W}_{0}$
is the Weyl curvature tensor of $\left(X,\overline{g}\right)$ restricted
to $\partial X$, and $\overline{\W}_{0}^{E},\overline{\W}_{0}^{B}$
is its decomposition into electric and magnetic parts\footnote{In abstract index notation, we have $\overline{\W}^{E}=\overline{\W}_{a00}{^{b}}$
and $\overline{\W}^{B}=\left(*\overline{\W}\right){}_{a00}{^{b}}$
along the fibers of the geodesic boundary defining function $x$.
When restricted to $\partial X$, this decomposition does not depend
on the choice of $h_{0}$. We refer to §\ref{subsec:Geometry-of-4-manifolds-in-a-geodesic-collar}
for more details.}. Furthermore, the following are equivalent:
\begin{enumerate}
\item the $0$-instanton obstruction tensor vanishes;
\item the anti-self-dual Weyl curvature $\overline{\W}^{-}$ vanishes along
$\partial X$;
\item every polyhomogeneous $0$-instanton is smooth modulo gauge.
\end{enumerate}
\end{thm*}
The last part of the paper concerns the \emph{renormalized Yang--Mills
energy} of $0$-instantons. If $A$ is an $\SU\left(2\right)$ instanton
over a \emph{closed }4-manifold, then its Yang--Mills energy\footnote{If $A$ is an $\SU\left(2\right)$ connection, its Yang--Mills energy
is the squared $L^{2}$ norm of its curvature, divided by $8\pi^{2}$.} is an integer, completely dependent on the topology of the $\SU\left(2\right)$
bundle. In our context, the Yang--Mills energy of a $0$-instanton
is necessarily \emph{infinite}. However, it is of frequent practice
in the theory of Poincaré--Einstein manifolds to \emph{renormalize}
divergent integrals. This procedure consists in choosing a metric
$h_{0}\in\mathfrak{c}_{\infty}\left(g\right)$, computing the integral
on the compact submanifold $X_{t}$ obtained by removing the infinite-volume
end $[0,t)\times\partial X$ from the geodesic collar induced by $h_{0}$,
and then taking an expansion as $t\to0$; the renormalized integral
is the constant term in this expansion. The coefficients of this expansion
generally depend on $h_{0}$, but the constant term is often independent
of the choice of $h_{0}$ in $\mathfrak{c}_{\infty}\left(g\right)$.
This leads to very interesting conformal invariants; for example,
Anderson proved in \cite{AndersonVolumeRenormalization} that, if
$\left(X^{4},g\right)$ is Poincaré--Einstein, then the renormalized
\emph{volume} $^{R}\mathrm{Vol}\left(X,g\right)$ satisfies the beautiful
identity
\[
^{R}\mathrm{Vol}\left(X,g\right)=\frac{4}{3}\pi^{2}\chi\left(X\right)-\frac{1}{6}\int_{X}\left|\W\right|_{g}^{2}\dVol_{g}.
\]
An analogous result was later proved by Alexakis--Mazzeo in \cite{AlexsakisMazzeo}
for the\emph{ }renormalized \emph{area }of minimal surfaces: if $\Sigma$
is a properly embedded minimal surface of a Poincaré--Einstein 3-manifold
$\left(X^{3},g\right)$, then its renormalized area satisfies the
equation
\[
^{R}\mathrm{Area}\left(\Sigma\right)=-2\pi\chi\left(\Sigma\right)-\frac{1}{2}\int_{\Sigma}\left|\mathring{\II}\right|_{\iota^{*}g}^{2}\dVol_{\iota^{*}g},
\]
where $\mathring{\II}$ is the trace-free second fundamental form
of $\Sigma$ in $X$. Note that both $^{R}\mathrm{Vol}\left(X,g\right)$
and $^{R}\mathrm{Area}\left(\Sigma\right)$ are conformal invariants.
Our final result provides a similar statement for the renormalized
Yang--Mills energy:
\begin{thm*}
(Theorem \ref{thm:renormalized-YM-energy}) Let $A$ be a $0$-instanton
on a conformally compact 4-manifold $\left(X^{4},g\right)$. If $g$
is asymptotically Poincaré--Einstein to third order, then the renormalized
Yang--Mills energy of $A$ satisfies
\[
^{R}\mathcal{E}_{\mathrm{YM}}\left(A\right)=-\CS_{X}\left(\partial X,\mathfrak{c}_{\infty}\left(g\right)\right).
\]
Here $\CS_{X}\left(\partial X,\mathfrak{c}_{\infty}\left(g\right)\right)$
is the Chern--Simons invariant of the conformal infinity. In particular,
$^{R}\mathcal{E}_{\mathrm{YM}}\left(A\right)$ is independent of $A$
and conformally invariant.
\end{thm*}

\subsection*{Future directions}

The work leading up to this paper was initiated by the author in his
PhD thesis \cite{UsulaPhD}. There, the author started the study of
the \emph{moduli space} of $0$-instantons, with the long-term goal
of obtaining new enumerative invariants. This program is, of course,
closely related to Witten's program concerning knot invariants: Witten's
idea of ``counting'' solutions of the Kapustin--Witten equation,
subject to the Nahm pole boundary condition, requires at least proving
that the moduli space of these solutions is smooth, finite-dimensional,
and compact(ifiable). Witten's program has been carried over only
partially: while the local theory presents no major problems, and
is effectively settled in \cite{MazzeoWittenNahmI, MazzeoWittenNahmII},
the compactness theory still seems very challenging.

The common feature of the self-duality and the Kapustin--Witten equations
(without knots), subject to the Nahm pole boundary condition, is that
they are both \emph{semilinear $0$-elliptic equations} (modulo gauge).
In simple terms, $0$-ellipticity is the natural notion of ellipticity
for geometric operators on conformally compact manifolds; the usual
Laplacians, Dirac operators etc. are all $0$-elliptic operators.
The study of these operators was initiated by Mazzeo in \cite{MazzeoPhD},
Mazzeo--Melrose in \cite{MazzeoMelroseResolvent}, and extended by
Mazzeo in \cite{MazzeoEdgeI}. While standard elliptic operators on
closed manifolds are always Fredholm, the Fredholm theory for $0$-elliptic
operators is much more delicate; in brief, there is a scale of \emph{weighted}
(Sobolev, Hölder, etc.) spaces associated to the geometry, and Fredholmness
for any given weight depends on the invertibility of a family of microlocal
models for the operator, at each point of the boundary.

Mazzeo and Witten proved in \cite{MazzeoWittenNahmI} that the moduli
space of solution of the Kapustin--Witten equations, satisfying the
Nahm pole boundary condition, is indeed a finite-dimensional manifold.
This requires showing that a certain $0$-elliptic operator (the linearized
gauge-fixed equation) is invertible, on an appropriate scale of weighted
spaces. It turns out that the analogous result does \emph{not }hold
for $0$-instantons; more precisely, if one fixes the Nahm pole ``residue''
$\mathfrak{r}$ (cf. §\ref{subsec:Nahm-poles-and-spin-structures-on-the-boundary}),
the moduli space $\mathcal{M}_{\mathfrak{r}}$ of $0$-instantons
with Nahm pole residue $\mathfrak{r}$ is an \emph{infinite-dimensional
}manifold. Indeed, the linearized, gauge-fixed self-duality equation
(which is a $0$-elliptic Dirac operator) is only \emph{semi-Fredholm
surjective}, but with infinite-dimensional kernel, on the appropriate
scale of weighted spaces. These proofs have been developed by the
author in his PhD thesis \cite{UsulaPhD}; since they require completely
different techniques than the ones used here, we decided to develop
this part of the theory in detail in an upcoming paper \cite{UsulaNahmPolesII}.

The fact that the moduli space $\mathcal{M}_{\mathfrak{r}}$ is infinite-dimensional
is to be expected. Indeed, even for standard instantons without singularities,
the moduli space on a compact 4-manifold with boundary is generically
infinite-dimensional. Take for example the case of $S_{+}^{4}$, the
closed upper half of the round 4-sphere. As beautifully explained
by Atiyah in the survey \cite{AtiyahNewInvariants}, the problem of
finding $\SU\left(2\right)$ instantons on $S_{+}^{4}$ is a non-linear,
4-dimensional version of the problem of finding holomorphic functions
on the upper half of $\mathbb{C}\mathrm{P}^{1}$. There are infinitely
many holomorphic functions on $\mathbb{C}\mathrm{P}_{+}^{1}$, but
once one fixes an appropriate boundary value --- a function on $S^{1}$
with only non-negative Fourier coefficients --- the solution becomes
unique. Moreover, solutions which extend holomorphically to the whole
of $\mathbb{C}\mathrm{P}^{1}$ form a 1-dimensional space, corresponding
to the intersection space of the boundary values of holomorphic and
anti-holomorphic functions on $\mathbb{C}\mathrm{P}_{+}^{1}$ ---
the constants. Similarly, the moduli space of instantons on $S^{4}$
can be seen as the ``intersection locus'' of the boundary values
of \emph{self-dual and anti-self-dual} instantons on $S_{+}^{4}$.
This observation leads to very interesting mathematics; we refer to
\cite{DonaldsonFloer} and the citations thereof, for more details
on this topic.

In the present case, we have a similar picture. Indeed, the infinite-dimensionality
of $\mathcal{M}_{\mathfrak{r}}$ is neatly consistent with the results
concerning the asymptotic expansion of $0$-instantons. We have seen
that if $A$ is a $0$-instanton and $\alpha_{h_{0}}\left(x\right)$
is its geodesic normal family associated to a metric $h_{0}\in\mathfrak{c}_{\infty}\left(g\right)$,
then the term $\symmtf\alpha_{1}$ in the asymptotic expansion of
$\alpha_{h_{0}}\left(x\right)$ is gauge-invariant and formally free.
The term $\symmtf\alpha_{1}$ can be interpreted canonically as a
symmetric trace-free $2$-tensor on $\partial X$, and therefore we
can define a ``boundary map''
\begin{align*}
\beta_{h_{0}}:\mathcal{M}_{\mathfrak{r}} & \to C^{\infty}\left(\partial X;\mathring{S}^{2}\left(T^{*}\partial X\right)\right)\\
\left[A\right] & \mapsto\symmtf\alpha_{1}.
\end{align*}
While $\mathcal{M}_{\mathfrak{r}}$ is itself infinite-dimensional,
one expects that the \emph{fibers} of this boundary map will indeed
be (generically) finite-dimensional manifolds. Moreover, one expects
that the range space of $\beta_{h_{0}}$, intersected with the analogous
space of boundary values for \emph{anti-self-dual} $0$-instantons,
will yield a finite-dimensional moduli space. Analogously to the closed
case, this intersection locus should correspond to the moduli space
of $0$-instantons on $S^{4}$ with a Nahm pole along the equatorial
$S^{3}$, subject to appropriate matching conditions along this hypersurface.
More in general, one could consider a closed oriented conformal 
4-manifold $X$, equipped with a closed embedded hypersurface $Y$,
and study the moduli space of $0$-instantons on $X$ satisfying the
Nahm pole boundary condition at $Y$, subject to appropriate matching
conditions along $Y$. This could potentially lead to new enumerative
invariants of 4-manifolds equipped with an embedded hypersurface.
On the technical side, this study requires the $0$-elliptic theory
of Mazzeo--Melrose, along with a recent extension of this theory
recently developed by the author in \cite{Usula0BVP}. We will discuss
these topics in depth in \cite{UsulaNahmPolesII}.

\subsection*{Acknowledgements}

The author is grateful to Joel Fine, for his continuous encouragement
and for many mathematical discussions and suggestions on the topics
of this paper. At various stages, the author was supported by KU Leuven,
the ERC consolidator grant 646649 \textquotedblleft SymplecticEinstein\textquotedblright ,
the EoS grant 40007524 ``Beyond Symplectic Geometry'', and the Wiener--Anspach
Foundation.

\section{\label{sec:Conformally-compact-geometry}Conformally compact geometry}

\subsection{\label{subsec:Definitions-and-main}Definitions and main properties}

Let $X^{n+1}$ be a compact manifold with boundary $\partial X$ and
interior $X^{\circ}$. Recall that a \emph{boundary defining function
}for $X$ is a smooth function $x:X\to[0,+\infty)$ such that $x^{-1}(0)=\partial X$
and its differential $dx$ is nowhere-vanishing along $\partial X$.
Given such a function $x$, any other boundary defining function $x'$
takes the form $x'=e^{\varphi}x$ for some smooth function $\varphi\in C^{\infty}\left(X\right)$.
\begin{defn}
A metric $g$ on the interior $X^{\circ}$ is said to be \emph{conformally
compact} if, for some (hence every) boundary defining function $x$,
the conformally rescaled metric $x^{2}g$ extends to a smooth metric
on the entire manifold $X$.
\end{defn}

\begin{example}
The prototypical example of a conformally compact manifold is hyperbolic
space. Consider the closed unit ball $B^{n}\subset\mathbb{R}^{n}$
with coordinates $y=\left(y^{1},\dots,y^{n}\right)$. The hyperbolic
metric on the interior of $B^{n}$ is given by 
\[
g=\frac{4dy^{2}}{(1-\left|y\right|^{2})^{2}}.
\]
Observing that $\rho=1-\left|y\right|^{2}$ is a boundary defining
function for $B^{n}$, it follows immediately that $g$ is conformally
compact.
\end{example}

Conformally compact manifolds share many geometric properties with
hyperbolic space. Let $\left(X,g\right)$ be conformally compact and
denote the restriction of $g$ to the interior by $g^{\circ}$. The
following properties hold: 
\begin{enumerate}
\item $\left(X^{\circ},g^{\circ}\right)$ is complete, with bounded geometry
and infinite volume.
\item The boundary is located ``at infinity'', meaning that any curve
$\gamma:[0,+\infty)\to X^{\circ}$ approaching a limit $\gamma\left(\infty\right)\in\partial X$
has infinite length.
\item The metric $g$ induces a conformal class on the boundary, known as
the \emph{conformal infinity} of $g$, defined by: 
\[
\mathfrak{c}_{\infty}\left(g\right):=\left\{ \left(x^{2}g\right){}_{|\partial X}\mid x\text{ is a boundary defining function}\right\} .
\]
\item The metric $g^{\circ}$ is ``asymptotically negatively curved''.
Specifically, for any sequence of points $p_{k}\in X^{\circ}$ converging
to $p_{\infty}\in\partial X$, and any sequence of tangent 2-planes
$\pi_{k}\subseteq T_{p_{k}}X^{\circ}$, the sectional curvatures $\kappa_{g^{\circ}}\left(\pi_{k}\right)$
converge to $-\left|dx\right|{}_{x^{2}g}^{2}\left(p_{\infty}\right)$.
\end{enumerate}
The restriction $\kappa_{\infty}^{g}:=-\left|dx\right|{}_{x^{2}g|\partial X}^{2}$
is a fundamental Riemannian invariant of $\left(X,g\right)$, independent
of the choice of $x$. Justified by the last point above, we call
$\kappa_{\infty}^{g}$ the \emph{sectional curvature at infinity }of
$\left(X,g\right)$, and when $\kappa_{\infty}^{g}\equiv-1$, we say
that $\left(X,g\right)$ is \emph{asymptotically hyperbolic}. These
manifolds admit a specialized class of boundary defining functions:
\begin{defn}
Let $\left(X,g\right)$ be asymptotically hyperbolic and fix a representative
$h_{0}\in\mathfrak{c}_{\infty}\left(g\right)$. A boundary defining
function $x$ is called a \emph{geodesic boundary defining function}
for $h_{0}$ if $\left(x^{2}g\right){}_{|\partial X}=h_{0}$ and $\left|dx\right|_{x^{2}g}^{2}\equiv1$
on a neighborhood of $\partial X$ in $X$.
\end{defn}

Graham and Lee established (cf. §5 of \cite{GrahamLeeEinstein}) that,
for any $h_{0}\in\mathfrak{c}_{\infty}\left(g\right)$, there exists
a geodesic boundary defining function inducing $h_{0}$, unique in
the collar neighborhood of the boundary where $\left|dx\right|_{x^{2}g}^{2}\equiv1$.
Geodesic boundary defining functions are extremely useful, in that
they allow the construction of a \emph{normal form }for asymptotically
hyperbolic manifolds. Let $x$ be a geodesic boundary defining function
inducing $h_{0}$. The $x^{2}g$-gradient of $x$, denoted here by
$\partial_{x}$, is transverse to $\partial X$ and inward-pointing.
The flow generated by $\partial_{x}$ then determines a collar neighborhood
$[0,\varepsilon)\times\partial X\hookrightarrow X$ for sufficiently
small $\varepsilon>0$. Provided $\varepsilon$ is small enough that
$\left|dx\right|_{x^{2}g}^{2}\equiv1$ holds throughout the collar,
$x$ becomes the unique geodesic boundary defining function for $h_{0}$
in this region. Furthermore, the metric $g$ can be written in this
collar as
\[
g=\frac{dx^{2}+h\left(x\right)}{x^{2}},
\]
where $h:[0,\varepsilon)\to C^{\infty}\left(\partial X;S^{2}\left(T^{*}\partial X\right)\right)$
is a smooth family of metrics on the boundary satisfying $h\left(0\right)=h_{0}$.
\begin{defn}
The normal form for $g$ described above is called the \emph{geodesic
normal form} for $g$ associated to $h_{0}$.
\end{defn}

Although a conformally compact metric $g$ is formally a metric only
on the interior $X^{\circ}$, it is more appropriate to view it as
a smooth geometric object on the entire compact manifold $X$. This
perspective was formalized by Mazzeo in §2.A of \cite{MazzeoPhD},
drawing on the work of Melrose and Mendoza \cite{MelroseMendozaTotallyCharacteristic, MelroseAPS}.
Let $\mathcal{V}_{0}\left(X\right)$ denote the space of smooth vector
fields on $X$ that vanish at the boundary; we call these \emph{$0$-vector
fields}. As $\mathcal{V}_{0}\left(X\right)$ is a locally finitely
generated, projective module over $C^{\infty}\left(X\right)$, the
Serre--Swan Theorem implies it is isomorphic to the space of sections
of a smooth vector bundle, which we call the \emph{$0$-tangent bundle},
denoted $^{0}TX$. We can analogously construct the ``$0$-version''
of any other tensor bundle. More precisely, if $E$ is a tensor bundle
associated to $\Fr_{\GL\left(n+1,\mathbb{R}\right)}\left(TX\right)$
via a representation of $\GL\left(n+1,\mathbb{R}\right)$, the corresponding
$0$-bundle $^{0}E$ is associated to $\Fr_{\GL\left(n+1,\mathbb{R}\right)}\left(^{0}TX\right)$
via the same representation. Key examples include:
\begin{enumerate}
\item the \emph{$0$-cotangent bundle} $^{0}T^{*}X$ (the dual of $^{0}TX$);
\item the bundle of \emph{$0$-$k$-forms} $^{0}\Lambda^{k}$ (the $k$-th
exterior power of $^{0}T^{*}X$);
\item the bundle of \emph{symmetric $0$-$2$-tensors} $S^{2}\left(^{0}T^{*}X\right)$.
\end{enumerate}
Sections of $^{0}E$ correspond to sections of $E$ with specific
decay or singularity rates at the boundary. Specifically, if $E$
is a tensor bundle of type $\left(r,s\right)$, a section $\omega$
of $^{0}E$ can be written as $\omega=x^{r-s}\overline{\omega}$,
where $\overline{\omega}$ is a smooth section of $E$. For instance,
a smooth section of $^{0}\Lambda^{k}$ corresponds to an $O\left(x^{-k}\right)$
$k$-form. In this framework, a conformally compact metric on $X$
is simply a smooth bundle metric on $^{0}TX$.

Let $\left(X^{n+1},g\right)$ be a conformally compact manifold. Then,
for every $p\in\partial X$, $X$ has an asymptotic half-space model
$X_{p}$, with a \emph{hyperbolic }metric of constant sectional curvature
$\kappa_{\infty}^{g}\left(p\right)$. As a manifold, $X_{p}$ is simply
the closed inward-pointing half of $T_{p}X$. In order to define $g_{p}$,
choose a boundary defining function $x$ for $X$, call $h_{0}=\left(x^{2}g\right)_{|\partial X}$,
and let $y^{1},...,y^{n}$ be normal coordinates for $\left(\partial X,h_{0}\right)$
centered at $p$. We call coordinates $x,y^{1},...,y^{n}$ for $X$
constructed in this way \emph{normal half-space coordinates }for $\left(X,g\right)$
associated to $h_{0}$. These coordinates induce global linear coordinates
for $T_{p}X$, which we still denote by $x,y^{1},...,y^{n}$ with
slight abuse of notation. In these coordinates, the half-space $X_{p}$
is defined by $x\geq0$, and the hyperbolic metric $g_{p}$ is
\[
g_{p}:=\frac{dx^{2}}{-\kappa_{\infty}^{g}\left(p\right)x^{2}}+\frac{dy^{2}}{x^{2}}.
\]
One can check that the definition of $g_{p}$ does not depend on the
choices made.

\subsection{\label{subsec:Poincar=0000E9=002013Einstein-metrics}Poincaré--Einstein
metrics}

So far, the conformally compact metrics considered were smooth. We
now consider a larger space of \emph{polyhomogeneous} conformally
compact metrics: more precisely, we consider $O\left(1\right)$ polyhomogeneous
bundle metrics on $^{0}TX$. We refer to §\ref{subsec:Polyhomogeneity}
and the references cited there, for details on polyhomogeneity.
\begin{defn}
A polyhomogeneous conformally compact metric $\left(X^{n+1},g\right)$
is called \emph{Poincaré--Einstein }if $g$ satisfies the equation
$\Ric\left(g\right)+ng=0$.
\end{defn}

\begin{rem}
In \cite{FeffermanGrahamAmbient}, Fefferman and Graham implicitly
assume that the metric is polyhomogeneous.
\end{rem}

Let $\left(X^{n+1},g\right)$ be a polyhomogeneous Poincaré--Einstein
manifold, with $n+1\geq3$. The specific normalization for the Einstein
equation implies that $g$ has the same scalar curvature of the hyperbolic
space $\HH^{n+1}$. It turns out that the Poincaré--Einstein condition
then implies that $g$ must be \emph{asymptotically hyperbolic}. Therefore,
given $h_{0}\in\mathfrak{c}_{\infty}\left(g\right)$, we can consider
a geodesic boundary defining function $x$ inducing $h_{0}$, and
the corresponding geodesic normal form
\[
g=\frac{dx^{2}+h\left(x\right)}{x^{2}}.
\]
$h\left(x\right)$ is then a $O\left(1\right)$ polyhomogeneous family
of metrics on $\partial X$ defined on $[0,\varepsilon)$, such that
$h\left(0\right)=h_{0}$. In \cite{FeffermanGrahamAmbient}, Fefferman
and Graham studied in detail the asymptotic expansion of $h\left(x\right)$
as $x\to0$. In particular, they showed that the coefficients of the
expansion satisfy various interesting properties:
\begin{enumerate}
\item if $n+1$ is \emph{even} or equal to $3$, then $h\left(x\right)$
has an\emph{ }expansion of the form
\[
h\left(x\right)\sim h_{0}+h_{2}x^{2}+\cdots+h_{n-1}x^{n-1}+h_{n}x^{n}+o\left(x^{n}\right);
\]
\item if $n+1\geq4$ is \emph{odd}, then $h\left(x\right)$ has an expansion
of the form
\[
h\left(x\right)\sim h_{0}+h_{2}x^{2}+\cdots+h_{n-2}x^{n-2}+\mathcal{O}x^{n}\log x+h_{n}x^{n}+o\left(x^{n}\right).
\]
\end{enumerate}
In both cases, the odd terms $h_{2k-1}$ with $1\leq2k-1<n$ are all
zero, and the terms $h_{2k}$ with $2\leq2k<n$ are all determined
by $h_{0}$ via differential relations. In the case $n+1$ even, the
$o\left(x^{n}\right)$ remainder contains only terms of the form $x^{k}h_{k}$
with $k>n$ (i.e. the expansion is smooth), and these coefficients
$h_{k}$ are determined by $h_{0}$ and $h_{n}$ by differential relations.
When $n+1$ is odd, instead, we potentially see log terms: the coefficient
$\mathcal{O}$ of the $x^{n}\log x$ term is a fundamental \emph{conformal
invariant} of the conformal infinity $\mathfrak{c}_{\infty}\left(g\right)$,
and it obstructs the smoothness of the remaining part of the expansion.
More precisely, the $o\left(x^{n}\right)$ part of the expansion can
a priori contain terms of the form $h_{nk+i,l}x^{nk+i}\left(\log x\right)^{l}$
with $k\geq1$, $0\leq i<n$, and $l\leq k$: however, the coefficients
with $l>0$ all vanish if $\mathcal{O}\equiv0$. For this reason,
$\mathcal{O}$ is known as the \emph{obstruction tensor }of $\mathfrak{c}_{\infty}\left(g\right)$.
\begin{rem}
\label{rem:regularity-PE}Using the concept of log-smoothness introduced
in Definition \ref{def:log-smoothness}, we can reformulate the regularity
result above: if $n+1$ is odd, then $h\left(x\right)$ is \emph{log-smooth
of order $n$}. In fact, the main result of \cite{ChruscielDelayLeeSkinnerBoundaryRegularity}
is that, if a $C^{2}$ Poincaré--Einstein metric $g$ has smooth
conformal infinity, then it is $C^{1,\lambda}$ diffeomorphic to a
Poincaré--Einstein metric which is either smooth (if $n+1$ is even
or equal to $3$) or log-smooth of order $n$ (if $n+1\geq4$ is odd).
Moreover, if $n+1\geq4$ is odd, then $g$ is $C^{1,\lambda}$ diffeomorphic
to a \emph{smooth} Poincaré--Einstein metric if and only if the obstruction
tensor vanishes.
\end{rem}

\begin{rem}
\label{rem:asymptotically-PE-equals-condition-on-h1-and-h2}Assume
that $n+1\geq4$. Let $h_{0}\in\mathfrak{c}_{\infty}\left(g\right)$,
and consider the geodesic normal form of $g$ associated to $h_{0}$
and the corresponding expansion $h\left(x\right)\sim h_{0}+h_{1}x+h_{2}x^{2}+o\left(x^{2}\right)$.
Then $g$ is asymptotically Poincaré--Einstein to second order, i.e.
$\Ric\left(g\right)+ng=o\left(x^{2}\right)$ as a section of $S^{2}\left(^{0}T^{*}X\right)$,
if and only if $h_{1}=0$ and $h_{2}=-\PP\left(h_{0}\right)$ where
$\PP\left(h_{0}\right)$ is the \emph{Schouten tensor} of the metric
$h_{0}$ (cf. Theorem 7.4 of \cite{FeffermanGrahamAmbient}). It is
easy to prove (cf. Lemma \ref{lem:computations-in-collar}) that $h_{1}=-2\II$,
where $\II$ is the second fundamental form of $\left(\partial X,h_{0}\right)$
in $\left(X,x^{2}g\right)$. Therefore, $h_{1}=0$ if and only if
$\left(\partial X,h_{0}\right)$ is totally geodesic\emph{ }in $\left(X,x^{2}g\right)$.
\end{rem}

\section{\label{sec:0-connections}$0$-connections}

In this section we introduce \emph{$0$-connections}, the main characters
of this paper. Fix a vector bundle $E\to X$, and denote by $E^{\circ}$
the restriction of $E$ to the interior $X^{\circ}$.
\begin{defn}
A \emph{$0$-connection} on a vector bundle $E\to X$ is a $\mathbb{R}$-bilinear
map
\[
A:\mathcal{V}_{0}\left(X\right)\times C^{\infty}\left(X;E\right)\to C^{\infty}\left(X;E\right)
\]
which is $C^{\infty}\left(X\right)$ linear in the first entry, and
satisfies the Leibniz identity
\[
A_{V}\left(fs\right)=\left(Vf\right)s+fA_{V}s
\]
for every $V\in\mathcal{V}_{0}\left(X\right)$, $f\in C^{\infty}\left(X\right)$,
and $s\in C^{\infty}\left(X;E\right)$.
\end{defn}

In other words, $0$-connections allow to take directional derivatives
along $0$-vector fields on $X$. Since $0$-vector fields are unconstrained
in the interior, the restriction of a $0$-connection $A$ on $E$
to $E^{\circ}$ determines a standard connection on $E^{\circ}$.
Moreover, since $\mathcal{V}_{0}\left(X\right)\subseteq\mathcal{V}\left(X\right)$,
standard connections on $E$ are $0$-connections. On the other hand,
generically, the converse is not true. Indeed, essentially by definition,
every $0$-connection on $E$ can be written in the form $A=D+a$,
where $D$ is a standard connection on $E$ and $a$ is a $\mathfrak{gl}\left(E\right)$-valued
$0$-$1$-form on $X$. Thus, $A$ is a standard connection if and
only if $a$ is a $1$-form on $X$, i.e. equivalently it is a $O\left(x\right)$
$0$-$1$-form. The restriction of $a$ to $\partial X$ therefore
obstructs $A$ to extend from the interior to a standard connection:
\begin{defn}
Let $A$ be a $0$-connection on $E$. The \emph{residue} of $A$
is the bundle map $\res_{A}:{^{0}TX_{|\partial X}}\to\mathfrak{gl}\left(E_{|\partial X}\right)$
defined as the restriction $\left(A-D\right)_{|\partial X}$, where
$D$ is an arbitrary standard connection on $E$.
\end{defn}

\begin{rem}
The definition above is well posed. Indeed, if $D_{1},D_{2}$ are
two standard connections on $E$, then their difference $D_{1}-D_{2}$
is a smooth section of $\Lambda^{1}\otimes\mathfrak{gl}\left(E\right)$;
seen as a section of ${^{0}\Lambda^{1}}\otimes\mathfrak{gl}\left(E\right)$,
this is $O\left(x\right)$, and therefore the restriction $\left(A-D\right)_{|\partial X}$
does not depend on $D$.
\end{rem}

\begin{rem}
\label{rem:residual-0conn}The residue of a $0$-connection $A$ on
$E$ can be interpreted equivalently as a family $p\mapsto A_{p}$
of $0$-connections on the trivial bundles $E_{p}\times X_{p}\to X_{p}$
over the model spaces $X_{p}$ (cf. §\ref{subsec:Definitions-and-main}),
smoothly parametrized by $p\in\partial X$. Concretely, $A_{p}$ is
defined as follows. Given a local frame $s_{1},...,s_{k}$ for $E$
near $p$, we can write in this frame $A=d+a$ where $a$ is an $\mathfrak{gl}\left(k,\mathbb{R}\right)$-valued
locally defined $0$-$1$-form. The evaluation $a_{p}$ can be canonically
interpreted as a globally defined $\mathfrak{gl}\left(k,\mathbb{R}\right)$-valued
$0$-$1$-form on $X_{p}$. The local expression of $A_{p}$ in terms
of the global frame $s_{1|p},...,s_{k|p}$ of $E_{p}\times X_{p}$
is then precisely $A_{p}=d+a_{p}$.
\end{rem}

The next lemma shows that $0$-connections arise naturally in conformally
compact geometry:
\begin{lem}
Let $\left(X,g\right)$ be a conformally compact manifold. Then the
Levi-Civita connection $\nabla^{\circ}$ of $g^{\circ}$ on any tensor
bundle of $X^{\circ}$ extends canonically to a $0$-connection on
the corresponding $0$-tensor bundle.
\end{lem}

\begin{proof}
It suffices to show that the Levi-Civita connection on $TX^{\circ}$
extends to a $0$-connection on $^{0}TX$. If $A$ is a $0$-connection
on $^{0}TX$, then the formula
\[
T^{A}\left(U,V\right):=A_{U}V-A_{V}U-\left[U,V\right],\quad U,V\in\mathcal{V}_{0}\left(X\right)
\]
defines a section $T^{A}$ of ${^{0}\Lambda^{2}}\otimes{^{0}TX}$
which we can call the \emph{$0$-torsion}; calling $A^{\circ}$ the
restriction of $A$ to the interior, the restriction of $T^{A}$ to
$X^{\circ}$ is precisely the torsion of $A^{\circ}$. Now, exactly
as for standard connections, one can prove that there exists a unique
metric $0$-connection $\nabla$ on $^{0}TX$ with vanishing $0$-torsion.
This $0$-connection must then extend the Levi-Civita connection in
the interior.
\end{proof}
\begin{example}
\label{exa:LC-hyperbolic}Let us compute the Levi-Civita $0$-connection
on the $0$-cotangent bundle of the hyperbolic space $\HH^{n+1}$.
It is convenient to work in the non-compact half-space model, with
coordinates $x\geq0,y\in\mathbb{R}^{n}$ and metric
\[
g=\frac{dx^{2}+dy^{2}}{x^{2}}.
\]
Define
\begin{align*}
e^{0} & =-\frac{dx}{x},\quad e^{j}=\frac{dy^{j}}{x}.
\end{align*}
Then $e^{0},...,e^{n}$ is an $\SO\left(n+1\right)$\footnote{We use the ``Stokes'' convention to orient manifolds with boundary.
According to this convention, a basis $v_{1},...,v_{n}$ of $T_{p}\partial X$
is positively oriented if and only if the basis $\vec{n},v_{1},...,v_{n}$
of $T_{p}X$ is positively oriented, where $\vec{n}$ is an \emph{outward-pointing}
tangent vector in $T_{p}X$.} global frame of $^{0}T^{*}\HH^{n+1}$. The Levi-Civita $0$-connection
with respect to this frame takes the form $\nabla=d+\omega$, where
$\omega$ is a skew-symmetric matrix of $0$-$1$-forms. The torsion-free
condition can be expressed in matrix form as
\[
de+\omega\land e=0,
\]
where $e$ is the column vector of $0$-$1$-forms $\left(e^{0},...,e^{n}\right)^{T}$.
A rapid computation shows that
\[
\omega=\left(\begin{matrix}0 & -e^{1} & \cdots & -e^{n}\\
e^{1} & 0\\
\vdots &  & \ddots\\
e^{n} &  &  & 0
\end{matrix}\right).
\]
\end{example}

\begin{rem}
\label{rem:residual-connection-of-LC}Let $g$ be a conformally compact
metric on $X$. As observed in §\ref{subsec:Definitions-and-main},
for every $p\in\partial X$ the model space $X_{p}$ inherits a rescaled
hyperbolic metric $g_{p}$. If $\nabla$ is the Levi-Civita $0$-connection
on $^{0}TX$, then its residual $0$-connection $\nabla_{p}$ on $^{0}TX_{p}$
is the Levi-Civita $0$-connection of the model hyperbolic space $\left(X_{p},g_{p}\right)$.
To see this, it is sufficient to note that if $A$ is a metric $0$-connection
on $^{0}TX$, then the $0$-torsion of $A$ evaluated at a point $p\in\partial X$
can be canonically identified with the $0$-torsion of the residual
$0$-connection $A_{p}$ on $^{0}TX_{p}$.
\end{rem}

Given a $0$-connection $A$ on $E$, the exterior differential $d_{A}$
extends from the interior to an operator acting on $\mathfrak{gl}\left(E\right)$-valued
$0$-differential forms. As in the standard case, we have $d_{A}^{2}=\left[F_{A}\land\cdot\right]$;
however, in our case, the curvature form $F_{A}$ extends from the
interior to a $\mathfrak{gl}\left(E\right)$-valued $0$-$2$-form.

\section{\label{sec:The-Nahm-pole-boundary-condition}The Nahm pole boundary
condition}

In this section, we introduce the \emph{Nahm pole boundary condition
}for $0$-connections on asymptotically hyperbolic 4-manifolds, and
we study its relation with the self-duality equation.

\subsection{\label{subsec:Nahm-poles-and-self-duality}Nahm poles and self-duality}

Let us fix an oriented conformally compact 4-manifold $\left(X^{4},g\right)$,
and an $\SU\left(2\right)$ vector bundle $E\to X$\footnote{We focus on $\SU\left(2\right)$ gauge theory mainly for simplicity,
but everything we say in this paper has a direct counterpart in $\SO\left(3\right)$
gauge theory. The extension to more complicate Lie groups is delicate,
and will not be treated here.}. Since the self-duality equation is conformally invariant, and every
conformally compact metric has an asymptotically hyperbolic representative,
we can assume without loss of generality that $g$ is asymptotically
hyperbolic. The Hodge star operator $*$ on $\left(X^{\circ},g^{\circ}\right)$
extends from the interior to a bundle map $*:{^{0}\Lambda^{k}}\to{^{0}\Lambda^{4-k}}$;
in particular, it induces an automorphism of the bundle $^{0}\Lambda^{2}$
which squares to $1$; its $\pm1$ eigenbundles are denoted by $^{0}\Lambda^{\pm}$.
Concretely, a section of $^{0}\Lambda^{\pm}$ can be written as $x^{-2}\overline{\omega}$,
where $x$ is an auxiliary boundary defining function and $\overline{\omega}$
is a smooth section of the usual bundle $\Lambda^{\pm}$ associated
to the conformal class of $x^{2}g$. If $A$ is a $0$-connection
on $E$, the curvature $F_{A}\in{^{0}\Omega_{\mathfrak{su}\left(E\right)}^{2}\left(X\right)}$
decomposes orthogonally into its self-dual and anti-self-dual parts,
$F_{A}=F_{A}^{+}+F_{A}^{-}$.
\begin{defn}
A $0$-connection $A$ on $E$ is called \emph{self-dual} if $F_{A}^{-}=0$.
\end{defn}

From now on, we will focus on the study of self-dual $0$-connections.
The first step in our analysis is the formulation of an asymptotic
boundary condition, which we call the \emph{Nahm pole boundary condition}:
the name is chosen for its similarity with the boundary condition
introduced by Witten in the study of the Kapustin--Witten equations
\cite{WittenFivebranes}, cf. §\ref{subsec:Comparison-with-KW}.

Let $x$ be an auxiliary boundary defining function $x$. The $0$-$1$-form
$dx/x$ depends on $x$, but its restriction to $\partial X$ does
not: indeed, if $\tilde{x}=e^{\varphi}$ is another boundary defining
function, then
\[
\frac{d\tilde{x}}{\tilde{x}}=\frac{dx}{x}+e^{\varphi}d\varphi
\]
and, seen as a $0$-$1$-form, $e^{\varphi}d\varphi$ \emph{vanishes}
along the boundary. Therefore, the annihilator of $\left(dx/x\right)_{|\partial X}$
is a \emph{canonical} rank 3 smooth subbundle $K\subset{}^{0}TX_{|\partial X}$,
independent of $x$. This bundle inherits a metric from $g$ and an
orientation from $\partial X$, i.e. it has a natural $\SO\left(3\right)$
structure.


\begin{rem}
\label{rem:non-canonical-isomorphism-K-TbX}In general, for any compact
manifold with boundary $X^{n+1}$, the bundle $K$ defined above is
non-canonically isomorphic to $T\partial X$. More precisely, given
a boundary defining function $x$, the map on sections
\begin{align*}
\mathcal{V}\left(X\right) & \to\mathcal{V}_{0}\left(X\right)\\
V & \mapsto xV
\end{align*}
defines an isomorphism of vector bundles $TX\to{^{0}TX}$, which restricts
to an isomorphism $T\partial X\to K$. Note that if $x,x'$ are two
boundary defining functions, and the restriction of (the unique extension
to $X$ of) $x^{-1}x'$ to $\partial X$ is $1$, then $x$ and $x'$
induce the same isomorphism $T\partial X\to K$. In particular, if
$X$ is equipped with a conformally compact metric $g$, and $x$
is a boundary defining function for $X$ inducing the metric $h_{0}\in\mathfrak{c}_{\infty}\left(g\right)$,
then the isomorphism $T\partial X\to K$ induced by $x$ depends only
on $h_{0}$; in this case, it is an isomorphism of $\SO\left(3\right)$
bundles, where $T\partial X$ is equipped with the metric $h_{0}$.
\end{rem}


As every $\SO\left(3\right)$ vector bundle, $K$ can be canonically
identified with the bundle of Lie algebras $\mathfrak{so}\left(K\right)$:
fibrewise, this identification corresponds to the ``cross-product''
identification $\mathbb{R}^{3}\simeq\mathfrak{so}\left(3\right)$
mapping $e_{i}$ to $\mathfrak{a}_{i}$, where
\[
\mathfrak{a}_{1}=\left(\begin{matrix}0 & 0 & 0\\
0 & 0 & -1\\
0 & 1 & 0
\end{matrix}\right),\mathfrak{a}_{2}=\left(\begin{matrix}0 & 0 & 1\\
0 & 0 & 0\\
-1 & 0 & 0
\end{matrix}\right),\mathfrak{a}_{3}=\left(\begin{matrix}0 & -1 & 0\\
1 & 0 & 0\\
0 & 0 & 0
\end{matrix}\right).
\]

\begin{defn}
\label{def:nahm-pole-bc}An $\SU\left(2\right)$ $0$-connection $A$
on $E$ satisfies the \emph{Nahm pole boundary condition} if its residue
satisfies the following two properties:
\begin{enumerate}
\item $\res_{A}$ vanishes on the orthogonal complement $K^{\bot}\subset{^{0}TX_{|\partial X}}$;
\item the restriction of $\res_{A}$ to $K$ defines an isomorphism of Lie
algebra bundles $\res_{A}:K\to\mathfrak{su}\left(E_{|\partial X}\right)$.
\end{enumerate}
\end{defn}

\begin{rem}
\label{rem:adaptedness_at_infinity}The first property of the previous
definition does not require $X$ to be 4-dimensional, nor does it
require $E$ to be an $\SU\left(2\right)$ vector bundle. Moreover,
this condition only depends on the restriction of $g$ to $^{0}TX_{|\partial X}$.
We can express it equivalently by saying that $\res_{A}$ is annihilated
by the $0$-vector field on $X$ $g$-dual of $dx/x$, restricted
to $\partial X$. If a $0$-connection $A$ satisfies this property,
then we say that $A$ is \emph{adapted to $g$ at infinity}. Note
that the computation of Example \ref{exa:LC-hyperbolic} implies that,
if $\left(X,g\right)$ is asymptotically hyperbolic, then the Levi-Civita
$0$-connections on its $0$-tensor bundles are all adapted to $g$
at infinity.
\end{rem}

It is useful to clarify the meaning of the Nahm pole boundary condition
in terms of a local $\SU\left(2\right)$ frame of $E$. Choose a metric
$h_{0}$ in the conformal infinity $\mathfrak{c}_{\infty}\left(g\right)$
of $g$, a geodesic boundary defining function $x$ for $\left(X,g\right)$
inducing $h_{0}$, and the corresponding geodesic normal form
\[
g=\frac{dx^{2}+h\left(x\right)}{x^{2}}
\]
in a geodesic collar $[0,\varepsilon)\times\partial X\hookrightarrow X$.
Let $s_{1},s_{2}$ be an $\SU\left(2\right)$ frame for $E$ defined
in the collar\footnote{It is possible to find such a frame, because $\SU\left(2\right)$
vector bundles over a closed oriented 3-manifold are trivial.}. In this frame, we can write the $0$-connection $A$ as
\[
A=d+\frac{b\left(x\right)}{x}+c\left(x\right)\frac{dx}{x},
\]
where $b\left(x\right)$ is a smooth family of $\mathfrak{su}\left(2\right)$-valued
$1$-forms on $\partial X$ defined on $[0,\varepsilon)$, and $c\left(x\right)$
is a smooth family of $\mathfrak{su}\left(2\right)$-valued functions
on $\partial X$ defined on $[0,\varepsilon)$. The fact that $A$
is adapted to $g$ at infinity is then equivalent to the fact that
$c\left(0\right)$ vanishes identically on $\partial X$. Concerning
the second point of the Nahm pole boundary condition, fix the following
basis of $\mathfrak{su}\left(2\right)$:
\[
\sigma_{1}=\left(\begin{matrix}i\\
 & -i
\end{matrix}\right),\sigma_{2}=\left(\begin{matrix} & 1\\
-1
\end{matrix}\right),\sigma_{3}=\left(\begin{matrix} & i\\
i
\end{matrix}\right).
\]
This basis is chosen so that the following structure equations are
satisfied:
\[
\left[\frac{\sigma_{i}}{2},\frac{\sigma_{j}}{2}\right]=\varepsilon_{ij}{^{k}}\frac{\sigma_{k}}{2}.
\]
Using the Einstein summation convention, we can then write
\[
b\left(x\right)=b^{i}\left(x\right)\frac{\sigma_{i}}{2}
\]
for a triple $b^{1}\left(x\right),b^{2}\left(x\right),b^{3}\left(x\right)$
of families of $1$-forms on $\partial X$ defined on $[0,\varepsilon)$.
The second point of the Nahm pole boundary condition then says exactly
that the triple $b^{1}\left(0\right),b^{2}\left(0\right),b^{3}\left(0\right)$
forms an $\SO\left(3\right)$ coframe of $\left(\partial X,h_{0}\right)$.

The next proposition shows that the Nahm pole boundary condition is
closely related to self-duality:
\begin{prop}
\label{prop:nahm-residues}Let $A$ be an $\SU\left(2\right)$ $0$-connection
on $E$. Suppose that $A$ is adapted to $g$ at infinity, and that
$\res_{A}$ does not vanish identically on any connected component
of $\partial X$. Then $A$ satisfies the Nahm pole boundary condition
if and only if $A$ is \emph{asymptotically self-dual}, i.e. if and
only if $F_{A}^{-}$ vanishes along $\partial X$.
\end{prop}

\begin{proof}
Let us work in a geodesic collar as above. By the previous discussion,
since $A$ is adapted to $g$ at infinity, we can write $A$ in terms
of an $\SU\left(2\right)$ frame for $E$ defined in the collar as
\[
A=d+x^{-1}\theta+o\left(x^{-1}\right),
\]
where $\theta$ is an $\mathfrak{su}\left(2\right)$-valued $1$-form
on $\partial X$, and the $o\left(x^{-1}\right)$ part is a $o\left(x^{-1}\right)$
$\mathfrak{su}\left(2\right)$-valued $1$-form on $X$. We then have
\begin{align*}
F_{A} & =d\left(x^{-1}\theta\right)+x^{-2}\left(\theta\land\theta\right)+o\left(x^{-2}\right)\\
 & =-x^{-2}\left(dx\land\theta\right)+x^{-2}\left(\theta\land\theta\right)+o\left(x^{-2}\right).
\end{align*}
Calling $*$ the Hodge star on $\left(X,g\right)$, and $\star_{0}$
the Hodge star on $\left(\partial X,h_{0}\right)$, we have
\[
*F_{A}=x^{-2}\star_{0}\theta-x^{-2}dx\land\star_{0}\left(\theta\land\theta\right)+o\left(x^{-2}\right).
\]
Therefore, $F_{A}^{-}$ vanishes along the boundary as a section of
$^{0}\Lambda^{-}$ if and only if
\[
\theta=\star_{0}\left(\theta\land\theta\right).
\]
Write $\theta=\theta^{i}\otimes\frac{\sigma_{i}}{2}$, where the $\theta^{i}$
are $1$-forms on $\partial X$ and the basis $\sigma_{1},\sigma_{2},\sigma_{3}$
of $\mathfrak{su}\left(2\right)$ is defined as above. The equation
above then translates into the system of equations
\begin{align*}
\theta^{1} & =\star_{0}\left(\theta^{2}\land\theta^{3}\right)\\
\theta^{2} & =\star_{0}\left(\theta^{3}\land\theta^{1}\right)\\
\theta^{3} & =\star_{0}\left(\theta^{1}\land\theta^{2}\right).
\end{align*}
Taking duals with respect to $h_{0}$, we obtain the system of equations
for the three vector fields $\theta_{1},\theta_{2},\theta_{3}$
\begin{align*}
\theta_{1} & =\theta_{2}\times\theta_{3}\\
\theta_{2} & =\theta_{3}\times\theta_{1}\\
\theta_{3} & =\theta_{1}\times\theta_{2}
\end{align*}
where $\times$ is the fibrewise cross-product on $T\partial X$ induced
by $h_{0}$ and the orientation. From these equations, we see that
each of the vectors $\theta_{1},\theta_{2},\theta_{3}$ must be orthogonal
to the space spanned by the other two. Therefore, taking pointwise
norms, we obtain the equation
\[
\begin{cases}
\left|\theta_{1}\right| & =\left|\theta_{2}\right|\left|\theta_{3}\right|\\
\left|\theta_{2}\right| & =\left|\theta_{3}\right|\left|\theta_{1}\right|\\
\left|\theta_{3}\right| & =\left|\theta_{1}\right|\left|\theta_{2}\right|
\end{cases}.
\]
On the other hand, since the triple product $\left\langle a\times b,c\right\rangle $
is invariant under cyclic permutations of the triple $\left(a,b,c\right)$,
it follows that $\left|\theta_{1}\right|=\left|\theta_{2}\right|=\left|\theta_{3}\right|$
at each point of $\partial X$. It follows that, at each point of
$\partial X$, we either have $\left|\theta_{1}\right|=\left|\theta_{2}\right|=\left|\theta_{3}\right|=0$
or $\left|\theta_{1}\right|=\left|\theta_{2}\right|=\left|\theta_{3}\right|=1$.
Since $\res_{A}$ is not identically zero on any connected component
of $\partial X$, it follows by continuity that we must have $\left|\theta_{1}\right|=\left|\theta_{2}\right|=\left|\theta_{3}\right|=1$
everywhere. In other words, $\theta_{1},\theta_{2},\theta_{3}$ must
be an $\SO\left(3\right)$ frame of $T\partial X$. This is equivalent
to the second condition in Definition \ref{def:nahm-pole-bc}.
\end{proof}
\begin{rem}
We remark that, if a $0$-connection $A$ satisfies the Nahm pole
boundary condition, then $A$ necessarily has \emph{infinite }energy.
Indeed, the field strength $\left|F_{A}\right|^{2}$ is a $O\left(1\right)$
polyhomogeneous function on $X$, but it does not vanish along $\partial X$
because of the Nahm pole condition. Therefore, $\int_{X}\left|F_{A}\right|^{2}\dVol_{g}$
is infinite because every conformally compact manifold has infinite
volume.
\end{rem}

\subsection{\label{subsec:Nahm-poles-and-spin-structures-on-the-boundary}Nahm
poles and spin structures on the boundary}

By definition, the choice of residue in the Nahm pole boundary condition
amounts to the choice of an isomorphism of Lie algebra bundles $\mathfrak{r}:K\to\mathfrak{su}\left(E_{|\partial X}\right)$.
We call such isomorphisms \emph{Nahm pole residues}. In this subsection
we describe in detail the information encoded by $\mathfrak{r}$.

Recall (cf. §\ref{subsec:Geometry-of-spin-3D}) that a spin structure
on an oriented Riemannian 3-manifold is determined by an $\SU\left(2\right)$
vector bundle $\mathbb{S}$, together with a Clifford multiplication
$\cl:T^{*}\partial X\to\mathfrak{su}\left(\mathbb{S}\right)$. Define
now $\mathbb{S}:=E_{|\partial X}$, and given a choice of $h_{0}\in\mathfrak{c}_{\infty}\left(g\right)$,
define $\mathfrak{r}_{h_{0}}:T\partial X\to\mathfrak{su}\left(\mathbb{S}\right)$
as the pre-composition of $\mathfrak{r}$ by the $\SO\left(3\right)$
isomorphism $\left(T\partial X,h_{0}\right)\to K$ induced by $h_{0}$
(cf. Remark \ref{rem:non-canonical-isomorphism-K-TbX}). We define
\[
\cl_{h_{0}}:=2\left(\mathfrak{r}_{h_{0}}\circ\sharp_{h_{0}}\right),
\]
where $\sharp_{h_{0}}:T^{*}\partial X\to T\partial X$ is the usual
musical isomorphism induced by $h_{0}$. It is easy to check that
$\mathfrak{r}$ is a Nahm pole residue if and only if $\cl_{h_{0}}$
satisfies the axioms of Clifford multiplication. Thus, $\mathfrak{r}$
induces a spin structure $\left(\mathbb{S},\cl_{h_{0}}\right)$ on
$\left(\partial X,h_{0}\right)$ for every $h_{0}\in\mathfrak{c}_{\infty}\left(g\right)$.

Since the self-duality equation is gauge invariant, it is important
to understand how gauge transformations affect the residue of a $0$-connection
and the Nahm pole boundary condition. Let $\Phi$ be a section of
$\SU\left(E\right)$, and let $A$ be a $0$-connection on $E$. Then
the pull-back $\Phi^{*}A$ is $\Phi^{*}A=A+\Phi^{*}d_{A}\Phi$. Writing
$A=d+a$ with respect to an $\SU\left(2\right)$ frame of $E$, and
interpreting $\Phi$ in this frame as an $\SU\left(2\right)$-valued
function, we have
\[
\Phi^{*}A=d+\Phi^{*}d\Phi+\Phi^{*}a\Phi.
\]
Note that the term $\Phi^{*}a\Phi$ is an $\mathfrak{su}\left(2\right)$-valued
$0$-$1$-form, while the term $\Phi^{*}d\Phi$ is an $\mathfrak{su}\left(2\right)$-valued
$1$-form \emph{smooth up to the boundary}. In particular, seen as
a $0$-$1$-form, $\Phi^{*}d\Phi$ \emph{vanishes} along the boundary.
Therefore, the term of $\Phi^{*}d\Phi$ does not affect the residue
of $\Phi^{*}A$, and we have
\[
\res_{\Phi^{*}A}=\Phi_{|\partial X}^{*}\res_{A}\Phi_{|\partial X}.
\]
This computation justifies the following
\begin{defn}
We say that two Nahm pole residues $\mathfrak{r},\mathfrak{r}':K\to\mathfrak{su}\left(E_{|\partial X}\right)$
are \emph{equivalent }if there exists a gauge transformation $\Phi$
of $E$ such that
\[
\mathfrak{\mathfrak{r}}'=\Phi_{|\partial X}^{*}\mathfrak{r}\Phi_{|\partial X}.
\]
\end{defn}

Directly from the definition, we can classify Nahm pole residues up
to equivalence. Fix an auxiliary $\SU\left(2\right)$ frame of $E_{|\partial X}$.
Then, given two Nahm poles $\mathfrak{r}$ and $\mathfrak{\mathfrak{r}}'$,
the map $\rho_{\mathfrak{\mathfrak{r}},\mathfrak{\mathfrak{r}}'}=\mathfrak{\mathfrak{r}}'\circ\mathfrak{\mathfrak{r}}^{-1}:\mathfrak{su}\left(E_{|\partial X}\right)\to\mathfrak{su}\left(E_{|\partial X}\right)$
can be thought of as a map $\rho_{\mathfrak{\mathfrak{r}},\mathfrak{\mathfrak{r}}'}:\partial X\to\SO\left(3\right)$.
$\mathfrak{\mathfrak{r}}$ and $\mathfrak{\mathfrak{r}}'$ are equivalent
if and only if $\rho_{\mathfrak{\mathfrak{r}},\mathfrak{\mathfrak{r}}'}$
\emph{lifts} to a map $\tilde{\rho}_{\mathfrak{\mathfrak{r}},\mathfrak{\mathfrak{r}}'}:\partial X\to\SU\left(2\right)$
which \emph{extends} to a map $X\to\SU\left(2\right)$. Both properties
are governed by topological invariants:
\begin{enumerate}
\item $\rho_{\mathfrak{\mathfrak{r}},\mathfrak{\mathfrak{r}}'}:\partial X\to\SO\left(3\right)$
lifts to $\SU\left(2\right)$ if and only if $w\left(\rho_{\mathfrak{\mathfrak{r}},\mathfrak{\mathfrak{r}}'}\right)=0$,
where $w\left(\rho_{\mathfrak{\mathfrak{r}},\mathfrak{\mathfrak{r}}'}\right)\in H^{1}\left(\partial X;\mathbb{Z}_{2}\right)$
is the pull-back via $\rho_{\mathfrak{\mathfrak{r}},\mathfrak{\mathfrak{r}}'}$
of the generator of $H^{1}\left(\SO\left(3\right);\mathbb{Z}_{2}\right)=\mathbb{Z}_{2}$;
\item $\tilde{\rho}_{\mathfrak{\mathfrak{r}},\mathfrak{\mathfrak{r}}'}:\partial X\to\SU\left(2\right)$
extends to $X$ if and only if $\deg\left(\tilde{\rho}_{\mathfrak{\mathfrak{r}},\mathfrak{\mathfrak{r}}'}\right)=0$.
\end{enumerate}
The two invariants are clearly independent of the choice of the auxiliary
$\SU\left(2\right)$ frame of $E_{|\partial X}$. Moreover, note that
the first condition is equivalent to the fact that, for every $h_{0}\in\mathfrak{c}_{\infty}\left(g\right)$,
the induced spin structures $\left(\mathbb{S},\cl_{h_{0}}\right)$
and $\left(\mathbb{S},\cl_{h_{0}}'\right)$ are equivalent. Finally,
note that \emph{every} isometry of $\SO\left(3\right)$ vector bundles
$\mathfrak{su}\left(E_{|\partial X}\right)\to\mathfrak{su}\left(E_{|\partial X}\right)$
can be realized as $\rho_{\mathfrak{\mathfrak{r}},\mathfrak{\mathfrak{r}}'}$
for some pair of Nahm pole residues, and that the two invariants $w\left(\rho_{\mathfrak{\mathfrak{r}},\mathfrak{\mathfrak{r}}'}\right)$
and $\deg\left(\rho_{\mathfrak{\mathfrak{r}},\mathfrak{\mathfrak{r}}'}\right)$
are related by the equation $\deg\left(\rho_{\mathfrak{\mathfrak{r}},\mathfrak{\mathfrak{r}}'}\right)\equiv w\left(\rho_{\mathfrak{\mathfrak{r}},\mathfrak{\mathfrak{r}}'}\right)^{3}\mod\mathbb{Z}_{2}$.
Therefore, we obtain the following
\begin{thm}
A choice of a reference Nahm pole residue determines a 1-1 correspondence
between equivalence classes of Nahm pole residues and pairs $\left(u,d\right)\in H^{1}\left(\partial X;\mathbb{Z}_{2}\right)\times\mathbb{Z}$
such that $d\equiv u^{3}\mod\mathbb{Z}_{2}$.
\end{thm}

\subsection{\label{subsec:Comparison-with-KW}Comparison with Nahm poles for
the KW equations}

We now compare the boundary condition presented here, with the Nahm
pole boundary condition introduced in \cite{WittenFivebranes} (cf.
also \cite{MazzeoWittenNahmI} for a more mathematically oriented
exposition). The \emph{Kapustin--Witten} equations are gauge-theoretic
equations for a pair $\left(A,\phi\right)$ consisting of an $\SU\left(2\right)$
connection $A$ over an $\SU\left(2\right)$ vector bundle $E$, and
and an $\mathfrak{su}\left(E\right)$-valued $1$-form $\phi$, over
an oriented Riemannian 4-manifold. The equations are
\begin{align*}
F_{A}-\phi\land\phi+*d_{A}\phi & =0\\
d_{A}*\phi & =0.
\end{align*}
In \cite{MazzeoWittenNahmI}, Mazzeo and Witten consider these equations
over a compact, oriented Riemannian 4-manifold with boundary $\left(X,g\right)$.
However, while $A$ is a standard connection, $\phi$ is allowed to
have a pole along the boundary. Specifically, if $x$ is the distance
from the boundary, then as $x\to0$ $\phi$ has an asymptotic expansion
\[
\phi\sim x^{-1}\sum_{i=1}^{3}\theta^{i}\mathfrak{t}_{i}+o\left(x^{-1}\right)
\]
where the $\theta^{i}$ form an $\SO\left(3\right)$ coframe of $\partial X$,
and the $\mathfrak{t}_{i}$ form a Nahm triple of elements in $\mathfrak{su}\left(2\right)$.

These conditions can be rephrased in terms of $0$-connections. Indeed,
writing $\mathcal{A}=A+i\phi$, $\mathcal{A}$ is precisely an $\SL\left(2,\mathbb{C}\right)$
$0$-connection on $E$, and the Nahm pole boundary condition says
that the residue of $\mathcal{A}$, seen as the map $\res_{\mathcal{A}}:TX_{|\partial X}\to\mathfrak{su}\left(E_{|\partial X}\right)\oplus i\mathfrak{su}\left(E_{|\partial X}\right)$
defined as $\res_{\mathcal{A}}=\left(\mathcal{A}-D\right)_{|\partial X}$
for some smooth connection $D$, satisfies the following three properties:
\begin{enumerate}
\item $\Re\left(\res_{\mathcal{A}}\right)$ vanishes;
\item $\Im\left(\res_{\mathcal{A}}\right)$ vanishes along the orthogonal
complement of $T\partial X$ in $TX_{|\partial X}$;
\item the restriction of $\Im\left(\res_{\mathcal{A}}\right)$ to $T\partial X$
defines an isomorphism of Lie algebra bundles $T\partial X\to\mathfrak{su}\left(E_{|\partial X}\right)$.
\end{enumerate}

\subsection{Examples}

We now provide examples of $\SU\left(2\right)$ $0$-connections satisfying
the Nahm pole boundary condition. First, we define the $0$-version
of a spin structure:
\begin{defn}
A \emph{$0$-spin structure} on an oriented conformally compact manifold
$\left(X^{n+1},g\right)$ is a principal $\Spin\left(n+1\right)$
bundle $P$ together with an equivariant bundle map $P\to\Fr_{\SO\left(n+1\right)}\left(^{0}TX,g\right)$
which is fibrewise the universal cover $\Spin\left(n+1\right)\to\SO\left(n+1\right)$.
\end{defn}

On a 4-dimensional conformally compact manifold $\left(X^{4},g\right)$,
a $0$-spin structure can be concretely described in terms of the
following data:
\begin{enumerate}
\item two $\SU\left(2\right)$ vector bundles $\mathbb{S}_{\pm}$;
\item a \emph{$0$-Clifford multiplication} map
\[
^{0}\cl:{^{0}T^{*}X}\to\mathfrak{u}\left(\mathbb{S}_{+}\oplus\mathbb{S}_{-}\right)
\]
which maps $\mathbb{S}_{\pm}$ to $\mathbb{S}_{\mp}$, and satisfies
\[
^{0}\cl\left(\alpha\right)^{2}=-\left|\alpha\right|^{2}
\]
for every $\alpha\in{^{0}T^{*}X}$.
\end{enumerate}
Given this data, the principal bundle $P$ is the bundle of $\Spin\left(4\right)=\SU\left(2\right)\times\SU\left(2\right)$
frames of $\mathbb{S}=\mathbb{S}_{+}\oplus\mathbb{S}_{-}$, and the
map $P\to\Fr_{\SO\left(4\right)}\left(^{0}TX,g\right)$ is the one
which, to a $\Spin\left(4\right)$ frame $s_{1},s_{2},\tilde{s}_{1},\tilde{s}_{2}$
of $\mathbb{S}$, associates the unique $\SO\left(4\right)$ coframe
$e^{0},e^{1},e^{2},e^{3}$ (here we identify $^{0}TX\simeq{^{0}T^{*}X}$
using $g$) for which
\begin{align}
^{0}\cl\left(e^{0}\right) & =\left(\begin{matrix} & \boldsymbol{1}\\
-\boldsymbol{1}
\end{matrix}\right)\nonumber \\
^{0}\cl\left(e^{i}\right) & =\left(\begin{matrix} & \sigma_{i}\\
\sigma_{i}
\end{matrix}\right),\label{eq:clifford-dim4}
\end{align}
where $\sigma_{1},\sigma_{2},\sigma_{3}\in\mathfrak{su}\left(2\right)$
are defined as in §\ref{subsec:Nahm-poles-and-self-duality}.

As for any associated vector bundle, the spinor bundles $\mathbb{S}_{\mp}$
inherit from $g$ a natural $0$-connection, which we call the Levi-Civita
spin $0$-connection. We will now compute this $0$-connection explicitly
on $\HH^{4}$. As in Example \ref{exa:LC-hyperbolic}, we use half-space
coordinates $\left(x,y^{1},y^{2},y^{3}\right)$ with $x\geq0$, and
metric
\[
g_{\hyp}=\frac{dx^{2}+dy^{2}}{x^{2}}.
\]
Define $e^{0}=-dx/x$ and $e^{i}=dy^{i}/x$, so that $\left(e^{0},e^{1},e^{2},e^{3}\right)$
is an $\SO\left(4\right)$ global frame of $^{0}T^{*}\HH^{4}$, in
accordance with the orientation convention used in Example \ref{exa:LC-hyperbolic}.
The essentially unique $0$-spin structure on $\HH^{4}$ is realized
in this model by the two trivial $\SU\left(2\right)$ bundles $\mathbb{S}_{\pm}\equiv\HH^{4}\times\mathbb{C}^{2}$
and the $0$-Clifford multiplication operator defined\emph{ }as in
Formula (\ref{eq:clifford-dim4}). The spin $0$-connections on $\mathbb{S}_{+}$
and $\mathbb{S}_{-}$ are uniquely determined by the fact that they
make $^{0}\cl$ parallel. An easy computation then shows that
\begin{align*}
\nabla^{\pm} & =d\pm e^{i}\otimes\frac{\sigma_{i}}{2}.
\end{align*}
Denote by $e_{0}=-x\partial_{x}$ and $e_{i}=x\partial_{y^{i}}$ the
$g$ duals of the $0$-$1$-forms $e^{0},e^{i}$ defined above; then
$\res_{\nabla^{+}}$ is annihilated by $x\partial_{x}$ and sends
$e_{i|\partial\mathrm{H}^{4}}$ to $\frac{\sigma_{i}}{2}_{|\partial\mathrm{H}^{4}}$;
it follows that $\nabla^{+}$ satisfies the Nahm pole boundary condition.
This is in fact true for the Levi-Civita spin $0$-connection of \emph{every
}asymptotically hyperbolic spin 4-manifold:
\begin{prop}
Let $\left(X^{4},g\right)$ be an oriented asymptotically hyperbolic
4-manifold equipped with a $0$-spin structure. Then the spin $0$-connection
on $\mathbb{S}_{+}$ satisfies the Nahm pole boundary condition.
\end{prop}

\begin{proof}
Let $p\in\partial X$. Then a choice of a $0$-spin structure $\left(\mathbb{S}_{+},\mathbb{S}_{-},{^{0}\cl}\right)$
on $\left(X,g\right)$ induces a natural $0$-spin structure on the
model hyperbolic space $\left(X_{p},g_{p}\right)$. More precisely,
the spinor bundles are just the trivial bundles $X_{p}\times\mathbb{S}_{\pm|p}$,
and the $0$-Clifford multiplication is just $^{0}\cl_{p}$. The Levi-Civita
spin $0$-connection $\nabla^{\pm}$ on $\mathbb{S}_{\pm}$ has a
residual $0$-connection $\nabla_{p}^{\pm}$, which is a spin $0$-connection
on $X_{p}\times\mathbb{S}_{\pm|p}$. Exactly as for the Levi-Civita
connections on tensor bundles (cf. Example \ref{exa:LC-hyperbolic}),
$\nabla_{p}^{\pm}$ is precisely the Levi-Civita spin $0$-connection
relative to the $0$-spin structure $\left(X_{p}\times\mathbb{S}_{\pm|p},{}^{0}\cl_{p}\right)$.
Since $g$ is asymptotically hyperbolic, $\left(X_{p},g_{p}\right)$
is an isometric copy of $\HH^{4}$, so the computation above implies
that $\nabla_{p}^{+}$ satisfies the Nahm pole boundary condition.
Since this is true for every $p\in\partial X$, it follows that $\nabla^{+}$
satisfies the Nahm pole boundary condition, as claimed.
\end{proof}
This proposition is not surprising. Indeed, recall the following well-known
fact of 4-dimensional geometry, explained for example at the beginning
of \cite{AtiyahHitchinSingerSelfDuality}:
\begin{thm}
On a 4-dimensional oriented Riemannian manifold with a spin structure,
the metric is Einstein if and only if the spin connection on $\mathbb{S}_{+}$
(resp. $\mathbb{S}_{-}$) is self-dual (resp. anti-self-dual).
\end{thm}

Applying this theorem to our case, we obtain the following
\begin{cor}
\label{cor:LC-0-conn-S+-of-PE-is-SD}Let $\left(X^{4},g\right)$ be
a Poincaré--Einstein metric equipped with a $0$-spin structure $\left(\mathbb{S}_{\pm},{^{0}\cl}\right)$.
Then the spin $0$-connection on $\mathbb{S}_{+}$ is self-dual and
satisfies the Nahm pole boundary condition.
\end{cor}

\section{\label{sec:Asymptotic-expansion-of-self-dual-0-connections}Asymptotic
expansion of $0$-instantons}

In this section, inspired by the Fefferman--Graham expansion for
Poincaré--Einstein manifolds (cf. §\ref{subsec:Poincar=0000E9=002013Einstein-metrics}),
we compute the asymptotic expansion for self-dual $0$-connections
satisfying the Nahm pole boundary condition. So far we considered
only smooth $0$-connections; however, as in the Poincaré--Einstein
case, it is more appropriate to work with \emph{polyhomogeneous} $0$-connections
(cf. §\ref{subsec:Polyhomogeneity}).
\begin{defn}
\label{def:phg-0conn}Let $X^{n+1}$ be a compact manifold with boundary,
and let $E\to X$ be a vector bundle. A \emph{polyhomogeneous $0$-connection}
is a $0$-connection on $E$ of the form $A=D+a$, where $D$ is a
smooth $0$-connection, and $a\in\mathcal{A}_{\phg}^{\mathcal{E}}\left(X;{^{0}\Lambda^{1}}\otimes\mathfrak{gl}\left(E\right)\right)$
with index set $\mathcal{E}$ of the form $\mathcal{E}=\mathbb{N}\cup\mathcal{I}$
with $\Re\left(\mathcal{I}\right)>0$. We say that $A$ is \emph{log-smooth
of order $k$} if $a$ is log-smooth of order $k$ (cf. Definition
\ref{def:log-smoothness}).
\end{defn}

\begin{rem}
Note that, if $A$ is a polyhomogeneous $0$-connection in the sense
of the previous definition, then the residue $\res_{A}$ is again
a well-defined smooth bundle map $\res_{A}:{^{0}TX_{|\partial X}}\to\mathfrak{gl}\left(E_{|\partial X}\right)$.
Indeed, we can write $A=D+a$ with $D$ smooth $0$-connection and
$a\in\mathcal{A}_{\phg}^{\mathcal{I}}\left(X;{^{0}\Lambda^{1}}\otimes\mathfrak{gl}\left(E\right)\right)$
for some index set $\mathcal{I}$ with $\Re\left(\mathcal{I}\right)>0$;
since $a$ vanishes on the boundary, it does not affect the residue,
which is simply equal to $\res_{D}$.
\end{rem}

As in the Poincaré--Einstein case, working with polyhomogeneous $0$-connections
in our context is both natural and necessary. Indeed, as we shall
see more in detail in the upcoming paper \cite{UsulaNahmPolesII},
the self-duality equation is an example of a nonlinear \emph{$0$-elliptic}
equation (modulo gauge); solutions of equations of this type, with
smooth ``boundary values'' are \emph{expected }to be polyhomogeneous,
but they are rarely smooth. We refer to \cite{ChruscielDelayLeeSkinnerBoundaryRegularity, MazzeoYamabeRegularity}
for two examples of this phenomenon (cf. Remark \ref{rem:regularity-PE})

\subsection{\label{subsec:A-normal-form-for-0connections-adapted-to-g-at-infinity}A
geodesic normal form for $0$-connections}

In §\ref{subsec:Poincar=0000E9=002013Einstein-metrics}, we saw that
a fundamental preliminary ingredient in order to compute the Fefferman--Graham
expansion, is the existence of a normal form for asymptotically hyperbolic
metrics. In this subsection, we establish the existence of a normal
form for $\SU\left(2\right)$ $0$-connections adapted to $g$ at
infinity. First, we need the following technical lemma:
\begin{lem}
Let $Y$ be a closed manifold. Consider the cylinder $\left[0,1\right]\times Y$,
and call $x$ the coordinate on the interval factor. Let $c$ be a
$\mathbb{R}^{N\times N}$-valued polyhomogeneous function on $\left[0,1\right]\times Y$,
smooth across the face $x=1$ and $O\left(x^{\delta}\right)$ polyhomogeneous
across the face $x=0$, for some $\delta>0$. Finally, let $u_{0}\in C^{\infty}\left(Y;\mathbb{R}^{N}\right)$.
Then the system of ODEs for a $\mathbb{R}^{N}$-valued function $u$
on
\[
\begin{cases}
x\partial_{x}u+cu & =0\\
u\left(0,y\right) & =u_{0}\left(y\right)
\end{cases}
\]
admits a unique solution $u$ in $\left[0,1\right]\times Y$, smooth
across the face $x=1$ and $O\left(1\right)$ polyhomogeneous across
the face $x=0$. Moreover, if $c$ is smooth in $\left[0,1\right]\times Y$,
then the solution $u$ is smooth as well.
\end{lem}

\begin{proof}
Define $a=-x^{-1}c$ and write the equation as $\dot{u}=au$. Define
now the sequence
\begin{align*}
\tilde{u}_{0} & =u_{0}\\
\tilde{u}_{k+1} & =\int_{0}^{x}a\left(s\right)\tilde{u}_{k}\left(s\right)ds.
\end{align*}
Since $a\left(x\right)$ is $O\left(x^{\delta-1}\right)$ polyhomogeneous
and $\delta>0$, and $u_{0}$ is smooth, we inductively have that
$\tilde{u}_{k}$ is $O\left(x^{k\delta}\right)$ polyhomogeneous.
In particular, if $c$ is smooth, then the condition that $c=o\left(1\right)$
implies that $a$ is smooth as well, and therefore $\tilde{u}_{k}$
is smooth and $O\left(x^{k}\right)$. Now, define
\[
u_{k}=\sum_{j=0}^{k}\tilde{u}_{j}.
\]
Then inductively we have
\[
u_{k+1}\left(x\right)=u_{0}+\int_{0}^{x}a\left(s\right)u_{k}\left(s\right)ds.
\]
By a standard fixed point argument, the sequence $u_{k}$ converges
to a limit $u$ in $C^{0}\left(\left[0,1\right]\times Y\right)$.
Similarly, for every fixed $k$, the sequence $u_{h}-u_{k}$ converges
to $u-u_{k}$ in $x^{k\delta}C^{0}\left(\left[0,1\right]\times Y\right)$,
and the same is true for any partial derivative $\left(x\partial_{x}\right)^{j}V_{1}\cdots V_{s}\left(u_{h}-u_{k}\right)$,
where $V_{1},...,V_{s}$ are vector fields in $Y$. It follows that
the limit $u$ is indeed $O\left(1\right)$ polyhomogeneous, with
expansion
\[
u\sim\sum_{k=0}^{\infty}\tilde{u}_{k}.
\]
Observe that if $c$ is smooth, the fact that $c=o\left(1\right)$
implies that actually $c=O\left(x\right)$, and therefore $a=-x^{-1}c$
is smooth as well. Therefore, $\tilde{u}_{k}$ is smooth and $O\left(x^{k}\right)$,
so $u$ is smooth.
\end{proof}
This lemma allows us to interpret geometrically the condition of being
adapted to $g$ at infinity. Fix an asymptotically hyperbolic manifold
$\left(X^{n+1},g\right)$, and let $E\to X$ be a rank $N$ real vector
bundle on $X$. Choose a metric $h_{0}\in\mathfrak{c}_{\infty}\left(g\right)$,
a geodesic boundary defining function $x$ for $X$ inducing $h_{0}$,
and consider the corresponding geodesic normal form
\[
g=\frac{dx^{2}+h\left(x\right)}{x^{2}}
\]
on a geodesic collar $[0,\varepsilon)\times\partial X\hookrightarrow X$.
Moreover, denote by $\partial_{x}$ the $x^{2}g$-gradient of $x$.
\begin{prop}
\label{prop:parallel-transport}Let $A$ be a $0$-connection on $E$
adapted to $g$ at infinity. Then the parallel transport of $A$ along
the flow of $\partial_{x}$ starting at $\partial X$ is well-defined
and polyhomogeneous in a neighborhood of $\partial X$. More precisely,
for every $u_{0}\in C^{\infty}\left(\partial X;E_{|\partial X}\right)$,
there exists a (unique near the boundary) $O\left(1\right)$ polyhomogeneous
section $u$ of $E$ such that $A_{x\partial_{x}}u\equiv0$ near $\partial X$
and $u_{|\partial X}\equiv u_{0}$. Moreover, if $A$ is smooth, then
$u$ is smooth as well.
\end{prop}

\begin{proof}
Write the geodesic normal form of $g$ induced by $h_{0}$ as
\[
g=\frac{dx^{2}+h\left(x\right)}{x^{2}}.
\]
For simplicity, assume that $E$ is trivial and choose a smooth frame
for $E$ defined on the collar. In this frame, we can write the $0$-connection
$A$ as
\[
A=d+b+c\frac{dx}{x},
\]
where $b\left(x\right)$ is a $O\left(x^{-1}\right)$ polyhomogeneous
family of $\mathfrak{gl}\left(N\right)$-valued $1$-forms on $\partial X$,
and $c\left(x\right)$ is a $O\left(1\right)$ polyhomogeneous family
of $\mathfrak{gl}\left(N\right)$-valued functions on $\partial X$.
Writing $u$ in terms of this frame as a $\mathbb{R}^{N}$-valued
function on $\left[0,\varepsilon\right]\times\partial X$, the conditions
$A_{x\partial_{x}}u=0$ and $u_{|\partial X}=0$ become the system
of ODEs
\[
\begin{cases}
x\partial_{x}u+cu & =0\\
u\left(0\right) & =u_{0}
\end{cases}.
\]
Since $A$ is adapted to $g$ at infinity, we have $c\left(0\right)\equiv0$;
therefore, by polyhomogeneity, we actually have $c\left(x\right)=O\left(x^{\delta}\right)$
for some $\delta>0$. Therefore, by the previous lemma, this system
admits a unique polyhomogeneous solution on $\left[0,\varepsilon\right]\times\partial X$,
smooth at $x=\varepsilon$ and $O\left(1\right)$ polyhomogeneous
at $x=0$. Moreover, if $A$ is smooth, then $c\left(x\right)$ is
smooth as well, so by the previous lemma the unique solution is smooth.
\end{proof}
We can now define the normal form for $0$-connections adapted to
$g$ at infinity:
\begin{defn}
\label{def:geodesic-normal-family}Let $A$ be a $0$-connection on
$E$ adapted to $g$ at infinity. Let $h_{0}\in\mathfrak{c}_{\infty}\left(g\right)$,
and consider the associated geodesic collar $[0,\varepsilon)\times\partial X\to X$.
The \emph{geodesic normal family of $A$ associated to $h_{0}$} is
the family $\alpha_{h_{0}}\left(x\right)$ of connections on $E_{|\partial X}$
defined for $x\in\left(0,\varepsilon\right)$ as follows: if $t\in\left(0,\varepsilon\right)$,
\[
\alpha_{h_{0}}\left(t\right):=\left(T_{\partial_{x},t}^{A}\right)^{*}\left(A_{|x=t}\right),
\]
where $A_{|x=t}$ is the restriction of $A$ to the bundle $E_{|x=t}$,
and $T_{\partial_{x},t}^{A}:E_{|x=0}\to E_{|x=t}$ is the parallel
transport of $A$ along the integral curves of $\partial_{x}$ from
$x=0$ to $x=t$.
\end{defn}

\begin{rem}
\label{rem:geodesic-normal-family-in-coords}The geodesic normal family
$\alpha_{h_{0}}\left(x\right)$ defined as above is in fact a $O\left(x^{-1}\right)$
polyhomogeneous family of connections on $[0,\varepsilon)$, of the
form
\[
\alpha_{h_{0}}\left(x\right)\sim\mathfrak{r}_{h_{0}}x^{-1}+\alpha_{0}+\sum_{\begin{smallmatrix}\left(\gamma,l\right)\in\mathcal{E}\\
\left(\gamma,l\right)\not=0
\end{smallmatrix}}\alpha_{\gamma,l}x^{\gamma}\left(\log x\right)^{l}.
\]
Here:
\begin{enumerate}
\item $\mathfrak{r}_{h_{0}}:T\partial X\to\mathfrak{gl}\left(E_{|\partial X}\right)$
is the tensor induced by the restriction of $\res_{A}:{^{0}TX}_{|\partial X}\to\mathfrak{gl}\left(E_{|\partial X}\right)$
to $K=\ker\left(\left(dx/x\right)_{|\partial X}\right)$, precomposed
by the isomorphism $T\partial X\to K$ induced by $h_{0}$ (cf. Remark
\ref{rem:non-canonical-isomorphism-K-TbX});
\item $\alpha_{0}$ is a connection on $E_{|\partial X}$;
\item the other coefficients $\alpha_{\gamma,l}$ are $\mathfrak{gl}\left(E_{|\partial X}\right)$-valued
$1$-forms on $\partial X$.
\end{enumerate}
We can compute the expansion explicitly in terms of a local frame
$s_{1},...,s_{N}$ of $E_{|\partial X}$. By Proposition \ref{prop:parallel-transport},
we can extend $s_{1},...,s_{N}$ uniquely to a local frame of $E$
in the open subset of the collar $[0,\varepsilon)\times U$, in such
a way that $A_{x\partial_{x}}s_{i}\equiv0$. In this local frame,
we have
\[
A=d+a\left(x\right)
\]
where $a\left(x\right)$ is a $O\left(x^{-1}\right)$ polyhomogeneous
family of $\mathfrak{su}\left(E_{|\partial X}\right)$-valued $1$-forms
on $\partial X$. The connection $\alpha_{h_{0}}\left(x\right)$ in
the frame $s_{1},...,s_{N}$ is simply
\[
\alpha_{h_{0}}\left(x\right)=d_{\partial X}+a\left(x\right).
\]
Note that, again from Proposition \ref{prop:parallel-transport},
if $A$ is smooth then the expansion of $\alpha_{h_{0}}\left(x\right)$
contains only $x^{k}$ terms with $k\in\mathbb{Z}$, $k\geq-1$.
\end{rem}

Directly from the definition, we see that the geodesic normal family
is a fundamental\emph{ gauge invariant} of $A$, modulo gauge transformations
equal to the identity along the boundary:
\begin{prop}
\label{prop:geodesic-normal-family-is-gauge-invariant}Let $\left(X,g\right)$,
$E\to X$ and $A$ be as in Proposition \ref{prop:parallel-transport}.
Let $\Phi:E\to E$ be a $O\left(1\right)$ polyhomogeneous gauge transformation
of $E$. Assume furthermore that $\Phi_{|\partial X}=1$. Then $A$
and $\Phi^{*}A$ have the same geodesic normal family with respect
to any $h_{0}\in\mathfrak{c}_{\infty}\left(g\right)$.
\end{prop}

\begin{proof}
Fix $h_{0}\in\mathfrak{c}_{\infty}\left(g\right)$, and work on the
associated geodesic collar $[0,\varepsilon)\times\partial X\to X$
of $\left(X,g\right)$. To avoid notational problems, call $B=\Phi^{*}A$
and denote by $\alpha\left(x\right)$ and $\beta\left(x\right)$ the
geodesic normal families of $A$ and $B$ associated to $h_{0}$.
We need to prove that $\alpha=\beta$. Let $\gamma:\left[0,1\right]\to\partial X$
be a smooth curve. Given $t\in\left(0,\varepsilon\right)$, denote
by $T_{\gamma}^{\alpha\left(t\right)},T_{\gamma}^{\beta\left(t\right)}:E_{\gamma\left(0\right)}\to E_{\gamma\left(1\right)}$
the parallel transports of $\alpha\left(t\right)$ and $\beta\left(t\right)$
along $\gamma$. Since a connection is determined by its parallel
transports, it suffices (by the generality of $t$ and $\gamma$)
to prove that $T_{\gamma}^{\alpha\left(t\right)}=T_{\gamma}^{\beta\left(t\right)}$.
By definition of $\alpha$ and $\beta$, we have
\begin{align*}
T_{\gamma}^{\alpha\left(t\right)} & =\left(T_{\partial_{x},t}^{A}\right)^{-1}T_{\left\{ t\right\} \times\gamma}^{A}T_{\partial_{x},t}^{A}\\
T_{\gamma}^{\beta\left(t\right)} & =\left(T_{\partial_{x},t}^{B}\right)^{-1}T_{\left\{ t\right\} \times\gamma}^{B}T_{\partial_{x},t}^{B}.
\end{align*}
On the other hand, since $B=\Phi^{*}A$ and $\Phi_{|\partial X}=1$,
we have $T_{\partial_{x},t}^{A}=\Phi_{|x=t}T_{\partial_{x},t}^{B}$
and $T_{\left\{ t\right\} \times\gamma}^{A}\Phi_{\gamma\left(0\right)}=\Phi_{\gamma\left(1\right)}T_{\left\{ t\right\} \times\gamma}^{B}$.
Substituting in the previous equations and simplifying, we obtain
$T_{\gamma}^{\beta\left(t\right)}=T_{\gamma}^{\alpha\left(t\right)}$
as claimed.
\end{proof}

\subsection{\label{subsec:Asymptotic-expansions-of-self-dual-0-connections}Asymptotic
expansions of $0$-instantons}

We now come back to the four-dimensional picture. Consider an asymptotically
hyperbolic 4-manifold $\left(X^{4},g\right)$, an $\SU\left(2\right)$
vector bundle $E\to X$, and a polyhomogeneous $0$-connection $A$
on $E$ satisfying the Nahm pole boundary condition, with Nahm pole
residue $\mathfrak{r}$. Choose $h_{0}\in\mathfrak{c}_{\infty}\left(g\right)$
and consider the geodesic normal form for $g$ in a geodesic collar
$[0,\varepsilon)\times\partial X\hookrightarrow X$,
\begin{equation}
g=\frac{dx^{2}+h\left(x\right)}{x^{2}}.\label{eq: geodesic normal form}
\end{equation}
The geodesic normal family $\alpha_{h_{0}}\left(x\right)$ of $A$
associated to $h_{0}$ is then a family of $\SU\left(2\right)$ connections
on $\mathbb{S}=E_{|\partial X}$, with expansion of the form
\[
\alpha_{h_{0}}\left(x\right)\sim\mathfrak{r}_{h_{0}}x^{-1}+\alpha_{0}+\sum_{\begin{smallmatrix}\left(\gamma,l\right)\in\mathcal{E}\\
\left(\gamma,l\right)\not=0
\end{smallmatrix}}\alpha_{\gamma,l}x^{\gamma}\left(\log x\right)^{l}.
\]
In this subsection, we show that if $A$ satisfies the self-duality
equation, then the connection $\alpha_{0}$ and the $\mathfrak{su}\left(\mathbb{S}\right)$-valued
$1$-forms $\alpha_{\gamma,l}$ satisfy various differential relations,
and the index set $\mathcal{E}$ is greatly restricted, exactly as
in the Poincaré--Einstein case.

We rename $\theta:=\mathfrak{r}_{h_{0}}$ for ease of notation. The
$\mathfrak{su}\left(\mathbb{S}\right)$-valued $1$-form $\theta$
is canonically associated to the spin 3-manifold $\left(\partial X,h_{0},\mathbb{S},\cl_{h_{0}}\right)$:
it is the \emph{soldering form} of the spin 3-manifold (cf. §\ref{subsec:Geometry-of-spin-3D}).
Concretely, given an $\SU\left(2\right)$ frame $s_{1},s_{2}$ of
$\mathbb{S}$, there exists a unique $\SO\left(3\right)$ co-frame
$\theta^{1},\theta^{2},\theta^{3}$ of $T^{*}\partial X$ such that
$\cl_{h_{0}}\left(\theta^{i}\right)=\sigma_{i}$ with $\sigma_{1},\sigma_{2},\sigma_{3}\in\mathfrak{su}\left(2\right)$
defined as in §\ref{subsec:A-normal-form-for-0connections-adapted-to-g-at-infinity};
in this frame, the soldering form is
\[
\theta=\theta^{i}\otimes\frac{\sigma_{i}}{2}.
\]
As explained in §\ref{subsec:Geometry-of-spin-3D}, the soldering
form determines a canonical isomorphism
\begin{align*}
T^{*}\partial X\otimes\mathfrak{su}\left(\mathbb{S}\right) & \simeq\mathfrak{gl}\left(T\partial X\right).\\
\gamma & \mapsto\theta^{-1}\circ\gamma.
\end{align*}
Under this identification we can talk about the adjoint $\gamma^{*}$,
the skew-adjoint and self-adjoint parts $\sk\gamma$ and $\symm\gamma$,
and the trace $\tr\gamma$. We also denote by $\symmtf\gamma$ the
self-adjoint trace-free part of $\gamma$. We denote by $\star=\star\left(x\right)$
the $x$-dependent Hodge star associated to $h\left(x\right)$, seen
as a family of bundle maps from $\Lambda^{1}$ to $\Lambda^{2}$ of
$\partial X$ (not to be confused with $*$, the 4-dimensional Hodge
star on $\left(X,g\right)$). We can expand $\star$ in power series
as $x\to0$:
\[
\star\left(x\right)\sim\sum_{k=0}^{\infty}\star_{k}x^{k},
\]
where $\star_{0}$ is the Hodge star of $h_{0}$ and the coefficients
$\star_{k}$ are bundle maps $\Lambda^{s}\to\Lambda^{3-s}$ on $\partial X$.
\begin{lem}
In the geodesic collar induced by $h_{0}$, the equation $F_{A}=*F_{A}$
is equivalent to the evolution equation $\partial_{x}\alpha_{h_{0}}=-\star f_{\alpha_{h_{0}}}$
for $\alpha_{h_{0}}\left(x\right)$, where $f_{\alpha_{h_{0}}}$ is
the curvature of $\alpha_{h_{0}}$.
\end{lem}

\begin{proof}
Write $A=d+a$ with respect to an $\SU\left(2\right)$ frame $s_{1},s_{2}$
of $E$ such that $A_{x\partial_{x}}s_{i}\equiv0$. The curvature
of $A$ is then
\begin{align*}
F_{A} & =da+a\land a\\
 & =d_{\partial X}a+a\land a+dx\land\partial_{x}a
\end{align*}
Since $a=a\left(x\right)$ is a family of $\mathfrak{su}\left(2\right)$-valued
$1$-forms on $\partial X$, and because $g$ is in normal form in
the given collar, an elementary computation yields
\[
*F_{A}=-dx\land\star\left(d_{\partial X}a+a\land a\right)-\star\partial_{x}a.
\]
Therefore, the self-duality equation is equivalent to $\partial_{x}a=-\star\left(d_{\partial X}a+a\land a\right)$,
which invariantly reads $\partial_{x}\alpha_{h_{0}}=-\star f_{\alpha_{h_{0}}}$
as claimed.
\end{proof}
\begin{thm}
\label{thm:index-set-0-instantons}Let $A$ be a self-dual polyhomogeneous
$0$-connection satisfying the Nahm pole boundary condition. Given
a choice of $h_{0}\in\mathfrak{c}_{\infty}\left(g\right)$, let $x$
be a geodesic boundary defining function inducing $h_{0}$, and let
$\alpha_{h_{0}}\left(x\right)$ be the geodesic normal family of $A$
associated to $h_{0}$. Then the expansion of $\alpha\left(x\right)$
as $x\to0$ takes the form
\begin{equation}
\alpha_{h_{0}}\left(x\right)\sim\theta x^{-1}+\alpha_{0}+\sum_{k=1}^{\infty}\sum_{l=0}^{k}\left(x^{2k-1}\left(\log x\right)^{l}\alpha_{2k-1,l}+x^{2k}\left(\log x\right)^{l}\alpha_{2k,l}\right).\label{eq:expansion}
\end{equation}
In other words, the expansion of $\alpha\left(x\right)$ can contain
only nonzero terms of the form $x^{k}\left(\log x\right)^{l}$, with
$k\in\mathbb{Z},k\geq-1$ and $l\leq\left(k+1\right)/2$. Moreover,
the expansion is smooth (i.e. it does not contain log terms) if and
only if $\alpha_{1,1}=0$.
\end{thm}


\begin{proof}
Write the expansion of $\alpha$ as
\[
\alpha\left(x\right)\sim\theta x^{-1}+\alpha_{0}+\sum_{\begin{smallmatrix}\left(\gamma,l\right)\in\mathcal{E}\\
\left(\gamma,l\right)\not=0
\end{smallmatrix}}\alpha_{\gamma,l}x^{\gamma}\left(\log x\right)^{l}.
\]
Here $\alpha_{0}$ is an $\SU\left(2\right)$ connection on $\mathbb{S}$,
the coefficients $\alpha_{\gamma,l}$ with $\left(\gamma,l\right)\not=0$
are sections of $T^{*}\partial X\otimes\mathfrak{su}\left(\mathbb{S}\right)$,
and the index set $\mathcal{E}$ satisfies $\Re\left(\mathcal{E}\right)>-1$.
Let us put a partial order on $\mathcal{E}$: we say that $\left(\gamma_{1},l_{1}\right)<\left(\gamma_{2},l_{2}\right)$
if either $\Re\left(\gamma_{1}\right)<\Re\left(\gamma_{2}\right)$
or $\gamma_{1}=\gamma_{2}$ and $l_{1}>l_{2}$. Let $\mathcal{S}$
be the set of pairs $\left(\gamma,l\right)\in\mathcal{E}$ for which
$\alpha_{\gamma,l}\not=0$ and either $\gamma\not\in\mathbb{N}$ or
$\gamma\in\mathbb{N}$ and $l>\left(\gamma+1\right)/2$. Our aim is
to show that $\mathcal{S}$ must be empty. Suppose for the sake of
contradiction that $\mathcal{S}\not=\emptyset$: then, since clearly
$\mathcal{S}$ cannot have an infinite descending chain, $\mathcal{S}$
must have a minimal element $\left(\gamma,l\right)$. Consider again
the equation $\partial_{x}\alpha=-\star f_{\alpha}$, and let's focus
on the coefficient of $x^{\gamma-1}\left(\log x\right)^{l}$ on the
left and on the right. Since $\left(\gamma,l\right)$ is a minimal
element and $\left(\gamma,l+1\right)<\left(\gamma,l\right)$, the
coefficient $\alpha_{\gamma,l+1}$ must vanish; therefore, the coefficient
of $x^{\gamma-1}\left(\log x\right)^{l}$ on the left can only be
obtained by differentiating the term $\alpha_{\gamma,l}x^{\gamma}\left(\log x\right)^{l}$,
and it must be $\gamma\alpha_{\gamma,l}$. On the other hand, write
$\alpha=\alpha_{0}+\tilde{\alpha}$, and write
\[
\star f_{\alpha}=\star f_{\alpha_{0}}+\star d_{\alpha_{0}}\tilde{\alpha}+\star\frac{1}{2}\left[\tilde{\alpha}\land\tilde{\alpha}\right].
\]
The term $\star f_{\alpha_{0}}$ is smooth, so it cannot contribute
to the $x^{\gamma-1}\left(\log x\right)^{l}$ coefficient because
$l>0$. The only terms of the expansion of $\star d_{\alpha_{0}}\tilde{\alpha}$
which could contribute to the $x^{\gamma-1}\left(\log x\right)^{l}$
term are $\star_{i}d_{\alpha_{0}}\alpha_{\gamma-i-1,l}$, for $i\geq0$;
however, $\left(\gamma-i-1,l\right)<\left(\gamma,l\right)$, so again
by minimality of $\left(\gamma,l\right)$ we must have $\alpha_{\gamma-i-1,l}=0$;
it follows that $\star d_{\alpha_{0}}\tilde{\alpha}$ does not contribute
either. Finally, we claim that the only terms in the expansion of
$\star\left(\tilde{\alpha}\land\tilde{\alpha}\right)$ that can contribute
to the $x^{\gamma-1}\left(\log x\right)^{l}$ term are $\star_{0}\frac{1}{2}\left[\theta\land\alpha_{\gamma,l}\right]$
and $\star_{0}\frac{1}{2}\left[\alpha_{\gamma,l}\land\theta\right]$.
Indeed, any other candidate must be of the form $\star_{i}\frac{1}{2}\left[\alpha_{\gamma_{1},l_{1}}\land\alpha_{\gamma_{2},l_{2}}\right]$,
with $i\geq0$, $\gamma_{1}+\gamma_{2}+i=\gamma-1$, and $l_{1}+l_{2}=l$.
However, again by minimality of $\left(\gamma,l\right)$, we must
have $\Re\left(\gamma_{i}\right)\geq\Re\left(\gamma\right)$; therefore,
\begin{align*}
\Re\left(\gamma\right)-1 & =\Re\left(\gamma_{1}+\gamma_{2}\right)+i\\
 & \geq2\Re\left(\gamma\right)+i\\
 & \geq2\Re\left(\gamma\right),
\end{align*}
from which we obtain $\Re\left(\gamma\right)\leq-1$. This is in contradiction
with the fact that $\Re\left(\mathcal{E}\right)>-1$. Concluding,
isolating the $x^{\gamma-1}\left(\log x\right)^{l}$ term from both
sides of $\partial_{x}\alpha=-\star f_{\alpha}$, we obtain the equation
\[
\gamma\alpha_{\gamma,l}=-\frac{1}{2}\left(\star_{0}\left[\theta\land\alpha_{\gamma,l}\right]+\star_{0}\left[\alpha_{\gamma,l}\land\theta\right]\right).
\]
Now we use Lemma \ref{lem:useful}: interpreting $\alpha_{\gamma,l}$
as a section of $\mathfrak{gl}\left(T\partial X\right)$, the equation
above becomes
\[
\gamma\alpha_{\gamma,l}=\alpha_{\gamma,l}^{*}-\theta\tr\alpha_{\gamma,l}.
\]
Decomposing
\[
\alpha_{\gamma,l}=\sk\alpha_{\gamma,l}+\symmtf\alpha_{\gamma,l}+\frac{\theta}{3}\tr\alpha_{\gamma,l}
\]
and substituting in the previous equation, we obtain
\[
\left(\gamma+1\right)\sk\alpha_{\gamma,l}+\left(\gamma-1\right)\symmtf\alpha_{\gamma,l}+\left(\gamma+2\right)\frac{\theta}{3}\tr\alpha_{\gamma,l}=0.
\]
If none of the coefficients $\gamma+1$, $\gamma-1$ and $\gamma+2$
vanishes, then the equation above implies that all the components
of $\alpha_{\gamma,l}$ vanish, contradicting the fact that $\alpha_{\gamma,l}\not=0$.
It follows that some of these coefficients must vanish. However, we
have $\Re\left(\gamma\right)>-1$ and either $\gamma\not\in\mathbb{N}$
or $\gamma\in\mathbb{N}$ and $l>\left(\gamma+1\right)/2$. The only
possibility is therefore that $\gamma=1$, in which case we have $l>1$
and the equation above becomes
\[
2\sk\alpha_{1,l}+\theta\tr\alpha_{1,l}=0.
\]
This equation then forces $\alpha_{1,l}$ to be self-adjoint trace-free,
but it does not force $\alpha_{1,l}=0$. Now we must use the fact
that $l>1$ implies a contradiction. This contradiction is obtained
by looking at the term $\left(\log x\right)^{l-1}$ at both sides
of the equation $\partial_{x}\alpha=-\star f_{\alpha}$. Note that,
since $l-1>0$, the smooth term $\star f_{\alpha_{0}}$ still does
not contribute. However, the contribution from $\partial_{x}\alpha$
is $\alpha_{1,l-1}+l\alpha_{1,l}$: the first term comes from differentiating
$x\left(\log x\right)^{l-1}\alpha_{1,l-1}$, while the second term
comes from differentiating $x\left(\log x\right)^{l}\alpha_{1,l}$.
Exactly as above, the contribution for $-\star f_{\alpha}$ is $-\left(\star_{0}\left[\theta\land\alpha_{1,l-1}\right]+\star_{0}\left[\alpha_{1,l-1}\land\theta\right]\right)/2$.
Therefore, we get the equation
\[
\alpha_{1,l-1}+l\alpha_{1,l}=-\frac{1}{2}\left(\star_{0}\left[\theta\land\alpha_{1,l}\right]+\star_{0}\left[\alpha_{1,l}\land\theta\right]\right)
\]
which using Lemma \ref{lem:useful} becomes
\[
2\sk\alpha_{1,l-1}+\theta\tr\alpha_{1,l-1}=-l\alpha_{1,l}.
\]
We have previously established that $\alpha_{1,l}$ is self-adjoint
trace-free; therefore, since $l\not=0$, the equation above implies
that $\alpha_{1,l}=0$. This is in contradiction with the fact that
$\left(1,l\right)\in\mathcal{S}$. It follows that $\mathcal{S}=\emptyset$,
as claimed. We conclude by showing that the expansion is smooth if
and only if $\alpha_{1,1}=0$. Denote by $\mathcal{S}'$ the set of
pairs $\left(k,l\right)$ for which $\alpha_{k,l}\not=0$, $k\geq1$,
and $0<l\leq\left(k+1\right)/2$. The expansion is smooth if and only
if $\mathcal{S}'=\emptyset$. Therefore, we need to show that $\mathcal{S}'=\emptyset$
if and only if $\alpha_{1,1}=0$. The implication $\mathcal{S}'=\emptyset\Rightarrow\alpha_{1,1}=0$
is trivial. Conversely, suppose that $\mathcal{S}'\not=\emptyset$.
Then $\mathcal{S}'$ must have a minimal element, and arguing exactly
as above, this minimal element must be $\left(1,1\right)$, i.e. $\alpha_{1,1}$
must be nonzero.
\end{proof}
We will now describe the coefficients of the expansion (\ref{eq:expansion})
more in detail. We will make heavy use of the results of §\ref{subsec:Geometry-of-4-manifolds-in-a-geodesic-collar},
so let us recall some conventions introduced there. Since we are working
in the geodesic collar $[0,\varepsilon)\times\partial X\to X$ determined
by $h_{0}$, we use the following notation:
\begin{enumerate}
\item $\overline{g}$ is the metric $x^{2}g=dx^{2}+h\left(x\right)$;
\item $\shape$ is the $x$-dependent family of endomorphisms $T\partial X\to T\partial X$
describing the shape operator of the level sets of $x$ in the collar;
\item $\HH$ is the $x$-dependent trace of $\shape$, i.e. the $x$-dependent
mean curvature of the level sets of $x$;
\item $\overline{\R}$, $\overline{\Ric}$, $\overline{\s}$ are the Riemann
curvature, Ricci curvature and scalar curvature of $\overline{g}:=x^{2}g$;
\item $\Ric$, $\s$, $\G$ are the $x$-dependent Ricci tensor, scalar
curvature, and Einstein tensor of the level sets of $x$;
\item $\overline{\R}^{N}$ is the $x$-dependent family of endomorphisms
$T\partial X\to T\partial X$ defined by $\overline{\R}\left(V,\partial_{x}\right)\partial_{x}$;
\item $\overline{\Ric}^{N}$ is the $x$-dependent trace of $\overline{\R}^{N}$.
\end{enumerate}
Whenever we have an $x$-dependent family $\mathcal{S}$ of endomorphisms
$T\partial X\to T\partial X$, we call:
\begin{enumerate}
\item $\mathcal{S}^{\star}$ the dual $x$-dependent family of endomorphisms
$T^{*}\partial X\to T^{*}\partial X$;
\item $\mathcal{S}_{0}$ the evaluation of $\mathcal{S}$ at $x=0$;
\item $\mathring{\mathcal{S}}$ the trace-free part of $\mathcal{S}$.
\end{enumerate}
\begin{thm}
\label{thm:coefficients-of-the-instanton-expansion}Let $A$, $h_{0}$,
$x$, $\alpha\left(x\right)$ be as in Theorem \ref{thm:index-set-0-instantons}.
Let $h\left(x\right)\sim h_{0}+h_{1}x+h_{2}x^{2}+o\left(x^{2}\right)$
be the family of metrics on $\partial X$ in the geodesic normal form
(\ref{eq: geodesic normal form}) associated to $h_{0}$. Moreover,
call $\omega$ the Levi-Civita spin connection on $\left(\mathbb{S},\cl_{h_{0}}\right)\to\left(\partial X,h_{0}\right)$.
Then the coefficients of the expansion (\ref{eq:expansion}) satisfy
the following properties:
\begin{enumerate}
\item The $\SU\left(2\right)$ connection $\alpha_{0}$ on $\mathbb{S}$
is completely determined by the torsion equation
\[
d_{\alpha_{0}}\theta=\star_{1}\theta.
\]
In particular, $\alpha_{0}=\omega$ if and only if $\left(\partial X,h_{0}\right)$
is totally geodesic in $\left(X,x^{2}g\right)$.
\item The skew and trace parts of $\alpha_{1}$ are completely determined
by the equations
\begin{align*}
\sk\alpha_{1} & =\frac{1}{2}\sk\star_{0}\left(\star_{2}\theta-f_{\alpha_{0}}\right)\\
\tr\alpha_{1} & =\frac{1}{3}\tr\star_{0}\left(\star_{2}\theta-f_{\alpha_{0}}\right),
\end{align*}
while the self-adjoint trace-free part of $\alpha_{1}$ is formally
free.
\item The coefficient $\alpha_{1,1}$ is self-adjoint, trace-free, and completely
determined by the equation
\begin{align*}
\alpha_{1,1} & =\symmtf\star_{0}\left(\star_{2}\theta-f_{\alpha_{0}}\right).
\end{align*}
In particular, if $\left(\partial X,h_{0}\right)$ is totally geodesic
in $\left(X,x^{2}g\right)$, then $\alpha_{1,1}$ vanishes if and
only if $\tf h_{2}=-\mathring{\Ric}\left(h_{0}\right)$.
\item Every coefficient $\alpha_{k,l}$ in the expansion \ref{eq:expansion},
with $k>1$, is determined by $\symmtf\alpha_{1}$ and the symmetric
$2$-tensors $h_{0},h_{1},...,h_{k+1}$ on $\partial X$, via differential
relations.
\end{enumerate}
\end{thm}


\begin{proof}[Proof of Theorem \ref{thm:coefficients-of-the-instanton-expansion}]
As in the previous proof, write the equation $\partial_{x}\alpha=-\star f_{\alpha}$
in the form
\begin{equation}
\partial_{x}\alpha=-\star\left(f_{\alpha_{0}}+d_{\alpha_{0}}\tilde{\alpha}+\frac{1}{2}\left[\tilde{\alpha}\land\tilde{\alpha}\right]\right),\label{eq:SD-in-vertical-gauge}
\end{equation}
with $\tilde{\alpha}=\alpha-\alpha_{0}$.\\
(1) Isolating the $x^{-1}$ term in Equation (\ref{eq:SD-in-vertical-gauge}),
we obtain the equation
\[
0=-\star_{0}d_{\alpha_{0}}\theta-\frac{1}{2}\star_{1}\left[\theta\land\theta\right].
\]
Since $\theta$ is the soldering form, we get
\[
\star_{0}\frac{1}{2}\left[\theta\land\theta\right]=\theta.
\]
Moreover, since $\star^{2}=1$, we get $\star_{0}\star_{1}=-\star_{1}\star_{0}$.
Therefore, applying $\star_{0}$ to the previous equation, we obtain
\[
d_{\alpha_{0}}\theta=\star_{1}\theta.
\]
By Lemma \ref{lem:connection-induced-by-torsion}, we have
\[
\alpha_{0}=\omega+\sk\star_{0}\star_{1}\theta-\symmtf\star_{0}\star_{1}\theta+\frac{1}{6}\left(\tr\star_{0}\star_{1}\theta\right)\theta.
\]
By Lemma \ref{lem:computations-in-collar}, we have $\star_{1}=\star_{0}\left(2\shape_{0}^{\star}\theta-\HH_{0}\theta\right)$.
Therefore, we get
\begin{align*}
\alpha_{0} & =\omega-2\symmtf\left(\shape_{0}^{\star}\theta\right)-\frac{1}{6}\HH_{0}\theta\\
 & =\omega-2\shape_{0}^{\star}\theta+\frac{1}{2}\HH_{0}\theta.
\end{align*}
It follows that $\alpha_{0}=\omega$ if and only if $\shape_{0}=0$,
i.e. if and only if the embedding $\left(\partial X,h_{0}\right)\to\left(X,\overline{g}\right)$
is totally geodesic.\\
(2 and 3) Using the computations of the previous proof, we isolate
the $\log x$ term in Equation (\ref{eq:SD-in-vertical-gauge}) and
we obtain
\[
2\sk\alpha_{1,1}+\theta\tr\alpha_{1,1}=0,
\]
which implies that $\alpha_{1,1}$ is self-adjoint trace-free. Now,
isolating the constant term, we obtain the equation

\[
\alpha_{1}+\alpha_{1,1}=-\star_{0}f_{\alpha_{0}}-\star_{1}d_{\alpha_{0}}\theta-\star_{0}\frac{1}{2}\left(\left[\theta\land\alpha_{1}\right]+\left[\alpha_{1}\land\theta\right]\right)-\star_{2}\frac{1}{2}\left[\theta\land\theta\right].
\]
Using Lemma \ref{lem:useful}, and the previous identity $d_{\alpha_{0}}\theta=\star_{1}\theta$,
we obtain
\[
\alpha_{1}+\alpha_{1,1}=-\star_{0}f_{\alpha_{0}}-\star_{1}^{2}\theta-\theta\tr\alpha_{1}+\alpha_{1}^{*}-\star_{2}\frac{1}{2}\left[\theta\land\theta\right].
\]
Note that, since $\star^{2}=1$, the coefficient of $x^{2}$ in the
expansion of $\star^{2}$ must be zero, i.e. $\star_{0}\star_{2}+\star_{1}^{2}+\star_{2}\star_{0}=0$.
Moreover, the Nahm pole boundary condition implies that
\[
\frac{1}{2}\left[\theta\land\theta\right]=\star_{0}\theta.
\]
Therefore, the equation above becomes
\[
2\sk\alpha_{1}+\theta\tr\alpha_{1}+\alpha_{1,1}=\star_{0}\left(\star_{2}\theta-f_{\alpha_{0}}\right).
\]
This implies that
\begin{align*}
\sk\alpha_{1} & =\frac{1}{2}\sk\star_{0}\left(\star_{2}\theta-f_{\alpha_{0}}\right)\\
\tr\alpha_{1} & =\frac{1}{3}\tr\star_{0}\left(\star_{2}\theta-f_{\alpha_{0}}\right)\\
\alpha_{1,1} & =\symmtf\star_{0}\left(\star_{2}\theta-f_{\alpha_{0}}\right).
\end{align*}
By Lemma \ref{lem:computations-in-collar}, $h_{1}=0$ if and only
if $\left(\partial X,h_{0}\right)$ is totally geodesic in $\left(X,\overline{g}\right)$.
Assuming this condition, by the previous point we have $f_{\alpha_{0}}=f_{\omega}$.
As explained in §\ref{subsec:Geometry-of-spin-3D}, the $\mathfrak{su}\left(\mathbb{S}\right)$-valued
$1$-form $\star_{0}f_{\omega}$ is the endomorphism $\G_{0}^{\sharp}:T\partial X\to T\partial X$
obtained by raising an index of the Einstein tensor $\G_{0}$ of $h_{0}$.
Moreover, still by Lemma \ref{lem:computations-in-collar}, we have
\begin{align*}
\star_{0}\star_{2} & =\overline{\R}_{0}^{N\star}-\frac{1}{2}\overline{\Ric}_{0}^{N}.
\end{align*}
Therefore,
\[
\star_{0}\left(\star_{2}\theta-f_{\alpha_{0}}\right)=\left(\overline{\R}_{0}^{N\star}-\frac{1}{2}\overline{\Ric}_{0}^{N}-\G_{0}^{\sharp\star}\right)\theta.
\]
Taking the trace-free part and lowering an index, we get
\begin{align*}
\alpha_{1,1}^{\flat} & =\tf\overline{\R}_{0}^{N}-\mathring{\Ric}_{0}.
\end{align*}
Again by Lemma \ref{lem:computations-in-collar}, if $h_{1}=0$ then
$h_{2}=-\overline{\R}_{0}^{N}$. Therefore, $\alpha_{1,1}=0$ if and
only if $\tf h_{2}=-\mathring{\Ric}_{0}$.\\
(4) If we isolate the coefficients of $x^{k-1}\left(\log x\right)^{l}$
in Equation (\ref{eq:SD-in-vertical-gauge}), arguing as above we
get an equation of the form
\[
k\alpha_{k,l}+\star_{0}\frac{1}{2}\left(\left[\alpha_{-1}\land\alpha_{k,l}\right]+\left[\alpha_{k,l}\land\alpha_{-1}\right]\right)=T_{k,l},
\]
where $T_{k,l}$ is a term which depends only on the metric and the
coefficients $\alpha_{k',l'}$ of the expansion of $\alpha$ with
$\left(k',l'\right)<\left(k,l\right)$. Since $k>1$, the left hand
side can be seen as a linear combination of the skew-adjoint, trace,
and self-adjoint trace-free parts of $\alpha_{k,l}$ with constant
non-zero coefficients. Therefore, this equation completely prescribes
$\alpha_{k,l}$.
\end{proof}

\subsection{\label{subsec:The--instanton-obstruction}The $0$-instanton obstruction
tensor}

Let $A$ be a self-dual polyhomogeneous $0$-connection satisfying
the Nahm pole boundary condition. As above, let us work in a geodesic
collar induced by a choice of $h_{0}\in\mathfrak{c}_{\infty}\left(g\right)$.
In the previous subsection, we proved that the associated geodesic
normal family $\alpha_{h_{0}}\left(x\right)$ of $A$ has an expansion
of the form
\[
\alpha_{h_{0}}\left(x\right)\sim\theta x^{-1}+\alpha_{0}+\alpha_{1,1}x\log x+\alpha_{1}x+o\left(x\right),
\]
where:
\begin{enumerate}
\item $\theta$ is the $\mathfrak{su}\left(\mathbb{S}\right)$-valued soldering
form;
\item $\alpha_{0}$ is the $\SU\left(2\right)$ connection on $\mathbb{S}$
determined by $d_{\alpha_{0}}\theta=\star_{1}\theta$;
\item $\alpha_{1,1}$ is the self-adjoint trace-free endomorphism of $T\partial X$
\[
\alpha_{1,1}=\symmtf\star_{0}\left(\star_{2}\theta-f_{\alpha_{0}}\right).
\]
\end{enumerate}
We also saw that the expansion does not contain log terms if and only
if $\alpha_{1,1}=0$. This tensor is clearly analogous to the obstruction
tensor arising in the Fefferman--Graham expansion of Poincaré--Einstein
metrics, cf. §\ref{subsec:Poincar=0000E9=002013Einstein-metrics}.
For this reason, we formulate the following
\begin{defn}
We call $\symmtf\star_{0}\left(\star_{2}\theta-f_{\alpha_{0}}\right)$
the \emph{$0$-instanton obstruction tensor }of $\left(X,g\right)$.
\end{defn}

Note that the $0$-instanton obstruction tensor is manifestly independent
of $A$. However, \emph{a priori}, this tensor seems to depend on
the choice of $h_{0}\in\mathfrak{c}_{\infty}\left(g\right)$. In the
next theorem, we prove that the $0$-instanton obstruction tensor
is in fact a \emph{conformal invariant} of $\left(X,g\right)$: more
precisely, it is a conformal invariant which depends only on the second
jet of the conformal class of $g$ at $\partial X$. In order to formulate
the theorem precisely, recall again some notation from §\ref{subsec:Geometry-of-4-manifolds-in-a-geodesic-collar}:
\begin{enumerate}
\item $\overline{\W}$ is the Weyl curvature tensor of $\left(X,\overline{g}\right)$
on the collar;
\item $\overline{\W}^{B}$ and $\overline{\W}^{E}$ are the $x$-dependent
magnetic and electric parts of $\overline{\W}$ along the level sets
of $x$, i.e. the $x$-dependent endomorphisms of $T\partial X$ defined
by
\begin{align*}
T\partial X & \to T\partial X\\
\overline{\W}^{E}:V & \mapsto\overline{\W}\left(V,\partial_{x}\right)\partial_{x}\\
\overline{\W}^{B}:V & \mapsto\left(*\overline{\W}^{E}\right)V.
\end{align*}
\end{enumerate}
Again, we denote by $\overline{\W}_{0}^{B}$ and $\overline{\W}_{0}^{E}$
the restrictions of $\overline{\W}^{B}$, $\overline{\W}^{E}$ at
$x=0$.
\begin{thm}
\label{thm:0-instanton-obstruction-tensor-conf-invariant}The $0$-instanton
obstruction tensor of $\left(X,g\right)$ is conformally invariant;
more precisely, it coincides with $2\left(\overline{\W}_{0}^{B}-\overline{\W}_{0}^{E}\right)$.
\end{thm}

\begin{proof}
The proof is an elementary book-keeping exercise in recognizing the
terms in $\symmtf\star_{0}\left(\star_{2}\theta-f_{\alpha_{0}}\right)$
with the help of Lemmas \ref{lem:computations-in-collar} and \ref{lem:useful}.
From Point 10 of Lemma \ref{lem:computations-in-collar}, we have
\[
\star_{0}\star_{2}=3\left(\shape_{0}^{\star}\right)^{2}-2\shape_{0}^{\star}\HH_{0}+\frac{1}{2}\HH_{0}^{2}+\overline{\R}_{0}^{N\star}-\frac{1}{2}\tr\left(\shape_{0}^{2}\right)-\frac{1}{2}\overline{\Ric}_{0}^{N}.
\]
Therefore, we obtain
\[
\symmtf\left(\star_{0}\star_{2}\theta\right)=\tf\left(3\left(\shape_{0}^{\star}\right)^{2}-2\HH_{0}\shape_{0}^{\star}+\overline{\R}_{0}^{N\star}\right)\theta.
\]
Now, from the proof of Theorem \ref{thm:coefficients-of-the-instanton-expansion},
the equation $d_{\alpha_{0}}\theta=\star_{1}\theta$ determines the
connection $\alpha_{0}$ as
\[
\alpha_{0}=\omega-2\shape_{0}^{\star}\theta+\frac{1}{2}\HH_{0}\theta,
\]
where $\omega$ is the Levi-Civita spin connection on $\left(\mathbb{S},\cl_{h_{0}}\right)\to\left(\partial X,h_{0}\right)$.
Define $\gamma=-2\shape_{0}^{\star}\theta+\frac{1}{2}\HH_{0}\theta$.
Then we have
\[
\star_{0}f_{\alpha_{0}}=\star_{0}f_{\omega}+\star_{0}d_{\omega}\gamma+\star_{0}\frac{1}{2}\left[\gamma\land\gamma\right].
\]
Since $\star_{0}f_{\omega}=\G_{0}^{\sharp}$, we have $\tf\star_{0}f_{\omega}=\mathring{\Ric}_{0}^{\sharp}$.
To compute the trace-free part of $\star_{0}\frac{1}{2}\left[\gamma\land\gamma\right]$,
we use Lemma \ref{lem:useful} and we obtain
\[
\tf\star_{0}\frac{1}{2}\left[\gamma\land\gamma\right]=\tf\left(4\left(\shape_{0}^{\star}\right)^{2}-3\HH_{0}\shape_{0}^{\star}\right)\theta.
\]
Concerning the term $\star_{0}d_{\omega}\gamma$, we see that $\star_{0}d_{\omega}\left(\HH_{0}\theta\right)$
is essentially the gradient of $\HH_{0}$, and therefore its interpretation
as an endomorphism of $T\partial X$ is skew-adjoint. The remaining
part is a multiple of $\star_{0}d_{\omega}\left(\shape_{0}^{\star}\theta\right)$,
and by definition and by Proposition \ref{prop:electric-magnetic-Weyl}
we have
\[
\star_{0}d_{\omega}\left(\shape_{0}^{\star}\theta\right)=\overline{\W}_{0}^{B}.
\]
Putting it all together and applying again Proposition \ref{prop:electric-magnetic-Weyl},
we obtain
\begin{align*}
\symmtf\star_{0}\left(\star_{2}\theta-f_{\alpha_{0}}\right) & =\tf\left(-\left(\shape_{0}^{\star}\right)^{2}+\overline{\R}_{0}^{N\star}-\Ric_{0}^{\sharp\star}+\HH_{0}\shape_{0}^{\star}\right)\theta+2\overline{\W}_{0}^{B\star}\theta\\
 & =2\left(\overline{\W}_{0}^{B}-\overline{\W}_{0}^{E}\right)
\end{align*}
as claimed.
\end{proof}
\begin{cor}
The $0$-instanton obstruction tensor of $\left(X,g\right)$ vanishes
if and only if the anti-self-dual Weyl tensor of the conformal metric
$\left[g\right]$ vanishes along $\partial X$.
\end{cor}

\begin{proof}
By the Gauss--Codazzi equations, any Weyl-type tensor is completely
determined on the boundary by its electric and magnetic parts. Now,
observe that
\begin{align*}
\left(\overline{\W}^{-}\right)^{E} & =\frac{1}{2}\left(\overline{\W}-*\overline{\W}\right)^{E}=\frac{1}{2}\left(\overline{\W}^{E}-\overline{\W}^{B}\right),
\end{align*}
and
\[
\left(\overline{\W}^{-}\right)^{B}=\left(*\overline{\W}^{-}\right)^{E}=-\left(\overline{\W}^{-}\right)^{E}.
\]
Therefore, $\overline{\W}_{0}^{-}$ vanishes if and only if $\overline{\W}_{0}^{E}-\overline{\W}_{0}^{B}=0$,
i.e. if and only if the $0$-instanton obstruction tensor vanishes.
\end{proof}
\begin{cor}
\label{cor:asymptotically-PE-implies-vanishing-obstruction}If $\left(X,g\right)$
is conformally Poincaré--Einstein to second order, then the $0$-instanton
obstruction tensor vanishes.
\end{cor}

\begin{proof}
Since the $0$-instanton obstruction tensor is conformally invariant,
we can assume that $g$ is Poincaré--Einstein. Choose $h_{0}\in\mathfrak{c}_{\infty}\left(g\right)$,
let $x$ be a geodesic boundary defining function inducing $h_{0}$,
and write the $0$-instanton obstruction tensor as
\[
\symmtf\star_{0}\left(\star_{2}\theta-f_{\alpha_{0}}\right).
\]
Since $\left(X,g\right)$ is Poincaré--Einstein to first order, the
boundary $\left(\partial X,h_{0}\right)$ is totally geodesic in $\left(X,x^{2}g\right)$,
and therefore $\alpha_{0}=\omega$ and we obtain the simplification
\[
\symmtf\star_{0}\left(\star_{2}\theta-f_{\alpha_{0}}\right)=\tf\left(\overline{\R}_{0}^{N\star}-\Ric_{0}^{\star}\right)\theta.
\]
Since $\left(X,g\right)$ is Poincaré--Einstein to second order,
then $\overline{\R}_{0}^{N}=-h_{2}$ equals the Schouten tensor of
$h_{0}$, and therefore $\tf\overline{\R}_{0}^{N}=\mathring{\Ric}_{0}$,
so the $0$-instanton obstruction tensor vanishes as claimed.
\end{proof}
We conclude with the following regularity result, which reminds of
similar behavior for Poincaré--Einstein metrics (cf. Remark \ref{rem:regularity-PE})
and therefore justifies the name given to the $0$-instanton obstruction
tensor:
\begin{thm}
\label{thm:obstruction-zero-iff-smooth}Let $E$ be an $\SU\left(2\right)$
vector bundle on $X$, and let $A$ be a polyhomogeneous self-dual
$0$-connection on $E$, satisfying the Nahm pole boundary condition.
Then $A$ is gauge equivalent, via a $O\left(1\right)$ polyhomogeneous
gauge transformation $E$, to a $0$-connection log-smooth of order
$2$. Moreover, $A$ is gauge-equivalent to a \emph{smooth} $0$-connection
if and only if the $0$-instanton obstruction tensor vanishes.
\end{thm}

\begin{proof}
By Theorem \ref{thm:index-set-0-instantons}, the expansion of the
geodesic normal family $\alpha\left(x\right)$ of $A$ contains only
terms $x^{k}\left(\log x\right)^{l}$ with $k\in\mathbb{Z}$, $k\geq-1$,
and $l\leq\left(k+1\right)/2$. This is exactly the same thing as
saying that the $0$-$1$-form representing $A$, with respect to
a frame $s_{1},s_{2}$ of $E$ which satisfies $A_{x\partial_{x}}s_{i}\equiv0$
near the boundary, is log-smooth of order $2$. By Proposition \ref{prop:parallel-transport},
the frame $s_{1},s_{2}$ is $O\left(1\right)$ polyhomogeneous. Therefore,
the gauge transformation $\Phi$ to this frame from any smooth frame
of $E$, is $O\left(1\right)$ polyhomogeneous. Furthermore, again
by Theorem \ref{thm:index-set-0-instantons}, the expansion of $\alpha\left(x\right)$
does not contain log terms if and only if the $0$-instanton obstruction
tensor vanishes; since the geodesic normal family is gauge invariant,
this is equivalent to the fact that $A$ is gauge equivalent to a
smooth $0$-connection, in the sense above.
\end{proof}
\begin{rem}
One can prove a result similar to the main theorem of \cite{ChruscielDelayLeeSkinnerBoundaryRegularity}:
roughly speaking, if $A$ is sufficiently regular to be able to define
a \emph{partial} asymptotic expansion $\alpha\left(x\right)\sim\alpha_{-1}x^{-1}+\alpha_{0}+\alpha_{1,1}x\log x+\alpha_{1}x+o\left(x\right)$
of the associated geodesic normal family, then $A$ is log-smooth
of order $2$ modulo gauge, \emph{provided that $\symmtf\alpha_{1}$
is smooth}. This regularity result will be proved in \cite{UsulaNahmPolesII}.
\end{rem}

\subsection{The boundary value of a $0$-instanton}

From Theorem \ref{thm:coefficients-of-the-instanton-expansion}, if
$A$ is a self-dual $0$-connection satisfying the Nahm pole boundary
condition, then the expansion
\[
\alpha\left(x\right)\sim\theta x^{-1}+\alpha_{0}+\alpha_{1,1}x\log x+\alpha_{1}x+o\left(x\right)
\]
of the geodesic normal family of $A$ associated to a choice of $h_{0}\in\mathfrak{c}_{\infty}\left(g\right)$
contains a formally free term, namely the symmetric trace-free part
of $\alpha_{1}\in T^{*}\partial X\otimes\mathfrak{su}\left(\mathbb{S}\right)\simeq S^{2}\left(T^{*}\partial X\right)$.
This term is of fundamental importance in the study of the self-duality
equation for $0$-connections, and therefore we give it a name:
\begin{defn}
Denote by $\left[A\right]$ the gauge class of $A$ modulo polyhomogeneous
gauge transformations equal to the identity at $\partial X$. We call
$\beta_{h_{0}}\left(\left[A\right]\right):=\symmtf\alpha_{1}$ the
\emph{boundary value} of $\left[A\right]$ associated to $h_{0}\in\mathfrak{c}_{\infty}\left(g\right)$.
\end{defn}

\begin{rem}
By Proposition \ref{prop:geodesic-normal-family-is-gauge-invariant},
the geodesic normal family associated to $h_{0}\in\mathfrak{c}_{\infty}\left(g\right)$
is a gauge invariant, so the definition of $\beta_{h_{0}}\left(\left[A\right]\right)$
is well-posed.
\end{rem}

In this subsection, we compute the boundary value of the most important
class of examples considered here:
\begin{thm}
\label{thm:boundary-value-LC-spin}Let $\left(X^{4},g\right)$ be
a Poincaré--Einstein manifold, and let $\left(\mathbb{S}_{\pm},{^{0}\cl}\right)$
be a $0$-spin structure on $\left(X^{4},g\right)$. Denote by $\nabla^{+}$
the Levi-Civita spin $0$-connection on $\mathbb{S}_{+}$. Finally,
let $h_{0}\in\mathfrak{c}_{\infty}\left(g\right)$. Then the boundary
value of $\left[\nabla^{+}\right]$ is
\[
\beta_{h_{0}}\left(\left[\nabla^{+}\right]\right)=\frac{1}{2}\mathring{\Ric}\left(h_{0}\right)^{\sharp}.
\]
\end{thm}


The remaining part of the subsection is devoted to the proof of the
result above. To simplify the computations slightly, we will actually
compute the first terms of the expansion of the geodesic normal family
of the Levi-Civita $0$-connection on $^{0}\Lambda^{+}$. These two
$0$-connections are completely equivalent: indeed, the chosen $0$-spin
structure $\left(\mathbb{S}_{\pm},{^{0}\cl}\right)$ determines a
natural 1-1 correspondence between $\SU\left(2\right)$ $0$-connections
on $\mathbb{S}_{+}$ and $\SO\left(3\right)$ $0$-connections on
$^{0}\Lambda^{+}$, under which the Levi-Civita $0$-connection on
$^{0}\Lambda^{+}$ corresponds to the Levi-Civita spin $0$-connection
on $\mathbb{S}_{+}$.

Choose $h_{0}\in\mathfrak{c}_{\infty}\left(g\right)$, and write the
metric in geodesic normal form as
\[
g=\frac{dx^{2}+h\left(x\right)}{x^{2}}.
\]
Since $g$ is Poincaré--Einstein, we have 
\[
h\left(x\right)\sim h_{0}-\PP\left(h_{0}\right)x^{2}+O\left(x^{3}\right),
\]
where $\PP\left(h_{0}\right)$ is the Schouten tensor of $h_{0}$
(cf. Remark \ref{rem:asymptotically-PE-equals-condition-on-h1-and-h2}).
We define as before $e^{0}=-dx/x$, and we denote by $e_{0}=-x\partial_{x}$
the $g$ metric dual of $e^{0}$. We now choose an $\SO\left(3\right)$
coframe $\theta^{1},\theta^{2},\theta^{3}$ for $\left(\partial X,h_{0}\right)$,
and we denote by $e^{i}$ the unique $0$-$1$-forms on the collar
such that
\begin{align*}
\nabla_{e_{0}}e^{i} & =0\\
\left(xe^{i}\right)_{|x=0} & =\theta^{i}.
\end{align*}
Note that $\nabla_{e_{0}}e^{0}=0$ as well. Indeed, since $\nabla$
is a torsion-free $0$-connection, we have
\[
d\omega\left(V,W\right)=W\lrcorner\left(\nabla_{V}\omega\right)-V\lrcorner\left(\nabla_{W}\omega\right)
\]
for every $V,W\in\mathcal{V}_{0}\left(X\right)$ and every $0$-$1$-form
$\omega$. Therefore, from $de^{0}=0$, we have
\[
V\lrcorner\left(\nabla_{e_{0}}e^{0}\right)=e_{0}\lrcorner\left(\nabla_{V}e^{0}\right)
\]
and, since $\left|e^{0}\right|\equiv1$ on the collar, we have
\begin{align*}
0 & =V\left|e^{0}\right|^{2}\\
 & =2\left\langle \nabla_{V}e^{0},e^{0}\right\rangle \\
 & =2e_{0}\lrcorner\left(\nabla_{V}e^{0}\right).
\end{align*}
From the generality of $V$, we have indeed $\nabla_{e_{0}}e^{0}=0$.
From this property, and the fact that $\left(e^{0},e^{1},e^{2},e^{3}\right)$
is an $\SO\left(4\right)$ coframe of $X$ along $\partial X$ by
construction, it follows that $\left(e^{0},e^{1},e^{2},e^{3}\right)$
is an $\SO\left(4\right)$ coframe on the whole collar; in particular,
the $e^{i}$ with $i\geq1$ are orthogonal to $e^{0}$, and therefore
we can write them as $e^{i}=x^{-1}b^{i}\left(x\right)$ where $b^{1}\left(x\right),b^{2}\left(x\right),b^{3}\left(x\right)$
is an $\SO\left(3\right)$ coframe of $\left(\partial X,h\left(x\right)\right)$
with $b^{i}\left(0\right)=\theta^{i}$.

Define now
\begin{align*}
\omega^{1} & =e^{01}+e^{23}\\
\omega^{2} & =e^{02}+e^{31}\\
\omega^{3} & =e^{03}+e^{12}.
\end{align*}
The triple $\omega^{1},\omega^{2},\omega^{3}$ is a positively oriented
orthogonal frame for $^{0}\Lambda^{+}$, with elements of constant
length $\sqrt{2}$. As such, the Levi-Civita $0$-connection on $^{0}\Lambda^{+}$
can be written in terms of this frame as $\nabla^{+}=d+\alpha$, where
\[
\alpha=\alpha^{i}\otimes\mathfrak{a}_{i}=\left(\begin{matrix} & -\alpha^{3} & \alpha^{2}\\
\alpha^{3} &  & -\alpha^{1}\\
-\alpha^{2} & \alpha^{1}
\end{matrix}\right)
\]
and $\mathfrak{a}_{1},\mathfrak{a}_{2},\mathfrak{a}_{3}$ is the basis
of $\mathfrak{so}\left(3\right)$ chosen in §\ref{subsec:Nahm-poles-and-self-duality}.
Since $\nabla_{e_{0}}e^{0}=\nabla_{e_{0}}e^{i}\equiv0$, we have $\nabla_{e_{0}}^{+}\omega^{i}\equiv0$
as well, and this implies that $\alpha^{i}=\alpha^{i}\left(x\right)$
are $O\left(x^{-1}\right)$ families of $1$-forms on $\partial X$.
Theorem \ref{thm:coefficients-of-the-instanton-expansion} then implies
that
\[
\alpha^{i}\left(x\right)\sim\theta^{i}x^{-1}+\alpha_{0}^{i}+\alpha_{1}^{i}x+O\left(x^{2}\right),
\]
where $\alpha_{0}=\alpha_{0}^{i}\otimes\mathfrak{a}_{i}$ is the local
expression in terms of the coframe $\theta^{1},\theta^{2},\theta^{3}$
of the Levi-Civita connection on $T^{*}\partial X$, i.e.
\begin{align*}
d\theta^{1}+\alpha_{0}^{2}\land\theta^{3}-\alpha_{0}^{3}\land\theta^{2} & =0\\
d\theta^{2}+\alpha_{0}^{3}\land\theta^{1}-\alpha_{0}^{1}\land\theta^{3} & =0\\
d\theta^{3}+\alpha_{0}^{1}\land\theta^{2}-\alpha_{0}^{2}\land\theta^{1} & =0.
\end{align*}
Our aim is to compute the $1$-forms $\alpha_{1}^{i}$. In order to
do that, we use the fact that $\nabla^{+}$ is uniquely characterized
as the unique torsion-free $\SO\left(3\right)$ $0$-connection on
$^{0}\Lambda^{+}$ (cf. Proposition 2.3 of \cite{FineGaugeTheoretic}).
In terms of the frame $\omega^{1},\omega^{2},\omega^{3}$, this condition
is
\begin{align}
d\omega^{1} & =\alpha^{3}\land\omega^{2}-\alpha^{2}\land\omega^{3}\nonumber \\
d\omega^{2} & =\alpha^{1}\land\omega^{3}-\alpha^{3}\land\omega^{1}\nonumber \\
d\omega^{3} & =\alpha^{2}\land\omega^{1}-\alpha^{1}\land\omega^{2}.\label{eq:torsion-free-lambda+-1}
\end{align}
These equations can be expanded in powers of $x$, with a leading
term of order $x^{-3}$. Isolating the $x^{-1}$ coefficients, we
obtain the desired condition for $\alpha_{1}=\alpha_{1}^{i}\otimes\mathfrak{a}_{i}$.

To simplify the notation, define $\theta=\theta^{i}\otimes\mathfrak{a}_{i}$
and $\omega=\omega^{i}\otimes\mathfrak{a}_{i}$. Then the torsion-free
equation above can be written compactly as
\[
d\omega+\left[\alpha\land\omega\right]=0.
\]
The next step is to compute the expansion of $\omega$ as $x\to0$.
Denote by $\rho$ the $\mathfrak{so}\left(3\right)$-valued $1$-form
$\PP\left(h_{0}\right)^{\sharp\star}\theta$, where $\PP\left(h_{0}\right)^{\star}$
is the self-adjoint endomorphism of $T^{*}\partial X$ induced by
the Schouten tensor $\PP\left(h_{0}\right)^{\sharp}$ with an index
raised.
\begin{lem}
We have
\[
b\left(x\right)\sim\theta-\frac{1}{2}\rho x^{2}+O\left(x^{3}\right).
\]
\end{lem}

\begin{proof}
Denote by $\overline{\nabla}$ the Levi-Civita connection of $\overline{g}=dx^{2}+h\left(x\right)$
on $1$-forms. Then $\overline{\nabla}$ is related to $\nabla$ by
the formula
\[
\overline{\nabla}_{V}\omega=\nabla_{V}\omega-\left(V\log x\right)\omega-\omega\left(V\right)d\log x+\left\langle \omega,d\log x\right\rangle _{g}V^{\flat}.
\]
Substitute $V=e_{0}$ and $\omega=e^{i}$: we obtain
\[
\overline{\nabla}_{e_{0}}e^{i}=\nabla_{e_{0}}e^{i}-\left(e_{0}\log x\right)e^{i}-e^{i}\left(e_{0}\right)d\log x+\left\langle e^{i},d\log x\right\rangle _{g}e^{0}.
\]
Note that $d\log x=x^{-1}dx=-e^{0}$. Since $\left\langle e^{0},e^{i}\right\rangle _{g}\equiv0$,
we have $e^{i}\left(e_{0}\right)=\left\langle e^{i},d\log x\right\rangle _{g}=0$.
Moreover, $e_{0}\log x=-x\partial_{x}\left(\log x\right)=-1$. Therefore,
from $\nabla_{e_{0}}e^{i}\equiv0$, we obtain
\[
-x\overline{\nabla}_{\partial_{x}}\left(x^{-1}b^{i}\right)=x^{-1}b^{i}.
\]
On the other hand, $\overline{\nabla}_{\partial_{x}}\left(x^{-1}b^{i}\right)=-x^{-2}b^{i}+x^{-1}\overline{\nabla}_{\partial_{x}}b^{i}$.
Therefore we get
\[
\overline{\nabla}_{\partial_{x}}b^{i}=0.
\]
Now, from Lemma \ref{lem:computations-in-collar}, we have $\dot{b}^{i}=\overline{\nabla}_{\partial_{x}}b^{i}-\shape^{\star}b^{i}$.
Evaluating at $x=0$ and using the fact that $\left(\partial X,h_{0}\right)$
is totally geodesic in $\left(X,x^{2}g\right)$, we obtain $\dot{b}^{i}\left(0\right)=0$,
so the coefficient of $x$ in the expansion of $b^{i}\left(x\right)$
is $0$. Concerning the coefficient of $x^{2}$, again from Lemma
\ref{lem:computations-in-collar} we have
\begin{align*}
\ddot{b}^{i} & =\overline{\nabla}_{\partial_{x}}\dot{b}^{i}-\shape^{\star}\dot{b}^{i}\\
 & =\overline{\nabla}_{\partial_{x}}\left(\overline{\nabla}_{\partial_{x}}b^{i}-\shape^{\star}b^{i}\right)-\shape^{\star}\left(\overline{\nabla}_{\partial_{x}}b^{i}-\shape^{\star}b^{i}\right).
\end{align*}
Evaluating again at $x=0$, we obtain $\ddot{b}^{i}\left(0\right)=-\left(\overline{\nabla}_{\partial_{x}}\left(\shape^{\star}b^{i}\right)\right)_{|x=0}.$
Now, we have
\begin{align*}
\overline{\nabla}_{\partial_{x}}\left(\shape^{\star}b^{i}\right) & =\left(\overline{\nabla}_{\partial_{x}}\shape^{\star}\right)b^{i}+\shape^{\star}\overline{\nabla}_{\partial_{x}}b^{i}\\
 & =\left(\overline{\nabla}_{\partial_{x}}\shape^{\star}\right)b^{i}\\
 & =\left(\partial_{x}\shape^{\star}\right)b^{i}\\
 & =\left(\left(\shape^{\star}\right)^{2}+\overline{\R}^{N\star}\right)b^{i}.
\end{align*}
Now, by Lemma \ref{lem:computations-in-collar} and using the fact
that $h_{1}=0$, we have $\overline{\R}_{0}^{N}=-h_{2}$; the Poincaré--Einstein
condition then implies $\overline{\R}_{0}^{N}=\PP\left(h_{0}\right)$.
Therefore, evaluating at $x=0$, we obtain
\[
\left(\overline{\nabla}_{\partial_{x}}\left(\shape^{\star}b^{i}\right)\right)_{|x=0}=\PP\left(h_{0}\right)^{\star}\theta^{i}=\rho^{i}.
\]
It follows that $\ddot{b}^{i}\left(0\right)=-\rho^{i}$. Since $b^{i}\left(x\right)\sim b^{i}\left(0\right)+\frac{1}{2}x^{2}\ddot{b}^{i}\left(0\right)+O\left(x^{3}\right)$,
we obtain the claim.
\end{proof}
\begin{cor}
The expansion of $\omega\left(x\right)$ is
\begin{align*}
\omega\left(x\right) & \sim\left(-dx\land\theta+\frac{1}{2}\left[\theta\land\theta\right]\right)x^{-2}\\
 & +\left(\frac{1}{2}dx\land\rho-\frac{1}{2}\left[\theta\land\rho\right]\right)\\
 & +O\left(x\right).
\end{align*}
\end{cor}

\begin{proof}
We have
\[
\omega=x^{-2}\left(-dx\land b+\frac{1}{2}\left[b\land b\right]\right).
\]
We know that
\[
b\sim\theta-\frac{1}{2}\rho x^{2}+O\left(x^{3}\right).
\]
Therefore,
\[
\frac{1}{2}\left[b\land b\right]\sim\frac{1}{2}\left[\theta\land\theta\right]-\frac{1}{2}\left[\theta\land\rho\right]x^{2}+O\left(x^{3}\right).
\]
It follows that
\begin{align*}
\omega\left(x\right) & \sim\left(-dx\land\theta+\frac{1}{2}\left[\theta\land\theta\right]\right)x^{-2}\\
 & +\left(\frac{1}{2}dx\land\rho-\frac{1}{2}\left[\theta\land\rho\right]\right)\\
 & +O\left(x\right)
\end{align*}
as claimed.
\end{proof}
From the expansion above, we obtain the following immediate
\begin{cor}
The coefficient of $x^{-1}$ in the expansion of $d\omega$ is zero.
\end{cor}

We are now ready to compute the equation for the coefficient $\alpha_{1}$
of $x$ in the expansion of $\alpha$. From the equation $d\omega+\left[\alpha\land\omega\right]=0$,
we obtain the equation
\[
\left[\alpha_{-1}\land\omega_{0}\right]+\left[\alpha_{1}\land\omega_{-2}\right]=0.
\]
Substituting $\alpha_{-1}=\theta$ and the expressions for $\omega_{0}$
and $\omega_{2}$, we obtain
\[
\left[\theta\land\left(\frac{1}{2}dx\land\rho-\frac{1}{2}\left[\theta\land\rho\right]\right)\right]+\left[\alpha_{1}\land\left(-dx\land\theta+\frac{1}{2}\left[\theta\land\theta\right]\right)\right].
\]
Isolating the $dx$ term, we obtain the equation
\[
\left[\alpha_{1}\land\theta\right]=\frac{1}{2}\left[\rho\land\theta\right],
\]
from which we obtain $\alpha_{1}=\frac{1}{2}\rho$. Invariantly, this
means that the coefficient $\alpha_{1}$ in the expansion of the geodesic
normal family $\alpha_{h_{0}}\left(x\right)$ of $\nabla^{+}$ is
$\frac{1}{2}\PP\left(h_{0}\right)^{\sharp}$. As a sanity check, observe
that this is compatible with the result of Theorem \ref{thm:coefficients-of-the-instanton-expansion}:
indeed, the trace of $\alpha_{1}$ is determined by $h_{0}$ and the
shape operator of the boundary, and in the Poincaré--Einstein case
it is precisely $\s\left(h_{0}\right)/8$, i.e. trace of $\frac{1}{2}\PP\left(h_{0}\right)$.
To conclude, the boundary term $\beta_{h_{0}}\left(\left[\nabla^{+}\right]\right)$
is precisely the trace-free part of $\alpha_{1}$, i.e. $\frac{1}{2}\mathring{\Ric}\left(h_{0}\right)^{\sharp}$,
as claimed.

\section{The renormalized Yang--Mills energy of a $0$-instanton}

Let us briefly recall some well-known facts about the Yang--Mills
energy of an $\SU\left(2\right)$ connection over a \emph{closed}
4-manifold. Let $\left(X^{4},g\right)$ be a closed, oriented Riemannian
4-manifold, and let $E\to X$ be an $\SU\left(2\right)$ vector bundle.
Recall that $E$ is determined topologically by its \emph{second Chern
class }$c_{2}\left(E\right)\in H^{4}\left(X;\mathbb{Z}\right)=\mathbb{Z}$.
Chern--Weil theory computes $c_{2}\left(E\right)$ in terms of an
arbitrary $\SU\left(2\right)$ connection $A$ as follows:
\begin{equation}
c_{2}\left(E\right)=\frac{1}{8\pi^{2}}\int_{X}\tr\left(F_{A}\land F_{A}\right).\label{eq:second-chern-number}
\end{equation}
Given an $\SU\left(2\right)$ connection $A$, the \emph{Yang--Mills
energy} of $A$ is
\begin{equation}
\mathcal{E}_{\mathrm{YM}}\left(A\right):=\frac{1}{8\pi^{2}}\int_{X}\left|F_{A}\right|^{2}\dVol_{g}.\label{eq:yang-mills-energy}
\end{equation}
If $E$ is trivial (i.e. $c_{2}\left(E\right)=0$), then the absolute
minimizers of $\mathcal{E}_{\mathrm{YM}}\left(A\right)$ are flat
connections. If $E$ is non-trivial, then the absolute minimizers
are either self-dual (if $c_{2}\left(E\right)<0$) or anti-self-dual
(if $c_{2}\left(E\right)>0$) connections: indeed, if $A$ is $\pm$
self-dual, then $\mathcal{E}_{\mathrm{YM}}\left(A\right)=\mp c_{2}\left(E\right)$.
Therefore, \emph{if $A$ is a self-dual connection on an $\SU\left(2\right)$
bundle over a closed oriented conformal 4-manifold, then its Yang--Mills
energy depends only on the topology of $E$}.

In this section, we will discuss the \emph{renormalized }Yang--Mills
energy of self-dual\emph{ $0$-connections }satisfying the Nahm pole
boundary condition. For these connections, the integral (\ref{eq:yang-mills-energy})
is \emph{infinite}, because of the presence of the Nahm pole. However,
if the base manifold $\left(X,g\right)$ is asymptotically Poincaré--Einstein
to sufficiently high order, then this energy can be \emph{renormalized}.
This renormalization procedure consists in choosing a metric $h_{0}\in\mathfrak{c}_{\infty}\left(g\right)$
and considering the \emph{truncated }Yang--Mills energy
\[
\mathcal{E}_{\mathrm{YM}}\left(A,h_{0},t\right):=\frac{1}{8\pi^{2}}\int_{X_{t}}\left|F_{A}\right|^{2}\dVol_{g},
\]
where $X_{t}$ is the portion of $X$ obtained by removing the geodesic
collar $[0,t)\times\partial X$ induced by $h_{0}$, for $t$ small
enough. We shall see that this function of $t$ has a polyhomogeneous
expansion as $t\to0$ of the form
\[
\mathcal{E}_{\mathrm{YM}}\left(A,h_{0},t\right)\sim\mathcal{E}_{-3}t^{-3}+\mathcal{E}_{-1}t^{-1}+\mathcal{E}_{0}+O\left(t\right).
\]

\begin{defn}
The constant term $\mathcal{E}_{0}$ in the expansion above is called
the \emph{renormalized Yang--Mills energy} of $A$. We denote it
by ${^{R}\mathcal{E}_{\mathrm{YM}}\left(A\right)}$.
\end{defn}

Since $\mathcal{E}_{\mathrm{YM}}\left(A,h_{0},t\right)$ depends on
$A$, $h_{0}$ and $t$, the coefficients of the expansion appear
to depend on $h_{0}$ and $A$. However, we shall see that the coefficients
of the divergent terms depend only on $h_{0}$, while ${^{R}\mathcal{E}_{\mathrm{YM}}\left(A\right)}$
is independent of $A$ \emph{and }$h_{0}$. In fact, we shall prove
that ${^{R}\mathcal{E}_{\mathrm{YM}}\left(A\right)}$ coincides with
the negative \emph{Chern--Simons invariant }of the conformal infinity
$\left(\partial X,\mathfrak{c}_{\infty}\left(g\right)\right)$.

Before stating the theorem precisely, let us recall the definition
of the Chern--Simons invariant. Let $Y$ be a closed oriented 3-manifold,
and let $V\to Y$ be an $\SU\left(2\right)$ vector bundle. Recall
that on a 3-manifold, all $\SU\left(2\right)$ vector bundles are
trivial. If $s=\left(s_{1},s_{2}\right)$ is a global frame of $V$,
and $\alpha$ is a $\SU\left(2\right)$ connection on $V$, we define
the \emph{Chern-Simons functional} \emph{relative to $s$} as
\[
\CS_{s}\left(\alpha\right)=\frac{1}{8\pi^{2}}\int_{Y}\tr\left(\alpha\land d\alpha+\frac{2}{3}\alpha\land\alpha\land\alpha\right),
\]
where in the formula above we identify $\alpha$ with the $\mathfrak{su}\left(2\right)$-valued
$1$-form representing $\alpha$ in the frame $s$. If $s'$ is another
frame, then $\CS_{s}\left(\alpha\right)-\CS_{s'}\left(\alpha\right)$
turns out to be an \emph{integer}, namely the degree of the map $Y\to\SU\left(2\right)$
representing the change of frames from $s$ to $s'$. Therefore, we
can drop the dependence on $s$ and think of $\CS\left(\alpha\right)$
as an $\mathbb{R}/\mathbb{Z}$-valued functional on the space of $\SU\left(2\right)$
connections on $Y$. If $Y$ is equipped with a Riemannian metric
$h$ and a spin structure $\left(\mathbb{S},\cl\right)$, and we denote
by $\omega$ the Levi-Civita spin connection on $\mathbb{S}$, then
$\CS\left(\omega\right)\in\mathbb{R}/\mathbb{Z}$ is called the \emph{Chern--Simons
invariant }of $\left(Y,h\right)$. In fact this invariant does not
depend on the choice of the spin structure, and crucially it is a
\emph{conformal invariant }of $\left(Y,\left[h\right]\right)$, which
we denote by $\CS\left(Y,\left[h\right]\right)$.

If $Y$ is the boundary of a compact, oriented 4-manifold $X$, then
the Chern--Simons functional on $Y$ can be promoted to an $\mathbb{R}$-valued
invariant, which however depends on the topology of $X$. Since all
$\SU\left(2\right)$ vector bundles on a compact 4-manifold with boundary
are trivial, we can extend $E$ arbitrarily to an $\SU\left(2\right)$
vector bundle $\boldsymbol{E}$ over $X$, and $\alpha$ to a $\SU\left(2\right)$
connection $A$ over $\boldsymbol{E}$. We can then define
\[
\CS_{X}\left(\alpha\right):=\frac{1}{8\pi^{2}}\int_{X}\tr\left(F_{A}^{2}\right)\in\mathbb{R}.
\]
This definition makes sense because, as the notation suggests, $\CS_{X}\left(\alpha\right)$
equals $\CS\left(\alpha\right)$ modulo $\mathbb{Z}$. Indeed, if
$A$ and $D$ are two $\SU\left(2\right)$ connections on $E$, then
writing $A=D+a$ we have
\[
\tr\left(F_{A}^{2}\right)=\tr\left(F_{D}^{2}\right)+d\tr\left(2a\land F_{D}+a\land d_{D}a+\frac{2}{3}a\land a\land a\right).
\]
Therefore, if we choose a trivialization $s$ of $E$, we extend it
to a trivialization of $\boldsymbol{E}$, and we choose $D$ as the
trivial connection associated to this trivialization, we have by Stokes'
Theorem
\begin{align*}
\frac{1}{8\pi^{2}}\int_{X}\tr\left(F_{A}^{2}\right) & =\CS_{s}\left(\alpha\right).
\end{align*}
The $\mathbb{Z}$ ambiguity in the definition of $\CS\left(\alpha\right)$
is broken by $X$, because the presence of the bulk manifold $X$
singles out a class of homotopic frames of $E$, namely those which
extend to $\boldsymbol{E}$. If $Y$ is equipped with a conformal
class $\left[h\right]$, we define $\CS_{X}\left(Y,\left[h\right]\right)$
by choosing a metric $h\in\left[h\right]$ and a spin structure $\left(\mathbb{S},\cl\right)\to\left(Y,h\right)$,
and then defining $\CS_{X}\left(Y,\left[h\right]\right):=\CS_{X}\left(\omega\right)$
where $\omega$ is the Levi-Civita spin connection on $\mathbb{S}$.

We can now state the theorem precisely:
\begin{thm}
\label{thm:renormalized-YM-energy}Let $\left(X^{4},g\right)$ be
an oriented conformally compact 4-manifold, asymptotically Poincaré--Einstein
to third order. Let $E\to X$ be an $\SU\left(2\right)$ vector bundle,
and let $A$ be a polyhomogeneous self-dual $0$-connection on $E$.
Given $h_{0}\in\mathfrak{c}_{\infty}\left(g\right)$, the asymptotic
expansion of the truncated Yang--Mills energy $\mathcal{E}_{\mathrm{YM}}\left(A,h_{0},t\right)$
is
\begin{align*}
\mathcal{E}_{\mathrm{YM}}\left(A,h_{0},t\right) & \sim\frac{1}{8\pi^{2}}\mathrm{Vol}\left(\partial X,h_{0}\right)t^{-3}\\
 & -\frac{3}{64\pi^{2}}\int_{\partial X}\s\left(h_{0}\right)\dVol_{h_{0}}t^{-1}\\
 & -\CS_{X}\left(\partial X,\mathfrak{c}_{\infty}\left(g\right)\right)\\
 & +O\left(t\right).
\end{align*}
In particular, the renormalized Yang--Mills energy is independent
of $A$ and $h_{0}$, and it equals $-\CS_{X}\left(\partial X,\mathfrak{c}_{\infty}\left(g\right)\right)$. 
\end{thm}

The remaining part of the section is devoted to the proof of this
theorem. Choose $h_{0}\in\mathfrak{c}_{\infty}\left(g\right)$, and
write as usual the geodesic normal form as
\[
g=\frac{dx^{2}+h\left(x\right)}{x^{2}}
\]
in the geodesic collar $[0,\varepsilon)\times\partial X$. Expand
$h\left(x\right)\sim h_{0}+h_{1}x+h_{2}x^{2}+h_{3}x^{3}+o\left(x^{3}\right)$.
Then the assumption that $g$ is asymptotically Poincaré--Einstein
to third order, i.e. that $\Ric\left(g\right)+3g=o\left(x^{3}\right)$
as a symmetric $0$-$2$-tensor, is equivalent to the conditions
\begin{align*}
h_{1} & =0\\
h_{2} & =-\PP\left(h_{0}\right)\\
\tr_{h_{0}}h_{3} & =0.
\end{align*}
The symmetric trace-free part of $h_{3}$ remains arbitrary, and we
do not need to fix it here.

Let $A$ be a polyhomogeneous self-dual $0$-connection satisfying
the Nahm pole boundary condition. The integrand $\left|F_{A}\right|^{2}\dVol_{g}$
is gauge invariant, and therefore by Theorem \ref{thm:obstruction-zero-iff-smooth}
and Corollary \ref{cor:asymptotically-PE-implies-vanishing-obstruction}
we can assume that $A$ is smooth. Following §\ref{subsec:Asymptotic-expansions-of-self-dual-0-connections},
the expansion of $\alpha_{h_{0}}\left(x\right)$ in the geodesic collar
takes the form
\[
\alpha_{h_{0}}\left(x\right)\sim\theta x^{-1}+\omega+\alpha_{1}x+\alpha_{2}x^{2}+o\left(x^{2}\right),
\]
where:
\begin{enumerate}
\item $\omega$ is the Levi-Civita spin connection on $\left(\mathbb{S},\cl_{h_{0}}\right)$;
\item $\sk\alpha_{1}=0$ and $\tr\alpha_{1}=\frac{\s\left(h_{0}\right)}{8}$,
while $\symmtf\alpha_{1}$ is not fixed;
\item $\alpha_{2}$ is completely determined by $h_{0}$ and $\symmtf\alpha_{1}$.
\end{enumerate}
Let $D$ be a connection on $E$ which equals $\omega$ on the collar.
Then, on the collar, we have $A-D=\tilde{\alpha}$, where
\[
\tilde{\alpha}\left(x\right)\sim\theta x^{-1}+\alpha_{1}x+\alpha_{2}x^{2}+o\left(x^{2}\right).
\]
By Stokes' theorem, we have
\[
\int_{X_{t}}\tr\left(F_{A}^{2}\right)=\int_{X_{t}}\tr\left(F_{D}^{2}\right)+\int_{\partial X_{t}}\tr\left(2\tilde{\alpha}\land f_{\omega}+\tilde{\alpha}\land d_{\omega}\tilde{\alpha}+\frac{2}{3}\tilde{\alpha}\land\tilde{\alpha}\land\tilde{\alpha}\right).
\]
Since $D_{|\partial X}=\omega$, we have by definition
\[
\lim_{t\to0}\frac{1}{8\pi^{2}}\int_{X_{t}}\tr\left(F_{D}^{2}\right)=\CS_{X}\left(\partial X,\mathfrak{c}_{\infty}\left(g\right)\right).
\]
Therefore, it remains to expand the boundary integral
\begin{equation}
\int_{\partial X_{t}}\tr\left(2\tilde{\alpha}\land f_{\omega}+\tilde{\alpha}\land d_{\omega}\tilde{\alpha}+\frac{2}{3}\tilde{\alpha}\land\tilde{\alpha}\land\tilde{\alpha}\right)\label{eq:boundary_integral}
\end{equation}
in powers of $t$. For ease of notation, we denote by $\s_{0}$ and
$\dVol_{0}$ the scalar curvature and volume forms of $\left(\partial X,h_{0}\right)$. 

\subsubsection*{Coefficient of $t^{-3}$}

The coefficient of $t^{-3}$ in the expansion of (\ref{eq:boundary_integral})
is
\[
\int_{\partial X}\tr\left(\frac{2}{3}\theta\land\theta\land\theta\right).
\]
Since $\theta$ is the soldering form, an elementary computation yields
\[
\tr\left(\frac{2}{3}\theta\land\theta\land\theta\right)=-\dVol_{0}.
\]
Therefore, the coefficient of $t^{-3}$ in the expansion of $\mathcal{E}_{\mathrm{YM}}\left(A,h_{0},t\right)$
is
\[
\frac{1}{8\pi^{2}}\mathrm{Vol}\left(h_{0}\right).
\]

\subsubsection*{Coefficient of $t^{-2}$}

The only term that could contribute to the coefficient of $t^{-2}$
in the expansion of (\ref{eq:boundary_integral}) would be the integral
of $\tr\left(\theta\land d_{\omega}\theta\right)$. However, since
$\omega$ is the Levi-Civita spin connection, $d_{\omega}\theta=0$.
Therefore, the coefficient of $t^{-2}$ vanishes.

\subsubsection*{Coefficient of $t^{-1}$}

Using the cyclic property of the trace, the coefficient of $t^{-1}$
in the expansion of (\ref{eq:boundary_integral}) is
\[
2\int_{\partial X}\tr\left(\theta\land f_{\omega}+\theta\land\theta\land\alpha_{1}\right).
\]
Since $\star_{0}f_{\omega}=\G_{0}^{\sharp}$, it is easy to compute
\[
2\tr\left(\theta\land f_{\omega}\right)=\frac{\s_{0}}{2}\dVol_{0}.
\]
Similarly, since $\tr\alpha_{1}=\frac{\s_{0}}{8}$, we obtain
\[
2\tr\left(\theta\land\theta\land\alpha_{1}\right)=-\frac{\s_{0}}{8}\dVol_{0}.
\]
It follows that the coefficient of $t^{-1}$ in the expansion of $\mathcal{E}_{\mathrm{YM}}\left(A,h_{0},t\right)$
is
\[
-\frac{3}{64\pi^{2}}\int_{\partial X}\s_{0}\dVol_{0}.
\]

\subsubsection*{Constant coefficient}

The constant coefficient in the expansion of (\ref{eq:boundary_integral})
is
\[
\int_{\partial X}\tr\left(\theta\land d_{\omega}\alpha_{1}+2\theta\land\theta\land\alpha_{2}\right).
\]
We want to prove that this integral vanishes. The $2$-form $d_{\omega}\alpha_{1}$
is the curl of $\alpha_{1}$, and $\tr\left(\theta\land d_{\omega}\alpha_{1}\right)$
is a constant multiple of its trace, multiplied by $\dVol_{0}$. The
Poincaré--Einstein condition implies that $\alpha_{1}$ is symmetric,
and the curl of a symmetric $2$-tensor is always trace-free; it follows
that $\tr\left(\theta\land d_{\omega}\alpha_{1}\right)=0$. Concerning
the second term, we have
\[
2\theta\land\theta\land\alpha_{2}=-\tr\alpha_{2}\dVol_{h_{0}}.
\]
We need to show that $\tr\alpha_{2}=0$. This trace can be computed
using the self-duality equation, as we did in the proof of Theorem
\ref{thm:coefficients-of-the-instanton-expansion}. Specifically,
we need to isolate the coefficient of $x$ in the expansion of $\partial_{x}\alpha=-\star\left(f_{\alpha_{0}}+d_{\alpha_{0}}\tilde{\alpha}+\frac{1}{2}\left[\tilde{\alpha}\land\tilde{\alpha}\right]\right)$.
From $h_{1}=0$, we have $\star_{1}=0$ by Lemma \ref{lem:computations-in-collar},
and moreover the asymptotic Poincaré--Einstein condition implies
that $\alpha_{0}=\omega$ and therefore $d_{\omega}\theta=0$. We
then obtain the equation
\[
2\alpha_{2}=-\star_{0}d_{\omega}\alpha_{1}-\star_{0}\left[\theta\land\alpha_{2}\right]-\star_{3}\frac{1}{2}\left[\theta\land\theta\right].
\]
By Lemma \ref{lem:useful}, we have $\star_{0}\left[\theta\land\alpha_{2}\right]=\left(\tr\alpha_{2}\right)1_{TY}-\alpha_{2}^{*}$.
Using again the properties of the soldering form, we obtain
\[
\alpha_{2}+2\sk\alpha_{2}=-\star_{0}d_{\omega}\alpha_{1}-\left(\tr\alpha_{2}\right)1_{TY}-\star_{3}\star_{0}\theta.
\]
Taking the trace, we obtain
\[
\tr\alpha_{2}=-\frac{1}{4}\tr\left(\star_{3}\star_{0}\theta\right).
\]
In order to compute this last trace, we remember from Lemma \ref{lem:computations-in-collar}
that
\[
\partial_{x}^{2}\star=\star\left(6\left(\shape^{\star}\right)^{2}-4\shape^{\star}\HH+\HH^{2}+2\overline{\R}^{N\star}-\tr\shape^{2}-\overline{\Ric}^{N}\right).
\]
We have $\star_{3}=\frac{1}{6}\left(\partial_{x}^{3}\star\right)_{|x=0}$,
so using the fact that $\shape_{0}=0$ and $\HH_{0}=0$, we obtain
\[
\star_{3}\star_{0}=-\star_{0}\star_{3}=-\frac{1}{6}\partial_{x}\left(2\overline{\R}^{N\star}-\overline{\Ric}^{N}\right)_{|x=0}.
\]
Taking the trace, and using again Lemma \ref{lem:computations-in-collar},
we obtain
\[
\tr\left(\star_{3}\star_{0}\theta\right)=\frac{1}{6}\left(\partial_{x}^{2}\HH\right)_{|x=0}=-\frac{1}{2}\tr_{h_{0}}h_{3}.
\]
By the asymptotic Poincaré--Einstein condition, this last trace vanishes,
and this completes the proof.

\appendix
\stepcounter{section}

\section*{Appendix}

\subsection{\label{subsec:Polyhomogeneity}Polyhomogeneity and log-smoothness}

In this subsection, we briefly recall the definitions and basic properties
of polyhomogeneous functions over a compact manifold with boundary
$X$. For a more complete analysis of this topic, we refer to §4 of
\cite{MelroseCorners}.

Roughly speaking, a polyhomogeneous functions on $X$ is a smooth
function on $X^{\circ}$ which admits a Taylor-like expansion near
the boundary, where the terms of this expansion are allowed to be
of the form $x^{\alpha}\left(\log x\right)^{l}$ with $\left(\alpha,l\right)$
ranging in a discrete subset of $\mathbb{C}\times\mathbb{N}$. Let
us be more precise.
\begin{defn}
We denote by $\dot{C}^{\infty}\left(X;\mathcal{D}_{X}^{1}\right)$
the Fréchet space of smooth $1$-densities on $X$ vanishing to infinite
order along $\partial X$, equipped with the topology of uniform convergence
of all derivatives. We denote by $C^{-\infty}\left(X\right)$ the
weak dual of $\dot{C}^{\infty}\left(X;\mathcal{D}_{X}^{1}\right)$.
Its elements are called \emph{extendible distributions} on $X$.
\end{defn}

The Fréchet space $C^{-\infty}\left(X\right)$ is a large ambient
space of distributions on $X$, and every Fréchet and Banach space
of functions considered in this paper embeds continuously into $C^{-\infty}\left(X\right)$.
This, among other things, allows us to take weak directional derivatives
of functions in $L^{p}$ spaces.
\begin{defn}
Let $\delta\in\mathbb{R}$, and let $x$ be a boundary defining function
for $X$. A \emph{$O\left(x^{\delta}\right)$ conormal }function on
$X$ is a function $u\in x^{\delta}L^{\infty}\left(X\right)$ such
that, for every finite set $V_{1},...,V_{k}$ of vector fields on
$X$ tangent to $\partial X$, the weak derivative $V_{1}\cdots V_{k}u$
is again in $x^{\delta}L^{\infty}\left(X\right)$. We denote by $\mathcal{A}^{\delta}\left(X\right)$
the space of $O\left(x^{\delta}\right)$ conormal functions on $X$.
\end{defn}

\begin{rem}
Conormal functions on $X$ are always smooth in $X^{\circ}$, and
they have stable regularity and decay under an arbitrary number of
derivatives in directions tangent to $\partial X$. However, they
can lose regularity or decay when differentiated by vector fields
transverse to $\partial X$. In particular, conormal functions do
not necessarily have a Taylor expansion on $\partial X$. For example,
if $x$ is a boundary defining function and $f\in C^{\infty}\left(X\right)$,
then the function $x^{\cos\left(f\right)}$ is $O\left(x^{-1-\varepsilon}\right)$
conormal for every $\varepsilon>0$, but does not generally have a
Taylor expansion.
\end{rem}

\begin{defn}
An \emph{index set }is a subset $\mathcal{E}\subseteq\mathbb{C}\times\mathbb{N}$
with the following properties:
\begin{enumerate}
\item if $\left(\alpha,l\right)\in\mathcal{E}$, then $\left(\alpha+k,l'\right)\in\mathcal{E}$
for every $k\in\mathbb{N}$ and $l'\leq l$;
\item for every $\delta\in\mathbb{R}$, the set $\mathcal{E}_{\delta}=\left\{ \left(\alpha,l\right)\in\mathcal{E}:\Re\left(\alpha\right)\leq\delta\right\} $
is finite.
\end{enumerate}
\end{defn}

Given an index set $\mathcal{E}$ and a real number $\delta$, we
say that $\Re\left(\mathcal{E}\right)>\delta$ if $\Re\left(\alpha\right)>\delta$
for every $\left(\alpha,l\right)\in\mathcal{E}$.
\begin{defn}
A \emph{polyhomogeneous function on $X$ with index set $\mathcal{E}$}
is a conormal function $u$ on $X$ such that, for some (hence every)
boundary defining function $x$ for $X$, there exist functions $u_{\alpha,l}\in C^{\infty}\left(X\right)$
such that
\[
u-\sum_{\begin{smallmatrix}\left(\alpha,l\right)\in\mathcal{E}\\
\Re\left(\alpha\right)\leq\delta
\end{smallmatrix}}u_{\alpha,l}x^{\alpha}\left(\log x\right)^{l}\in\mathcal{A}^{\delta}\left(X\right)
\]
for every $\delta\in\mathbb{R}$.
\end{defn}

\begin{rem}
We express the condition in the second point above as
\[
u\sim\sum_{\left(\alpha,l\right)\in\mathcal{E}}u_{\alpha,l}x^{\alpha}\left(\log x\right)^{l}.
\]
Sometimes, we will use a similar formula when the coefficients $u_{\alpha,l}$
are only in $C^{\infty}\left(\partial X\right)$. When we do that,
we mean that we are tacitly extending $u_{\alpha,l}$ to the interior;
specifically, we choose an auxiliary vector field $V$ on $X$ such
that $Vx\equiv1$ near $\partial X$, and we use the flow of $V$
starting at $\partial X$ to construct a collar neighborhood $[0,\varepsilon)\times\partial X\hookrightarrow X$
for $\partial X$ in $X$; we can then extend a function $C^{\infty}\left(\partial X\right)$
to a smooth function on $X$ by first extending it trivially to the
collar, and then multiplying it by a cutoff function supported in
the collar and equal to $1$ near $\partial X$.
\end{rem}

\begin{rem}
Let $\mathcal{E}$ be an index set of the form $\mathcal{E}=\mathbb{N}\cup\mathcal{I}$,
with $\Re\left(\mathcal{I}\right)>k\in\mathbb{N}$. If $u\in\mathcal{A}_{\phg}^{\mathcal{E}}\left(X\right)$,
then $u\in C^{k}\left(X\right)$, and the first $k+1$ Taylor coefficients
of $u$ are all smooth. Indeed, by definition, we can write $u=u_{0}+\tilde{u}$
with $u_{0}$ smooth and $\tilde{u}$ polyhomogeneous with index set
$\mathcal{I}$. In particular, $\tilde{u}=O\left(x^{\delta}\right)$
for some $\delta>k$.
\end{rem}

\begin{defn}
\label{def:log-smoothness}A polyhomogeneous function $u\in\mathcal{A}_{\phg}\left(X\right)$
is said to be \emph{log-smooth of order $k$} if $u\in\mathcal{A}_{\phg}^{\mathcal{E}}\left(X\right)$
with $\mathcal{E}$ generated by the pairs $\left(lk,l\right)$, with
$l,k\in\mathbb{N}$.
\end{defn}

\begin{rem}
The previous definition is not standard in the literature; however,
in the author's experience, log-smooth functions, or sections of bundles,
appear often as solutions of semi-linear $0$-elliptic problems (cf.
Remark \ref{rem:regularity-PE}).
\end{rem}

\subsection{\label{subsec:Geometry-of-spin-3D}Geometry of 3-manifolds}

In this and the next subsections, we fix our conventions on the Riemannian
and spin geometry of 3-manifolds, and we perform the elementary computations
needed in §\ref{sec:Asymptotic-expansion-of-self-dual-0-connections}.

Let $\left(Y^{3},h\right)$ be a closed oriented Riemannian 3-manifold.
Denote by
\begin{align*}
\flat & :TY\to\Lambda^{1}\\
\sharp & :\Lambda^{1}\to TY
\end{align*}
the musical isomorphisms, and call $\star$ the Hodge star operator.
We have a canonical isomorphism $TY\simeq\mathfrak{so}\left(TY\right)$,
induced by the isomorphism
\begin{align*}
\mathbb{R}^{3} & \to\mathfrak{so}\left(3\right)\\
e_{i} & \mapsto\mathfrak{a}_{i}
\end{align*}
where
\[
\mathfrak{a}_{1}=\left(\begin{matrix}0 & 0 & 0\\
0 & 0 & -1\\
0 & 1 & 0
\end{matrix}\right),\mathfrak{a}_{2}=\left(\begin{matrix}0 & 0 & 1\\
0 & 0 & 0\\
-1 & 0 & 0
\end{matrix}\right),\mathfrak{a}_{3}=\left(\begin{matrix}0 & -1 & 0\\
1 & 0 & 0\\
0 & 0 & 0
\end{matrix}\right);
\]
here the metric on $\mathfrak{so}\left(3\right)$ is
\[
\left\langle \mathfrak{a},\mathfrak{b}\right\rangle =-\frac{1}{2}\tr\left(\mathfrak{a}\mathfrak{b}\right).
\]
We call $\Ric\left(h\right)$, $\mathring{\Ric}\left(h\right)$, $\G\left(h\right)$
and $\s\left(h\right)$ the Ricci tensor, trace-free Ricci tensor,
Einstein tensor, and scalar curvature, respectively. Using the identification
$\Lambda^{1}\otimes\mathfrak{so}\left(TY\right)\simeq\Lambda^{1}\otimes\Lambda^{1}$,
we can decompose canonically and orthogonally
\[
\Lambda^{1}\otimes\mathfrak{so}\left(TY\right)=S^{2}\oplus\Lambda^{2},
\]
where $S^{2}$ is the bundle of symmetric $2$-tensors. Under the
identification $\Lambda^{1}\otimes\Lambda^{1}=\mathfrak{gl}\left(TY\right)$,
sections of $S^{2}$ (resp. $\Lambda^{2}$) correspond to self-adjoint
(resp. skew-adjoint) endomorphisms of $TY$. Using the endomorphism
notation, we define
\begin{align*}
\symm\gamma & =\frac{\gamma+\gamma^{*}}{2}\\
\sk\gamma & =\left(\frac{\gamma-\gamma^{*}}{2}\right).
\end{align*}
The bundle $S^{2}$ further decomposes orthogonally as $S_{0}^{2}\oplus\Lambda^{0}$,
where $S_{0}^{2}$ is the bundle of symmetric trace-free $2$-tensors
and $\Lambda^{0}$ is the trivial real line bundle. If $\gamma$ is
already self-adjoint, we define the trace-free part as
\[
\tf\gamma=\gamma-\frac{\tr\gamma}{3}\id
\]
and we define
\[
\symmtf\gamma=\tf\left(\symm\gamma\right).
\]
Let $\nabla$ be the Levi-Civita connection on $TY$. Its curvature
$F_{\nabla}$ is a section of $\Lambda^{2}\otimes\mathfrak{so}\left(TY\right)$.
Identifying $\mathfrak{so}\left(TY\right)$ with $\Lambda^{1}$, $\star F_{\nabla}$
can be interpreted as a section of $\Lambda^{1}\otimes\Lambda^{1}$.
Under this identification, $\star F_{\nabla}$ is the Einstein tensor
$\G$.

Let's now describe our convention on spin geometry.
\begin{defn}
A \emph{spin structure} on $\left(Y,h\right)$ is the datum of:
\begin{enumerate}
\item an $\SU\left(2\right)$ vector bundle $\mathbb{S}\to Y$, called the
\emph{spinor bundle};
\item a bundle map $\cl:T^{*}Y\to\mathfrak{su}\left(\mathbb{S}\right)$,
called \emph{Clifford multiplication}, such that $\frac{1}{2}\cl$
is an isomorphism of Lie algebra bundles, i.e.
\[
\cl\left(\star\left(\alpha\land\beta\right)\right)=\frac{1}{2}\left[\cl\alpha,\cl\beta\right]
\]
for every pair of covectors $\alpha,\beta$ in the same fiber of $T^{*}Y$.
\end{enumerate}
Two spin structures $\left(\mathbb{S},\cl\right)$ and $\left(\mathbb{S}',\cl'\right)$
are \emph{equivalent }if there exists an isomorphism of $\SU\left(2\right)$
bundles $\phi:\mathbb{S}\to\mathbb{S}'$ which is compatible with
the Clifford multiplications, i.e. for every $p\in Y$ and $\alpha\in T_{p}^{*}Y$
we have
\[
\phi_{p}\circ\cl_{p}\left(\alpha\right)s=\cl_{p}^{*}\left(\alpha\right)\circ\phi_{p}.
\]

\end{defn}

\begin{rem}
With the convention adopted here, we have $\cl\left(\dVol_{h}\right)=-1$.
\end{rem}


Fix a spin structure $\left(\mathbb{S},\cl\right)$ on $\left(Y,h\right)$.
Then we have a canonical morphism of principal bundles $\Fr_{\SU\left(2\right)}\left(\mathbb{S}\right)\to\Fr_{\SO\left(3\right)}\left(TY\right)$,
which restricts on every fiber to the universal covering map. This
implies that an $\SU\left(2\right)$ frame $s_{1},s_{2}$ of $\mathbb{S}$
determines a unique $\SO\left(3\right)$ coframe $\theta^{1},\theta^{2},\theta^{3}$
of $T^{*}Y$, i.e. the coframe such that, in terms of $s_{1},s_{2}$,
we have $\cl\left(\theta^{i}\right)=\sigma_{i}$ where $\sigma_{1},\sigma_{2},\sigma_{3}$
is the basis of $\mathfrak{su}\left(2\right)$ given by
\[
\sigma_{1}=\left(\begin{matrix}i\\
 & -i
\end{matrix}\right),\sigma_{2}=\left(\begin{matrix} & 1\\
-1
\end{matrix}\right),\sigma_{3}=\left(\begin{matrix} & i\\
i
\end{matrix}\right).
\]
We also get a canonical isomorphism $\mathfrak{so}\left(TY\right)\simeq\mathfrak{su}\left(\mathbb{S}\right)$:
using the $\SU\left(2\right)$ frame $s_{1},s_{2}$ of $\mathbb{S}$,
the associated $\SO\left(3\right)$ coframe $\theta^{1},\theta^{2},\theta^{3}$
of $T^{*}Y$, and its dual frame $\theta_{1},\theta_{2},\theta_{3}$,
the isomorphism $\mathfrak{su}\left(\mathbb{S}\right)\to\mathfrak{so}\left(TY\right)$
corresponds to the derivative at the identity of the covering map
$\SU\left(2\right)\to\SO\left(3\right)$, i.e. the map
\begin{align*}
\mathfrak{su}\left(2\right) & \to\mathfrak{so}\left(3\right)\\
\frac{\sigma_{i}}{2} & \mapsto\mathfrak{a}_{i}.
\end{align*}

\begin{defn}
The \emph{soldering form} is the $\mathfrak{su}\left(\mathbb{S}\right)$-valued
$1$-form
\[
\theta:=\frac{1}{2}\cl.
\]
\end{defn}

\begin{rem}
In terms of an $\SU\left(2\right)$ frame $s_{1},s_{2}$ of $\mathbb{S}$
and the associated $\SO\left(3\right)$ coframe $\theta^{1},\theta^{2},\theta^{3}$
of $T^{*}Y$, we have
\[
\theta=\theta^{i}\otimes\frac{\sigma_{i}}{2}.
\]
\end{rem}

The soldering form $\theta$ determines an identification of $\SO\left(3\right)$
vector bundles $TY\to\mathfrak{su}\left(\mathbb{S}\right)$. Therefore,
we can canonically identify an $\mathfrak{su}\left(\mathbb{S}\right)$-valued
$1$-form $\gamma$ with the endomorphism $\theta^{-1}\circ\gamma$
of $TY$. Using an $\SU\left(2\right)$ frame $s_{1},s_{2}$ of $\mathbb{S}$
and the associated coframe $\theta^{1},\theta^{2},\theta^{3}$ of
$T^{*}Y$, if
\[
\gamma=\gamma_{j}^{i}\theta^{j}\otimes\frac{\sigma_{i}}{2},
\]
then the associated endomorphism $\theta^{-1}\circ\gamma$ is
\[
\theta^{-1}\circ\gamma=\gamma_{j}^{i}\theta_{i}\otimes\theta^{j}.
\]
Under this identification, we have $\theta=\id$.
\begin{lem}
\label{lem:useful}Let $\gamma$ be an $\mathfrak{su}\left(\mathbb{S}\right)$-valued
$1$-form on $Y$. Call $\omega$ the Levi-Civita spin connection
on $\mathbb{S}$. Then
\begin{align*}
\star\left[\theta\land\gamma\right] & =\star\left[\gamma\land\theta\right]=\left(\tr\gamma\right)1_{TY}-\gamma^{*}\\
\star\frac{1}{2}\left[\gamma\land\gamma\right] & =\left(\gamma^{*}\right)^{2}-\left(\tr\gamma\right)\gamma^{*}+\frac{1}{2}\left[\left(\tr\gamma\right)^{2}-\tr\left(\gamma^{2}\right)\right]1_{TY}.
\end{align*}
In these equations, on the left hand side $\gamma$ is interpreted
as an $\mathfrak{su}\left(\mathbb{S}\right)$-valued $1$-form, while
on the right hand side it is interpreted as an endomorphism of $TY$.
\end{lem}

\begin{proof}
Choose an $\SU\left(2\right)$ frame $s_{1},s_{2}$ of $\mathbb{S}$
and let $\theta^{1},\theta^{2},\theta^{3}$ be the corresponding $\SO\left(3\right)$
coframe.\\
(1) We have
\begin{align*}
\star\left[\theta\land\gamma\right]^{1} & =\star\left(\theta^{2}\land\gamma^{3}-\theta^{3}\land\gamma^{2}\right)=-\theta^{1}\gamma_{1}^{1}-\theta^{2}\gamma_{1}^{2}-\theta^{3}\gamma_{1}^{3}+\theta^{1}\tr\gamma\\
\star\left[\theta\land\gamma\right]^{2} & =\star\left(\theta^{3}\land\gamma^{1}-\theta^{1}\land\gamma^{3}\right)=-\theta^{1}\gamma_{2}^{1}-\theta^{2}\gamma_{2}^{2}-\theta^{3}\gamma_{2}^{3}+\theta^{2}\tr\gamma\\
\star\left[\theta\land\gamma\right]^{3} & =\star\left(\theta^{1}\land\gamma^{2}-\theta^{2}\land\gamma^{1}\right)=-\theta^{1}\gamma_{3}^{1}-\theta^{2}\gamma_{3}^{2}-\theta^{3}\gamma_{3}^{3}+\theta^{3}\tr\gamma
\end{align*}
from which we obtain
\[
\star\left[\theta\land\gamma\right]=\theta\tr\gamma-\gamma^{*}.
\]
Since the bracket on $\mathfrak{su}\left(\mathbb{S}\right)$ and the
wedge product of $1$-forms are both skew-symmetric operations, we
have $\star\left[\theta\land\gamma\right]=\star\left[\gamma\land\theta\right]$.\\
(2) We have
\[
\frac{1}{2}\star\left[\gamma\land\gamma\right]^{1}=\frac{1}{2}\star\left(\gamma^{2}\land\gamma^{3}-\gamma^{3}\land\gamma^{2}\right).
\]
Therefore,
\begin{align*}
\frac{1}{2}\star\left[\gamma\land\gamma\right]_{1}^{1} & =\gamma_{2}^{2}\gamma_{3}^{3}-\gamma_{2}^{3}\gamma_{3}^{2}\\
\frac{1}{2}\star\left[\gamma\land\gamma\right]_{2}^{1} & =\gamma_{3}^{2}\gamma_{1}^{3}-\gamma_{3}^{3}\gamma_{1}^{2}\\
\frac{1}{2}\star\left[\gamma\land\gamma\right]_{3}^{1} & =\gamma_{1}^{2}\gamma_{2}^{3}-\gamma_{1}^{3}\gamma_{2}^{2}.
\end{align*}
Call $C\left(\gamma\right)$ the endomorphism whose matrix form with
respect to the frame $\theta_{1},\theta_{2},\theta_{3}$ is the cofactor
matrix of the matrix of $\gamma$. Then a direct comparison shows
that
\[
\frac{1}{2}\star\left[\gamma\land\gamma\right]_{i}^{1}=C\left(\gamma\right)_{i}^{1}.
\]
Arguing similarly for the second and third columns, we obtain
\[
\star\left[\gamma\land\gamma\right]=C\left(\gamma\right).
\]
By the Cayley--Hamilton Theorem, we have
\[
C\left(\gamma\right)=\left(\gamma^{*}\right)^{2}-\left(\tr\gamma\right)\gamma^{*}+\frac{1}{2}\left[\left(\tr\gamma\right)^{2}-\tr\left(\gamma^{2}\right)\right]\id.
\]
\end{proof}
Another useful property of $\theta$ concerns the \emph{torsion} of
an $\SU\left(2\right)$ connection on $\mathbb{S}$. Due to the canonical
isomorphism $\mathfrak{su}\left(\mathbb{S}\right)\simeq\mathfrak{so}\left(TY\right)$,
there is a canonical 1-1 correspondence between $\SU\left(2\right)$
connections on $\mathbb{S}$ and $\SO\left(3\right)$ connections
on $TY$. Therefore, given an $\SU\left(2\right)$ connection $\alpha$
on $\mathbb{S}$, we can define its torsion by interpreting $\alpha$
as a connection on $TY$. Using the canonical identification $\Lambda^{2}\otimes TY\simeq\Lambda^{2}\otimes\mathfrak{su}\left(\mathbb{S}\right)$,
the torsion is $\mathfrak{su}\left(\mathbb{S}\right)$-valued $1$-form.
A direct computation shows that the torsion is precisely $d_{\alpha}\theta$,
the covariant derivative of the soldering form. Therefore, $d_{\alpha}\theta$
completely determines $\alpha$. More precisely:
\begin{lem}
\label{lem:connection-induced-by-torsion}Let $\alpha$ be an $\SU\left(2\right)$
connection on $\mathbb{S}$, determined by the torsion equation $d_{\alpha}\theta=\star T$
for some $\mathfrak{su}\left(\mathbb{S}\right)$-valued $1$-form
$T$. Let $\omega$ be the Levi-Civita spin connection on $\left(\mathbb{S},\cl\right)\to\left(Y,h\right)$.
Then
\begin{align*}
\alpha & =\omega+\sk T-\symmtf T+\frac{1}{6}\tr T\theta,
\end{align*}
where $\sk:\Lambda^{1}\otimes\Lambda^{1}\to\Lambda^{2}$ and $\symmtf:\Lambda^{1}\otimes\Lambda^{1}\to S_{0}^{2}$
are the canonical orthogonal projections, and we use the identification
$\Lambda^{1}\otimes\Lambda^{1}\simeq\Lambda^{1}\otimes\mathfrak{su}\left(\mathbb{S}\right)$
to interpret $\sk T$ and $\symmtf T$ as $\mathfrak{su}\left(\mathbb{S}\right)$-valued
$1$-forms.
\end{lem}

\begin{proof}
Write $\alpha=\omega+\gamma$ for some $\mathfrak{su}\left(\mathbb{S}\right)$-valued
$1$-form $\gamma$. Decompose
\[
\gamma=\sk\gamma+\symmtf\gamma+\frac{1}{3}\tr\gamma\theta.
\]
Since $\omega$ is torsion-free, the equation $d_{\alpha}\theta=\star T$
simplifies to $\left[\gamma\land\theta\right]=\star T$. Applying
$\star$ and using Lemma \ref{lem:useful}, we then obtain
\[
\theta\tr\gamma-\gamma^{*}=T.
\]
From this, we obtain
\begin{align*}
\sk\gamma & =\sk T\\
\symmtf\gamma & =-\symmtf T\\
\tr\gamma & =\frac{1}{2}\tr T.
\end{align*}
\end{proof}

\subsection{\label{subsec:Geometry-of-4-manifolds-in-a-geodesic-collar}Geometry
of 4-manifolds in a geodesic collar}

We now consider a manifold $X^{4}=\left[0,1\right]\times Y^{3}$,
where $Y^{3}$ is a closed oriented 3-manifold, and we consider a
metric on $X$ of the form
\[
\overline{g}=dx^{2}+h\left(x\right),
\]
where $h\left(x\right)$ is a smooth family $\left[0,1\right]\to C^{\infty}\left(Y;S^{2}\left(T^{*}Y\right)\right)$
of metrics on $Y$. Denote by $\sharp\left(x\right),\flat\left(x\right)$
the musical isomorphisms of $h\left(x\right)$, and by $\star\left(x\right)$
the Hodge star of $h\left(x\right)$. We call $\partial_{x}$ the
$\overline{g}$-gradient of $x$. We orient $X^{4}$ in such a way
that, if $e_{1},e_{2},e_{3}$ is a positively oriented frame of $TY$,
then $-\partial_{x},e_{1},e_{2},e_{3}$ is a positively oriented frame
of $TX$.

In what follows, all the ``overlined'' symbols refer to curvature
tensors / geometric operators on $X$ associated to the metric $\overline{g}$,
while the ``non-overlined'' symbols refer to \emph{families} of
curvature tensors / geometric operators on $Y$ associated to the
family of metrics $h\left(x\right)$. In particular, we denote by
$\overline{\nabla}$, $\overline{\R}$, $\overline{\Ric}$, $\overline{\s}$
the Levi-Civita connection, Riemann curvature tensor, Ricci curvature,
and scalar curvature of $\left(X,\overline{g}\right)$. Then the symbols
$\nabla$, $\R$, $\Ric$, $\s$ refer to \emph{families} of the analogous
objects dependent on $x\in\left[0,1\right]$ on $\left(Y,h\left(x\right)\right)$.
Our conventions for the curvature tensors are
\begin{align*}
\overline{\R}\left(V_{1},V_{2}\right)V_{3} & =\overline{\nabla}_{V_{1}}\overline{\nabla}_{V_{2}}V_{3}-\overline{\nabla}_{V_{2}}\overline{\nabla}_{V_{1}}V_{3}-\overline{\nabla}_{\left[V_{1},V_{2}\right]}V_{3}\\
\overline{\Ric}\left(V_{1},V_{2}\right) & =\tr\left(W\mapsto\overline{\R}\left(W,V_{1}\right)V_{2}\right)\\
\overline{\s} & =\tr_{\overline{g}}\overline{\Ric}.
\end{align*}
Whenever we have a family $\omega=\omega\left(x\right)$ of symmetric
$2$-tensors on $Y$, we denote by $\omega^{\sharp}$ the family of
endomorphisms of $TY$ obtained from $\omega$ by raising an index
using $h$:
\[
\left\langle \omega^{\sharp}V,W\right\rangle _{h}=\omega\left(V,W\right).
\]
Whenever we have an endomorphism $\phi:TY\to TY$, we denote by $\phi^{\star}:\Lambda^{1}\to\Lambda^{1}$
the dual map.

In addition to the intrinsic objects $\nabla$, $\R$, $\Ric$, $\s$
in $\left(Y,h\right)$, we define the following extrinsic invariants:
\begin{enumerate}
\item the \emph{shape operator}, i.e. the family of endomorphisms $\shape=\shape\left(x\right):TY\to TY$
given by $\shape V:=-\overline{\nabla}_{V}\partial_{x}$;
\item the \emph{mean curvature}, i.e. the family of traces of the shape
operator: $\HH=\HH\left(x\right):=\tr\shape\left(x\right)$;
\item the\emph{ normal curvature operator}, i.e. the family of endomorphisms
$\overline{\R}^{N}=\overline{\R}^{N}\left(x\right):TY\to TY$ given
by $\overline{\R}^{N}V:=\overline{\R}\left(V,\partial_{x}\right)\partial_{x}$;
\item the \emph{normal Ricci curvature}, i.e. the family of $h$-traces
of the normal curvature operator: $\overline{\Ric}^{N}:=\tr_{h}\overline{\R}^{N}=\overline{\Ric}\left(\partial_{x},\partial_{x}\right)$.
\end{enumerate}
For every $x$, the operators $\shape\left(x\right)$ and $\overline{\R}^{N}\left(x\right)$
are both self-adjoint with respect to $h\left(x\right)$.

The next lemma collects various useful computations discussing the
first and second variations of the geometric operators associated
to $\left(Y,h\right)$.
\begin{lem}
\label{lem:computations-in-collar}Let $V_{1},V_{2}$ be vector fields
on $Y$ constant in $x$, and let $\alpha_{1},\alpha_{2}$ be $1$-forms
on $Y$ constant in $x$. Denote by $\star=\star\left(x\right)$ the
Hodge star $\Lambda^{1}\to\Lambda^{2}$ of $h\left(x\right)$.
\begin{enumerate}
\item $\partial_{x}\left\langle V_{1},V_{2}\right\rangle _{h}=-2\left\langle \shape V_{1},V_{2}\right\rangle _{h}$;
\item if $\alpha\left(x\right)$ is a smooth family of $1$-forms on $Y$,
then $\dot{\alpha}=\overline{\nabla}_{\partial_{x}}\alpha-\shape^{\star}\alpha$;
\item $\partial_{x}\left\langle \alpha_{1},\alpha_{2}\right\rangle _{h}=2\left\langle \shape^{\star}\alpha_{1},\alpha_{2}\right\rangle _{h}$;
\item $\partial_{x}\shape=\shape^{2}+\overline{\R}^{N}$;
\item $\partial_{x}^{2}\left\langle V_{1},V_{2}\right\rangle _{h}=2\left\langle \left(\shape^{2}-\overline{\R}^{N}\right)V_{1},V_{2}\right\rangle _{h}$;
\item $\partial_{x}\dVol_{h}=-\HH\dVol_{h}$;
\item $\partial_{x}\star=\star\left(2\shape^{\star}-\HH\right)$;
\item $\partial_{x}\shape^{\star}=\left(\shape^{\star}\right)^{2}+\overline{\R}^{N\star}$;
\item $\partial_{x}\HH=\tr\shape^{2}+\overline{\Ric}^{N}$;
\item $\partial_{x}^{2}\star=\star\left(6\left(\shape^{\star}\right)^{2}-4\shape^{\star}\HH+\HH^{2}+2\overline{\R}^{N\star}-\tr\shape^{2}-\overline{\Ric}^{N}\right)$.
\end{enumerate}
\end{lem}

\begin{proof}
(1) We have
\begin{align*}
\partial_{x}\left\langle V_{1},V_{2}\right\rangle _{h} & =\partial_{x}\left\langle V_{1},V_{2}\right\rangle _{\overline{g}}\\
 & =\left\langle \overline{\nabla}_{\partial_{x}}V_{1},V_{2}\right\rangle _{\overline{g}}+\left\langle V_{1},\overline{\nabla}_{\partial_{x}}V_{2}\right\rangle _{\overline{g}}\\
 & =\left\langle \overline{\nabla}_{V_{1}}\partial_{x},V_{2}\right\rangle _{\overline{g}}+\left\langle V_{1},\overline{\nabla}_{V_{2}}\partial_{x}\right\rangle _{\overline{g}}\\
 & =-\left\langle \shape V_{1},V_{2}\right\rangle _{\overline{g}}-\left\langle V_{1},\shape V_{2}\right\rangle _{\overline{g}}\\
 & =-2\left\langle \shape V_{1},V_{2}\right\rangle _{h}.
\end{align*}
(2) If $V$ is a vector field on $Y$ (constant in $x$), then
\begin{align*}
\dot{\alpha}\left(V\right) & =\partial_{x}\left(\alpha\left(V\right)\right)\\
 & =\left(\overline{\nabla}_{\partial_{x}}\alpha\right)\left(V\right)+\alpha\left(\overline{\nabla}_{\partial_{x}}V\right)\\
 & =\left(\overline{\nabla}_{\partial_{x}}\alpha\right)\left(V\right)+\left\langle \overline{\nabla}_{\partial_{x}}V,\alpha^{\sharp}\right\rangle _{h}\\
 & =\left(\overline{\nabla}_{\partial_{x}}\alpha\right)\left(V\right)-\left\langle \shape V,\alpha^{\sharp}\right\rangle _{h}\\
 & =\left(\overline{\nabla}_{\partial_{x}}\alpha\right)\left(V\right)-\left\langle V,\shape\alpha^{\sharp}\right\rangle _{h}\\
 & =\left(\overline{\nabla}_{\partial_{x}}\alpha\right)\left(V\right)-\left(\shape^{\star}\alpha^{\sharp}\right)\left(V\right).
\end{align*}
From the generality of $V$, we have
\[
\dot{\alpha}=\overline{\nabla}_{\partial_{x}}\alpha-\shape^{\star}\alpha.
\]
(3) Since $\alpha_{1}$ is independent of $x$, from the previous
point we have $\overline{\nabla}_{\partial_{x}}\alpha_{1}=\shape^{\star}\alpha_{1}$.
Therefore,
\begin{align*}
\partial_{x}\left\langle \alpha_{1},\alpha_{2}\right\rangle _{h} & =\left\langle \overline{\nabla}_{\partial_{x}}\alpha_{1},\alpha_{2}\right\rangle _{\overline{g}}+\left\langle \alpha_{1},\overline{\nabla}_{\partial_{x}}\alpha_{2}\right\rangle _{\overline{g}}\\
 & =\left\langle \shape^{\star}\alpha_{1},\alpha_{2}\right\rangle _{\overline{g}}+\left\langle \alpha_{1},\shape^{\star}\alpha_{2}\right\rangle _{\overline{g}}\\
 & =2\left\langle \shape^{\star}\alpha_{1},\alpha_{2}\right\rangle _{h}.
\end{align*}
(4) This equation is known as the Riccati Equation, cf. Corollary
3.3 of \cite{GrayTubes}. Alternatively, if $V$ is a vector field
on $Y$, we have
\begin{align*}
\left(\partial_{x}\shape\right)\left(V\right) & =\left[\partial_{x},\shape V\right]-\shape\left(\left[\partial_{x},V\right]\right)\\
 & =\overline{\nabla}_{\partial_{x}}\shape V-\overline{\nabla}_{\shape V}\partial_{x}\\
 & =-\overline{\nabla}_{\partial_{x}}\overline{\nabla}_{V}\partial_{x}+\shape^{2}V\\
 & =\shape^{2}V+\overline{\R}^{N}V.
\end{align*}
(5) We have
\begin{align*}
\partial_{x}^{2}\left\langle V_{1},V_{2}\right\rangle _{h} & =\partial_{x}\left(-2\left\langle \shape V_{1},V_{2}\right\rangle _{h}\right)\\
 & =-2\left\langle \overline{\nabla}_{\partial_{x}}\left(\shape V_{1}\right),V_{2}\right\rangle _{\overline{g}}-2\left\langle \shape V_{1},\overline{\nabla}_{\partial_{x}}V_{2}\right\rangle _{\overline{g}}\\
 & =-2\left\langle \overline{\nabla}_{\shape V_{1}}\partial_{x}+\left[\partial_{x},\shape V_{1}\right],V_{2}\right\rangle _{\overline{g}}+2\left\langle \shape V_{1},\shape V_{2}\right\rangle _{h}\\
 & =-2\left\langle -\shape^{2}V_{1}+\left(\partial_{x}\shape\right)V_{1},V_{2}\right\rangle _{h}+2\left\langle \shape^{2}V_{1},V_{2}\right\rangle _{h}\\
 & =2\left\langle 2\shape^{2}V_{1}-\left(\shape^{2}+\overline{\R}^{N}\right)V_{1},V_{2}\right\rangle _{h}\\
 & =2\left\langle \left(\shape^{2}-\overline{\R}^{N}\right)V_{1},V_{2}\right\rangle _{h}.
\end{align*}
(6) Theorem 3.11 of \cite{GrayTubes}.\\
(7) We have
\begin{align*}
\alpha_{1}\land\left(\partial_{x}\star\right)\alpha_{2} & =\partial_{x}\left(\alpha_{1}\land\star\alpha_{2}\right)\\
 & =\partial_{x}\left[\left\langle \alpha_{1},\alpha_{2}\right\rangle _{h}\dVol_{h}\right]\\
 & =2\left\langle \shape^{\star}\alpha_{1},\alpha_{2}\right\rangle _{h}\dVol_{h}-\left\langle \alpha_{1},\alpha_{2}\right\rangle _{h}\HH\dVol_{h}\\
 & =2\left\langle \alpha_{1},\shape^{\star}\alpha_{2}\right\rangle _{h}\dVol_{h}-\left\langle \alpha_{1},\alpha_{2}\right\rangle _{h}\HH\dVol_{h}\\
 & =\alpha_{1}\land\star\left(2\shape^{\star}-\HH\right)\alpha_{2}.
\end{align*}
From the generality of $\alpha_{1},\alpha_{2}$, we obtain $\partial_{x}\star=\star\left(2\shape^{\star}-\HH\right)$.\\
(8) Let $\alpha$ be a $1$-form on $Y$, and let $V$ be a vector
field on $Y$. We have
\begin{align*}
\partial_{x}\left[\left(\shape^{\star}\alpha\right)\left(V\right)\right] & =\partial_{x}\left[\alpha\left(\shape V\right)\right]\\
 & =\left(\mathcal{L}_{\partial_{x}}\alpha\right)\left(\shape V\right)+\alpha\left(\left(\partial_{x}\shape\right)V+\shape\left[\partial_{x},V\right]\right)\\
 & =\alpha\left(\left(\shape^{2}+\overline{\R}^{N}\right)V\right)\\
 & =\left[\left(\left(\shape^{\star}\right)^{2}+\overline{\R}^{N\star}\right)\alpha\right]\left(V\right).
\end{align*}
On the other hand,
\begin{align*}
\partial_{x}\left[\left(\shape^{\star}\alpha\right)\left(V\right)\right] & =\left(\left(\partial_{x}\shape^{\star}\right)\alpha\right)\left(V\right)+\shape^{\star}\left(\mathcal{L}_{\partial_{x}}\alpha\right)\left(V\right)+\left(\shape^{\star}\alpha\right)\left(\left[\partial_{x},V\right]\right)\\
 & =\left(\left(\partial_{x}\shape^{\star}\right)\alpha\right)\left(V\right).
\end{align*}
In these computations, we used the fact that, since $\alpha$ and
$V$ do not depend on $x$, we have $\mathcal{L}_{\partial_{x}}\alpha=0$
and $\left[\partial_{x},V\right]=0$. From the generality of $\alpha$
and $V$, it follows that $\partial_{x}\shape^{\star}=\left(\shape^{\star}\right)^{2}+\overline{\R}^{N\star}$.\\
(9) We have
\begin{align*}
\partial_{x}\HH & =\partial_{x}\tr\shape\\
 & =\tr\partial_{x}\shape\\
 & =\tr\left(\shape^{2}+\overline{\R}^{N}\right)\\
 & =\tr\left(\shape^{2}\right)+\overline{\Ric}^{N}.
\end{align*}
(10) From $\partial_{x}\star=\star\left(2\shape^{\star}-\HH\right)$,
we have
\begin{align*}
\partial_{x}^{2}\star & =\left(\partial_{x}\star\right)\left(2\shape^{\star}-\HH\right)+\star\left(2\partial_{x}\shape^{\star}-\partial_{x}\HH\right)\\
 & =\star\left(2\shape^{\star}-\HH\right)^{2}+\star\left(2\left(\shape^{\star}\right)^{2}+2\overline{\R}^{N\star}-\tr\shape^{2}-\overline{\Ric}^{N}\right)\\
 & =\star\left(6\left(\shape^{\star}\right)^{2}-4\shape^{\star}\HH+\HH^{2}+2\overline{\R}^{N\star}-\tr\shape^{2}-\overline{\Ric}^{N}\right).
\end{align*}
\end{proof}
We now consider an algebraic decomposition of the Weyl tensor $\overline{\W}$
of $\left(X,\overline{g}\right)$ along the fibers of $x:X\to\left[0,1\right]$.
We think of $\overline{\W}$ as a tensor of type (1,3). Moreover,
we denote by $*\overline{\W}$ the tensor of type (1,3) whose index
expression in terms of an $\SO\left(4\right)$ frame is
\[
\left(*\overline{\W}\right)_{abc}{^{d}}=\frac{1}{2}\varepsilon_{ab}{^{\alpha\beta}}\overline{\W}_{\alpha\beta c}{^{d}}.
\]
As usual, we define
\begin{align*}
\overline{\W}^{+} & =\frac{\overline{\W}+*\overline{\W}}{2}\\
\overline{\W}^{-} & =\frac{\overline{\W}-*\overline{\W}}{2}.
\end{align*}

\begin{defn}
$ $
\begin{enumerate}
\item The \emph{electric Weyl endomorphism }is the family $\overline{\W}^{E}=\overline{\W}^{E}\left(x\right)$
of endomorphisms of $TY$ defined by
\[
\overline{\W}^{E}V=\overline{\W}\left(V,\partial_{x}\right)\partial_{x}.
\]
\item The \emph{magnetic Weyl endomorphism }is the family $\overline{\W}^{B}=\overline{\W}^{B}\left(x\right)$
of endomorphisms of $TY$ defined by
\[
\overline{\W}^{B}=\left(*\overline{\W}\right)^{E}.
\]
\end{enumerate}
\end{defn}

The two components $\overline{\W}^{E}$ and $\overline{\W}^{B}$ can
be related to the intrinsic and extrinsic invariants described above,
via the Gauss and Codazzi equations. In order to formulate these identities,
we recall that on any closed oriented Riemannian 3-manifold $\left(Z,\eta\right)$,
there is a natural formally self-adjoint first order differential
operator acting on sections of the bundle of symmetric $2$-tensors,
called the \emph{symmetric} \emph{curl}. It is easier to describe
it as an operator acting on self-adjoint endomorphisms of $TZ$, using
the soldering form $\theta$, interpreted here as a canonical section
of $T^{*}Z\otimes\mathfrak{so}\left(TZ\right)$. If $\mathcal{S}:TY\to TY$
is self-adjoint, then
\[
\left(\symmtf\curl\mathcal{S}\right)^{\sharp}=\symmtf\star d_{\nabla}\left(\mathcal{S}^{\star}\theta\right)
\]
where $\mathcal{S}^{\star}:T^{*}Z\to T^{*}Z$ is the dual map. Here
$\star d_{\nabla}\left(\mathcal{S}^{\star}\theta\right)$ is again
a section of $T^{*}Z\otimes\mathfrak{so}\left(TZ\right)\equiv T^{*}Z\otimes TZ$,
so it can be decomposed into its self-adjoint trace-free, skew-adjoint
and pure trace parts.

We can now come back to our formulas for $\overline{\W}^{E}$ and
$\overline{\W}^{B}$:
\begin{prop}
\label{prop:electric-magnetic-Weyl}We have
\begin{align*}
\overline{\W}^{E} & =\frac{1}{2}\tf\left(\shape^{2}-\overline{\R}^{N}+\Ric^{\sharp}-\HH\shape\right)\\
\overline{\W}^{B} & =\symmtf\curl\shape.
\end{align*}
\end{prop}

\begin{proof}
These are the coordinate invariant expressions of the two equations
in page 49 of \cite{FeffermanGrahamAmbient}. We remark that, in our
convention, in abstract index notation we have $\overline{\W}^{E}=\overline{\W}_{a00}{^{b}}$.
Therefore, compared to the equations of Fefferman--Graham, the first
one has the opposite sign (because of the symmetries of $\overline{\W}$)
while the second one has the same sign (there are two sign flips,
one due to the different definition of $\overline{\W}^{E}$, and the
other due to the different orientation convention which flips the
sign of $*$).
\end{proof}
\bibliography{allpapers}{}
\bibliographystyle{alpha}
\newpage
\end{document}